\documentclass[11pt, a4paper]{amsart}

\usepackage[foot]{amsaddr}  

\usepackage[margin=2.54cm]{geometry}
\usepackage{fancyhdr}

\DeclareMathOperator{\Tam}{Tam}

\usepackage{amsmath, amssymb, amsthm}
\usepackage{thmtools}  
\usepackage{array}
\usepackage{latexsym}
\usepackage{nicematrix}
\usepackage{bigstrut}
\usepackage{enumerate}
\usepackage{paralist}

\usepackage{mathrsfs}

\usepackage{float}   
\usepackage{subcaption}  
\usepackage{pgf,tikz}
\usetikzlibrary{decorations.pathreplacing,calligraphy}
\usepackage{graphicx}
\usepackage[all]{xy}
\usetikzlibrary{arrows}
\usetikzlibrary{arrows.meta}

\usepackage{fdsymbol}
\usepackage{todonotes}
\setuptodonotes{size=\tiny}

\makeatletter
\define@key{todonotes}{AS}[]{%
	\setkeys{todonotes}{color=blue!50}}%

\makeatother

\newtheorem{theorem}{Theorem}[section]
\newtheorem{corollary}[theorem]{Corollary}
\newtheorem{proposition}[theorem]{Proposition}
\newtheorem{lemma}[theorem]{Lemma}

\newtheorem{question}[theorem]{Question}

\theoremstyle{definition}
\newtheorem{definition}[theorem]{Definition}
\newtheorem{remark}[theorem]{Remark}
\newtheorem{example}[theorem]{Example}

\usepackage{url}
\usepackage[colorlinks=true,citecolor=cyan,backref=page]{hyperref}
\usepackage[capitalise]{cleveref}

\title{$(P,\phi)$-Tamari and higher torsion lattices of type $\mathbf{A}$}

\author[A.~Segovia]{Adrien~Segovia}
\address{LACIM, Universit\'e du Qu\'ebec \`a Montr\'eal, Canada}
\email{segovia.adrien@courrier.uqam.ca}

\begin{document}

\begin{abstract}
The goal of this work is to study the combinatorics of the lattices of higher torsion classes of the higher Auslander and Nakayama algebras of type \textbf{A}. Combinatorial descriptions of these higher torsion classes were recently obtained by August \textit{et al.} (2025), and it was observed that the lattices that they form are not semidistributive. We study in some depth these lattices, proving in particular that the lattices of higher torsion classes of the higher Auslander algebras of type \textbf{A} are join-semidistributive, join-extremal and left modular. We also prove that the lattices of higher torsion classes of the higher Nakayama algebras of type \textbf{A} are lattice quotients of them. In order to prove these results, we define a general construction that produces a lattice, which we call $(P,\phi)$-Tamari, for any choice of poset $P$ and chain $\phi$ in $P$. We prove the lattice results for this general construction. When $P=\phi$ is a chain, we recover the Tamari lattice, whereas the lattices of higher torsion classes that we study are obtained for another very particular choice.    
\end{abstract}

\maketitle

\section{Introduction}

The Tamari lattice is familiar to many in combinatorics. It is perhaps most simply described as the lattice of planar binary trees ordered by tree rotation. For our purposes, it is also important to point out that it arises as the lattice of torsion classes for the type $\textbf{A}_n$ linearly oriented path algebra. Our goal is to present a certain combinatorial generalization of the Tamari lattice which we call the $(P,\phi)$-Tamari lattice, where $P$ is a poset and $\phi$ is a chain in $P$. We establish certain properties for these lattices, including join-semidistributivity, join-extremality, and left modularity.

The reason that we were inspired to formulate the definition of the $(P,\phi)$-Tamari lattices is that they include, as a very special case, certain lattices which recently appeared in representation theory. For any finite-dimensional algebra, its torsion classes ordered by inclusion form a lattice. These have attracted considerable interest \cite{demonet2023lattice,thomas2021introduction}. One property in particular of interest is that they are all semidistributive. There is a generalization due to J\o{}rgensen of torsion classes \cite{jorgensen2016torsion} in the setting of higher homological algebra \cite{IYAMA1,IYAMA2,IYAMA3}, known as $d$-torsion classes (or higher torsion classes). 
 Recently the authors of \cite{August_2025} showed that J\o{}rgensen's definition is equivalent to being closed under $d$-extensions and $d$-quotients, thus yielding a lattice ordered by inclusion on these $d$-torsion classes. They obtained a combinatorial description of the $d$-torsion classes for the higher Auslander and Nakayama algebras of type $\textbf{A}$, which are the two main examples where we are able to compute in higher homological algebra. 
 They noticed that the lattices of $d$-torsion classes of these algebras are not semidistributive in general, but we prove that they are a special case of our construction. Thus they inherit the properties of the $(P,\phi)$-Tamari lattices.

In \cref{sec:background} we give background on the relevant lattice theory (\cref{sec:posetsandlattices}) and introduce the lattices of higher torsion classes (\cref{sec:backgroundhighertorsionclasses}). In \cref{sectionPphiTamari} we define and study the $(P,\phi)$-Tamari lattices; the definition is given in \cref{sectionGenerality}, among other results the join-semidistributivity and join-extremality are proved in \cref{sec:finiteTam}, we briefly talk about enumeration in \cref{sec:enumgeneral} and results related to labellings, topology and congruences are proved in \cref{sectionCongruenceTopo}. In \cref{sectionMainExamples} we look at particular examples of our construction; the case of a poset $P$ that is a chain is treated in \cref{sec:Chain}, whereas the particular cases of the lattices of higher torsion classes of the higher Auslander and Nakayama algebras are respectively studied in \cref{sec:Auslander} and \cref{sec:Nakayama}. We finish by giving some open questions in \cref{sec:openquestionspphi}.

This paper is the full-length version of an extended abstract that was accepted to FPSAC 2025 \cite{segovia2025pphitamarilattices}.

\section*{Acknowledgements}

I am very grateful to Hugh Thomas for proposing this project and for his support. I also want to thank my other advisor Samuele Giraudo, and Baptiste Rognerud for interesting discussions. I was supported by NSERC Discovery Grants RGPIN-2022-03960 and RGPIN-2024-04465.

\section{Background}
\label{sec:background}

\subsection{Posets and lattices}
\label{sec:posetsandlattices}

We denote by $|E|$ the cardinality of a set $E$. For a positive integer $n$, denote $[n]:=\{1,2,\dots,n\}$. By $k<n$ we will mean $k\in \{0,1,\dots,n-1\}$.

If a partially ordered set (abbreviated as \textit{poset}) $(P,\leq)$ has a minimum, it is denoted $\hat{0}$, and $\hat{1}$ for the maximum. If a subset $X\subseteq P$ has a minimum element, it is denoted by $\min(X)$. If it has a maximum element, it is denoted by $\max(X)$. The dual of $P$ is denoted $P^{op}$. Two elements $x,y\in P$ are \emph{comparable} if $x\leq y$ or $y\leq x$. If this is not the case, then $x$ and $y$ are \emph{incomparable} and we denote it by $x\not\sim y$. A subset $X\subseteq P$ is an \emph{antichain} if the elements of $X$ are pairwise incomparable. The cover relations of $P$ are denoted $x\lessdot y$ and we say that $y$ \emph{covers} $x$ or $x$ \emph{is covered by} $y$. We denote by $\mathrm{Cov}_P^{\uparrow}(x)$ the set of elements that cover $x$, and by $\mathrm{Cov}_P^{\downarrow}(x)$ the set of elements that are covered by $x$. They form the set $E(P)$ of the edges of its Hasse diagram, which is drawn with smaller elements at the bottom. If $P$ has a minimum $\hat{0}$, the \emph{atoms} are the elements of $P$ that cover $\hat{0}$. If $P$ has a maximum $\hat{1}$, the \emph{coatoms} are the elements of $P$ that are covered by $\hat{1}$. The \emph{degree} of an element $x\in P$ is $\deg(x)=|\mathrm{Cov}_P^{\uparrow}(x)| + |\mathrm{Cov}_P^{\downarrow}(x)|$. The poset $P$ is \emph{Hasse-regular} if all its elements have the same degree.
The interval of $P$ between $x$ and $y$ is denoted by $[x,y]$. By a \emph{subposet} we always mean an induced subposet. Let $X\subseteq P$. If $P\setminus X$ is viewed as a poset, it corresponds to the subposet of $P$ on the elements $P\setminus X$. The subposet of $P$ on the elements $[x,y]\setminus \{y\}$ is denoted by $[x,y[$. Its Möbius function is written $\mu(x,y)$.
A subset $C$ of $P$ is \emph{convex} if for all $x$ and $y$ in $C$, we have $[x,y]\subseteq C$.   An \emph{order ideal} of a poset $P$ is a subset $I$ such that for all $x,y\in P$, if $y\leq x$ and $x\in I$, then $y\in I$. Dually, an \emph{order filter} is a subset $F$ such that for all $x,y\in P$, if $y\geq x$ and $x\in F$, then $y\in F$. The order ideal generated by a subset $C$ is denoted by $I_P(C):=\{y\in P\mid \exists x\in C,\,y\leq x\}$. If $C=\{x\}$, $I_P(\{x\})$ is a \emph{principal order ideal} and is denoted by $I_P(x)$. Dually the \emph{principal order filter} generated by $x\in P$ is denoted $F_P(x)$.
The poset ordered by inclusion on the order ideals of a poset $P$ is denoted $J(P)$. Then the number of order ideals of $P$, which is also its number of order filters, is $|J(P)|$. 
We denote by \emph{$C_n$} the chain $0<1<\cdots <n-1$. We also call \emph{chains} the totally ordered subsets of $P$. Equivalently, they are the image of an injective order preserving map $\phi : \{0,1,\dots,n\}\rightarrow P$, and we will use $\phi: \phi(0)< \phi(1) <\cdots <\phi(n)$ to denote them. In this case, the length of this chain, denoted by $|\phi|$, is $n$. We will also briefly consider the case where $\phi$ is infinite, meaning $\phi$ is an injective order preserving map from the nonnegative integers to $P$. If we cannot add elements to a chain, it is called a \emph{maximal chain}. The \emph{length} of a poset, denoted $\ell(P)$, is the maximum length of a chain in $P$. The chains of length $\ell(P)$ are called \emph{longest chains}. A finite poset is \emph{graded} if all its maximal chains have the same length. The \emph{spine} of a poset is the subset of the elements that lie on some longest chain. A \emph{linear extension} of a poset $(P,\leq)$ is a total order $\preceq$ on $P$ such that $x\leq y$ implies $x\preceq y$. The \emph{disjoint union} of two posets $(P,\leq_P)$ and $(Q,\leq_Q)$, denoted by $P\bigsqcup Q$, is the poset on $P\sqcup Q$ defined by $x\leq y$ if $x\leq_P y$ and $x,y\in P$, or $x\leq_Q y$ and $x,y\in Q$.
If $P$ and $Q$ are two posets, then the \emph{direct product} $P\times Q$ is the poset on the cartesian product $P\times Q$ defined by $(x,y)\leq (x',y')$ if $x\leq x'$ and $y\leq y'$. We will refer to this latter partial order as the \emph{product order}, sometimes denoted by $\leq_{prod}$.

A \emph{lattice} $L$ is a poset such that any pair of elements $\{x,y\}$ admits a least upper bound, called the \emph{join} and written $x\vee y$, and a greatest lower bound, called the \emph{meet} and written $x\wedge y$. If only the join exists for any pair of elements, then $L$ is a \emph{join-semilattice}. A \emph{complete lattice} is a poset in which any subset of the elements admits a least upper bound and greatest lower bound. Any finite lattice is a complete lattice but the converse is not true in general. We will denote $\Tam_n$ the Tamari lattice of size $n$ which has cardinality $\frac{1}{n+1} \binom{2n}{n}$.

The following definitions of join-irreducibles and meet-irreducibles are given for a finite lattice. A \emph{join-irreducible} $j\in L$ is an element that covers a unique element, denoted $j_*$, and a \emph{meet-irreducible} $m\in L$ is one that is covered by a unique element.
The sets of these elements are respectively denoted by $\mathrm{JIrr}(L)$ and $\mathrm{MIrr}(L)$. An \emph{edge-labelling} of $L$ is a map $\gamma: E(L) \rightarrow P$ where $P$ is a poset.
It is well-known that in a finite lattice, an element is always the join of the join-irreducibles below it. It follows: 

\begin{lemma}
For any finite lattice $L$, we have $\ell(L)\leq \mathrm{min}\big(|\mathrm{JIrr}(L)|,|\mathrm{MIrr}(L)|\big)$.
\end{lemma}

\begin{proof}
We prove $\ell(L)\leq |\mathrm{JIrr}(L)|$ as for the meet-irreducibles it is dual. Let $\phi: \phi(0)\lessdot \cdots \lessdot \phi(n)$ be a maximal chain of length $n$. For all $i<n$, there is always a join-irreducible that is below $\phi(i+1)$ but not $\phi(i)$. If it was not the case, then $\phi(i+1)$, which is the join of the join-irreducibles below it, would be below $\phi(i)$, which is absurd. This proves that we have at least $n$ join-irreducibles. 
\end{proof}

\begin{definition}[{\cite{MarkowskyExtremal}}]
    A lattice $L$ is \emph{join-extremal} if $\ell(L)=|\mathrm{JIrr}(L)|$, \emph{meet-extremal} if $\ell(L)=|\mathrm{MIrr}(L)|$ and \emph{extremal} if it is both join-extremal and meet-extremal. 
\end{definition}

\begin{definition}
A lattice $L$ is \emph{complemented} if any element $x\in L$ has a \emph{complement} $y\in L$, which is an element such that $x\vee y=\hat{1}$ and $x\wedge y=\hat{0}$. 
A lattice $L$ is \emph{meet-pseudocomplemented} if for any element $x\in L$, the set $\{y\in L\mid x\wedge y=\hat{0}\}$ has a maximum element. The \emph{join-pseudocomplemented} lattices are defined dually, and a lattice is \emph{pseudocomplemented} if it is both meet-pseudocomplemented and join-pseudocomplemented.    
\end{definition}

\begin{definition}
A lattice $L$ is \emph{join-semidistributive} if for all $x,y,z\in L$, $x\vee y=x\vee z$ implies $x\vee (y\wedge z)=x\vee y$. \emph{Meet-semidistributive} lattices are defined dually, and it is \emph{semidistributive} if it is both join-semidistributive and meet-semidistributive.
\end{definition}

The following two results can be found in \cite[Section 5]{freese1995free}.

\begin{lemma}
\label{lemJSDjoinmeetSD}
A lattice $L$ is semidistributive if and only if $L$ is join-semidistributive and $|\mathrm{JIrr}(L)|=|\mathrm{MIrr}(L)|$ .
\end{lemma}

\begin{lemma}
\label{lemEquiJSDpphi}
A lattice $L$ is join-semidistributive if and only if for all covers $b\lessdot c$, the set $I_L(c)\setminus I_L(b)$ has a minimum element. In this case, the minimum is join-irreducible, and $\gamma: E(L) \rightarrow \mathrm{JIrr}(L)$ that sends a cover $b\lessdot c$ to $\mathrm{min}(I_L(c)\setminus I_L(b))$ is a well-defined edge-labelling.
\end{lemma}

\begin{definition}
     An element $a\in L$ is \emph{left modular} if for all $b<c$ in $L$, we have $(b\vee a) \wedge c = b\vee (a\wedge c)$. 
    A maximal chain made of left modular elements is called a maximal left modular chain. The lattice $L$ is left modular if there exists a maximal left modular chain. 
\end{definition}

We now recall a result that enables us to prove that a lattice is left modular using some edge-labellings. The names for the labellings in the following definition are chosen to be consistent with \cite{segovia2025extremalitysemidistributivelattices}.

\begin{definition}
	\label{deflabellingpphi}
	Let $L$ be a lattice. Denote by $\psi :\, \hat{0}=x_0 <\dots <x_k=\hat{1}$ a chain containing $\hat{0}$ and $\hat{1}$. Using $\psi$, we define two edge-labellings $\gamma_{1,\psi}$ and $\gamma_{2',\psi}$. For $j\in \mathrm{JIrr}(L)$, denote $\delta_{\psi}(j):=\mathrm{min}\{i\,|\,j\leq x_i\}$. Then, for a cover relation $b\lessdot c$, we define
	\begin{align*}
		\gamma_{1,\psi} (b\lessdot c) &:= \mathrm{min}\{ \delta(j) \,|\, \text{$j$ join-irreducible},\,j\leq c,\,j\not\leq b\}, \\
		\gamma_{2',\psi} (b\lessdot c) &:= \mathrm{min}\{i \,|\, b \vee x_i \geq c\} .
	\end{align*}
\end{definition}

\begin{proposition}[{\cite[Theorem 3.3]{segovia2025extremalitysemidistributivelattices}}]
\label{prop:leftmodularwithlabellings}
Using the notations from \cref{deflabellingpphi}, for any lattice $L$, we have $\gamma_{1,\psi} = \gamma_{2',\psi}$ if and only if for all $i$, $x_i$ is left modular.    
\end{proposition}

We refer the reader to the papers of A. Björner and M. Wachs \cite{bjorner1983lexicographically,bjorner1996shellable,bjorner1997shellable} for details related to the following topological notions. Denote $\overline{L}:=L\setminus \{\hat{0},\hat{1}\}$. The \emph{order complex} $\Delta(P)$ of a poset $P$ is the simplicial complex of vertex set $P$ whose faces are the chains of $P$.  An \emph{$EL$-labelling} of a lattice $L$ is an edge-labelling such that in any interval, when reading the labels following the maximal chains from bottom to top, there is a unique maximal increasing chain and the label word of the increasing chain lexicographically precedes the label word of any other maximal chain. If $L$ admits an $EL$-labelling, then its order complex $\Delta(\overline{L})$ is shellable and homotopy equivalent to a wedge of spheres.
Moreover, for all $x<y$ in $L$ we have that $\mu(x,y)$ is given by the difference between the number of even length maximal decreasing chains and the number of odd length maximal decreasing chains of $[x,y]$. Thus if $L$ has at most one maximal decreasing chain in any interval, then $\mu (x,y)\in \{-1,0,1\}$ for all $x,y\in L$. Also, the order complex of each nonempty open interval $]x,y[$ has the homotopy type of either a sphere or a point.

\begin{proposition}[{\cite{LiuLeftmodular}}]
\label{prop:ELlabellingLiu}
If $\psi$ is a left modular chain, then $\gamma_{1,\psi}=\gamma_{2',\psi}$ and this is an EL-labelling. Thus a left modular lattice admits an $EL$-labelling.
\end{proposition}

\begin{definition}[{\cite{Day_1979}}]
\label{defDoublementDaypphi}
    Let $C$ be a convex subset of $L$. 
    The \emph{doubling} $L[C]$ is the subposet of $L\times C_2$ on the subset
    $\big(I_L(C) \times \{0\}\big) \,\bigsqcup \,\left[\big((L\setminus I_L(C))\cup C\big) \times \{1\} \right]$. It is in fact a lattice.
\end{definition}

 A subset $C$ is a \emph{lower pseudo-interval} if $C$ is a union of intervals sharing the same minimum element. 
 A lattice $L$ is \emph{congruence normal} if it is obtained from the one-element lattice by successive doublings of convex subsets. If at each step we double a lower pseudo-interval then $L$ is \emph{join-congruence uniform} (sometimes called lower-bounded in the literature).
 If we use only doublings of intervals, $L$ is called \emph{congruence uniform}. See \cref{fig:doublings} for an example of the construction of a join-congruence uniform lattice.

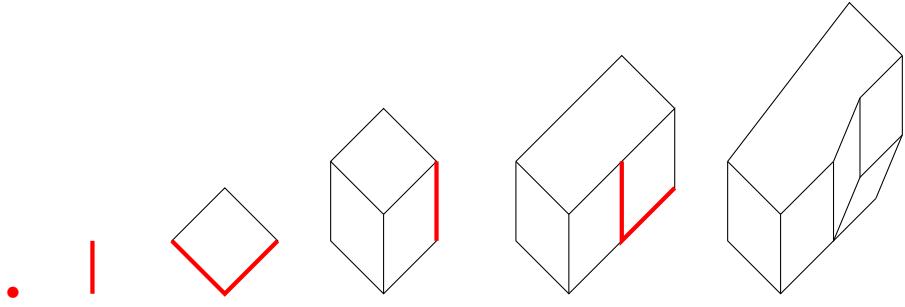
\begin{figure}
	\centering
	\begin{tikzpicture}
    \begin{scope}[scale=0.7]
    
        \begin{scope}[xshift=-4cm]
        \draw[ultra thick, red] (0,0) node{$\bullet$};    
        \end{scope}
        
        \begin{scope}[xshift=-2.5cm]
        \draw[ultra thick, red] (0,0)--(0,1);    
        \end{scope}
        
			\draw (0,0)--(1,1)--(0,2)--(-1,1)--(0,0);
			\draw[ultra thick, red] (-1,1)--(0,0)--(1,1);
			
			\begin{scope}[xshift=3cm, yshift= 1.5cm]
				\draw (0,0)--(1,1)--(0,2)--(-1,1)--(0,0);  
				\draw (-1,1)--(-1,-0.5)--(0,-1.5)--(1,-0.5)--(1,1);
				\draw[ultra thick, red] (1,-0.5)--(1,1);
				\draw (0,-1.5)--(0,0);
			\end{scope}
			
			\begin{scope}[xshift=6.5cm, yshift= 1.5cm]
				\draw (-1,1)--(0,0)--(1,1);  
				\draw (0,0)--(0,-1.5);
				\draw (-1,1)--(-1,-0.5)--(0,-1.5)--(1,-0.5)--(1,1)--(2,2)--(2,0.5)--(1,-0.5);   
				\draw (-1,1)--(1,3)--(2,2);
				\draw[ultra thick,red] (1,1)--(1,-0.5)--(2,0.5);
			\end{scope}
			
			\begin{scope}[xshift=10.5cm, yshift= 1.5cm]
				\draw (-1,1)--(0,0)--(1,1); 
				\draw (0,0)--(0,-1.5);
				\draw (-1,1)--(-1,-0.5)--(0,-1.5)--(1,-0.5);
				\draw (1,1)--(1,-0.5)--(1.8,0.3);
				\draw (1.5,2.2)--(1.5,0.7)--(2.3,1.5);
				\draw (1,1)--(1.5,2.2);  
				\draw (1,-0.5)--(1.5,0.7);  
				\draw (1.8,0.3)--(2.3,1.5);  
				\draw (1.5,2.2)--(2.3,3);
				\draw (2.3,1.5)--(2.3,3);
				\draw (-1,1)--(1.3,4)--(2.3,3);  
			\end{scope}
		\end{scope}
	\end{tikzpicture}
	\caption{We represent five successive doublings of a lattice by lower pseudo-intervals, which are represented with thick red edges.}
	\label{fig:doublings}
\end{figure}

A (lattice) \emph{congruence} on $L$ is an equivalence relation $\equiv$ on $L$ such that for all $ x_1,x_2,y_1,y_2$ in $L$, we have that $x_1\equiv x_2$ and $y_1\equiv y_2$ imply both $x_1\wedge y_1 \equiv x_2\wedge y_2$ and $x_1\vee y_1 \equiv x_2\vee y_2$. We say that a congruence $\equiv$ \emph{contracts} a join-irreducible $j$ if $j_*\equiv j$.
Two different congruences always contract different sets of join-irreducibles. Thus we will identify the congruences with the set of join-irreducibles they contract.
 Let $D$, called the \emph{join dependency relation}, be the binary relation on $\mathrm{JIrr}(L)$ defined by $a D b$ if $a\neq b$ and there exists $p\in L$ such that $a\leq b \vee p$ and $a \not \leq b_* \vee p$. Note that $a D b$ implies $a\not\leq b$.
A \emph{$D$-cycle} is a sequence of elements $a_1,\,a_2,\,\dots,a_k$ with $k\geq 2$ such that $a_1 D a_2 D \cdots D a_k D a_1$.

\begin{proposition}[\cite{Day_1979,freese1995free}]
\label{propDayLW}
    The lattice $L$ is join-congruence uniform if and only if it contains no $D$-cycles. In this case, the congruences of $L$ correspond to the subsets $S \subseteq \mathrm{JIrr}(L)$ such that if $a D b$ and $b\in S$, then $a\in S$.
\end{proposition}

A lattice congruence on a lattice $L$ gives rise to a \emph{lattice quotient}, which is the lattice on the equivalence classes ordered by the induced order of $L$ on the minimum elements of these classes (which are in fact intervals).

\subsection{Higher torsion classes}
\label{sec:backgroundhighertorsionclasses}

The goal of this section is to introduce the lattices of higher torsion classes and to recall the combinatorial descriptions of these higher torsion classes for the higher Auslander and Nakayama algebras of type $\mathbf{A}$.

We first recall some vocabulary of the representation theory of associative algebras (see for example \cite{Assem_Skowronski_Simson_2006}).
Let $A$ be a finite-dimensional $\mathbb{K}$-algebra, where $\mathbb{K}$ is an algebraically closed field. We denote by $\mathrm{mod} A$ the category of finite-dimensional right $A$-modules. A module is called \emph{indecomposable} if it is not isomorphic to a direct sum of two non-zero modules. Subcategories are assumed to be full and closed under finite direct sums and summands. It follows that a subcategory is the additive hull of its indecomposable modules. Thus we will identify a subcategory with the set of its indecomposable modules. The \emph{torsion classes} are the subcategories of $\mathrm{mod} A$ closed under extensions and quotients. They form a complete lattice with meet given by intersection that was extensively studied \cite{demonet2023lattice,thomas2021introduction}.

Higher homological algebra is a subject whose study was initiated by O. Iyama \cite{IYAMA1,IYAMA2,IYAMA3}. We are interested in a generalization of the torsion classes in the context of higher homological algebra. Let $d$ be a positive integer. The following is an equivalent definition of the $d$-cluster tilting subcategories (we take it as the definition to avoid introducing more notations):

\begin{definition}[{\cite[Proposition 4.4]{kvamme2021d}}]
A subcategory $\mathcal{M}$ of $\mathrm{mod} A$ is called \emph{$d$-cluster tilting} if $\mathrm{Ext}_A^i(\mathcal{M},\mathcal{M})=0$ for all $0<i<d$, and if for all $M\in \mathrm{mod} A$ there exist exact sequences
$$0\rightarrow X_d\rightarrow \cdots \rightarrow X_1 \rightarrow M\rightarrow 0 \qquad \text{and} \qquad 0\rightarrow M\rightarrow X^1\rightarrow \cdots  \rightarrow X^d\rightarrow 0$$
with $X_1,\dots ,X_d,X^1,\dots ,X^d \in \mathcal{M}$.
\end{definition}

In the sequel, $\mathcal{M}$ is a $d$-cluster tilting subcategory of $\mathrm{mod} A$.

\begin{definition}[\cite{jorgensen2016torsion}]
A subcategory $\mathcal{U}$ of $\mathcal{M}$ is a \emph{$d$-torsion class} if for all $M\in \mathcal{M}$ there exists an exact sequence $0\rightarrow \mathcal{U}_M \rightarrow M\rightarrow X^1\rightarrow \cdots \rightarrow X^d\rightarrow 0$
with $X^1,\dots,X^d \in \mathcal{M}$ 
 and $\mathcal{U}_M\in \mathcal{U}$, such that the following sequence is exact for all $U\in \mathcal{U}$:
$$0\rightarrow \mathrm{Hom}_A (U,X^1) \rightarrow \cdots \rightarrow \mathrm{Hom}_A (U,X^d)\rightarrow 0 .$$
\end{definition}

The $d$-torsion classes are also called \emph{higher torsion classes}. When $d=1$ we have $\mathcal{M}=\mathrm{mod} A$ and the $1$-torsion classes correspond to the usual torsion classes. Thus the $d$-torsion classes are a generalization of the usual torsion classes in higher homological algebra. What makes this generalization maybe more challenging to work with is that the $d$-cluster tilting subcategories are not abelian in general, so one needs to be careful when using standard homological algebra techniques.

\begin{definition}
A subcategory $\mathcal{U}$ of $\mathcal{M}$ is closed 
under 
\begin{itemize}
\item[$(1)$] \emph{$d$-extensions} if for all exact sequences
$0\rightarrow X\rightarrow X_1 \rightarrow \cdots \rightarrow X_d \rightarrow Y\rightarrow 0$
of $\mathcal{M}$ with $X,Y\in \mathcal{U}$, there exists an equivalent $d$-extension where all the objects are in $\mathcal{U}$.
\item[$(2)$] \emph{$d$-quotients} if for all $f:M\rightarrow U$ from $M\in\mathcal{M}$ to $U\in\mathcal{U}$, there exists an exact sequence
$M\rightarrow U\rightarrow X_1 \rightarrow \cdots \rightarrow X_d \rightarrow 0$
with $X_1,\dots,X_d\in \mathcal{U}$.
\end{itemize}
\end{definition}

The main result of \cite{August_2025} is the following equivalent definition of the $d$-torsion classes:

\begin{theorem}[{\cite[Theorem 1.1]{August_2025}}]
\label{thmAHdtorsion}
A subcategory $\mathcal{U}$ of $\mathcal{M}$ is a $d$-torsion class if and only if it is closed under $d$-extensions and $d$-quotients. 
\end{theorem}

It follows from \cref{thmAHdtorsion} that an arbitrary intersection of $d$-torsion classes is a $d$-torsion class, which proves the following:

\begin{theorem}[{\cite[Theorem 1.2]{August_2025}}]
    The poset of $d$-torsion classes in $\mathcal{M}$ ordered by inclusion is a complete lattice with meet given by intersection.
\end{theorem}

In higher homological algebra, the two main examples where we are able to compute and have a combinatorial understanding are the $(d-1)$-Auslander and the $(d-1)$-Nakayama algebras of type $\textbf{A}$ (introduced respectively in \cite{IYAMA3} and \cite{Jasso_2019}). These algebras are also called respectively \emph{higher} Auslander and Nakayama algebras of type $\mathbf{A}$. We do not recall their definitions, but the interested reader can look at \cite[Section 5 and 6]{August_2025}, where the definitions and essential properties are given. The only thing we need to know is that the module categories of these algebras always contain a unique $d$-cluster tilting subcategory, whose indecomposable modules are respectively indexed by $os_n^{d+1}$ and $os_{\underline{l}}^{d+1}$, two posets that we will shortly define.
When we will talk about the $d$-torsion classes of one of these algebras, we mean the $d$-torsion classes of its unique $d$-cluster tilting subcategory. We finish this background by recalling the combinatorial descriptions of the higher torsion classes of these algebras obtained in \cite{August_2025}. First we need to introduce some notations.

Let $n$ be a positive integer. Let $os_{n}^{d}$ be the set of non-decreasing sequences $x=(x_1,\dots,x_d)$ of size $d$ of elements of $\{ 0,\,1,\,\dots,\,n-1\}$. Let $\leq$ be the product order on $os_{n}^{d}$ (see \cref{fig:posetos33} for an example). On $os_{n}^{d}$ we also define a binary relation $x\rightsquigarrow y$ if  $x_1\leq y_1\leq x_2 \leq y_2\leq \dots \leq x_{d} \leq y_{d}$. On $\mathbb{Z}^{d}$ we define the map $\tau_{d}(x_1,\dots , x_d) = (x_1-1, x_2-1,\dots , x_d-1)$.

\begin{theorem}[{\cite[Theorem 5.13]{August_2025}}]
\label{thm:torsionAuslander}
The $d$-torsion classes of the $(d-1)$-Auslander algebra of type $\textbf{A}_n$ are identified with the subsets $T\subseteq os_n^{d+1}$ that satisfy the following two conditions:
\begin{enumerate}
\item \label{1aus} For all $i<n$, $\{x\in T\,|\,x_{d+1}=i\}$ is an order filter of the poset $\{x\in os_n^{d+1}\,|\,x_{d+1}=i\}$ ordered with the product order. 
\item \label{2aus} For all $x,z\in T$, if $x\rightsquigarrow \tau_d(z)$, then any $y\in os_n^{d+1}$ with $y_i\in \{x_i,z_i\}$ for each $i$ must be in $T$.
\end{enumerate}
\end{theorem}

Denote by $L_n^d$ the lattice of the $d$-torsion classes of the $(d-1)$-Auslander algebra of type $\mathbf{A}_n$. The only combinatorial properties that were discussed in \cite{August_2025} about the lattices $L_n^d$ are that they are not semidistributive nor Hasse-regular in general. Indeed, $L_3^3$ is neither semidistributive nor Hasse-regular (\cite[Example 5.21]{August_2025}). We will give different combinatorial properties satisfied by $L_n^d$ in \cref{sec:Auslander}.

\begin{remark}
Condition (\ref{2aus}) of \cref{thm:torsionAuslander} is harder to manipulate than condition (\ref{1aus}), but in \cref{sec:Auslander} we will give a reformulation of \cref{thm:torsionAuslander}, which will greatly simplify the study of these $d$-torsion classes.       
\end{remark}

A \emph{Kupisch series of type $\textbf{A}_n$} is a sequence $\underline{l}=(l_0,l_1,\dots,l_{n-1})$ of positive integers satisfying $l_0=1$ and $\forall i\geq 1,\, 2\leq l_i\leq l_{i-1}+1$. Let 
$os_{\underline{l}}^{d+1} := \{y\in os_n^{d+1} \mid y_1\geq y_{d+1}-l_{y_{d+1}} +1 \}  \subseteq os_n^{d+1}$.

\begin{theorem}[{\cite[Theorem 6.1]{August_2025}}]
\label{thm:torsionNakayama}
The $d$-torsion classes of the $(d-1)$-Nakayama algebra associated to the Kupisch series $\underline{l}$ are identified with the subsets $T\subseteq os_{\underline{l}}^{d+1}$ that satisfy the following two conditions:
\begin{enumerate}
\item For all $i<n$, $\{x\in T\,|\,x_{d+1}=i\}$ is an order filter of the poset $\{x\in os_{\underline{l}}^{d+1}\,|\,x_{d+1}=i\}$ ordered with the product order. 
\item For all $x,z\in T$, if $x\rightsquigarrow \tau_d(z)$, then any $y\in os_{\underline{l}}^{d+1}$ with $y_i\in \{x_i,z_i\}$ for each $i$ must be in $T$.
\end{enumerate}
\end{theorem}

Denote by $L_{\underline{l}}^d$ the lattice of the $d$-torsion classes of the $(d-1)$-Nakayama algebra of type $\textbf{A}$ associated to $\underline{l}$. This is a generalization of the case of the higher Auslander algebras, as this particular case is obtained by taking $\underline{l}=(1,2,\dots,n)$. We will give different combinatorial properties satisfied by $L_{\underline{l}}^d$ in \cref{sec:Nakayama}.

\section{$(P,\phi)$-Tamari lattices}
\label{sectionPphiTamari}

\subsection{Definition of the $(P,\phi)$-Tamari lattices}
\label{sectionGenerality}

Let $P$ be a poset such that all its principal order ideals are finite. Let $\phi: \phi(0) < \phi(1) < \cdots$ be any chain of $P$. 
For any $i\geq 0$, denote $\mathcal{C}_i := I_P(\phi(i)) \times \{i\}$ where $\times$ is the direct product. The finite poset $\mathcal{C}_i$ is called the \emph{$i$-th component}
of the poset $\mathcal{C}_P^{\phi}:=\bigsqcup_{i\geq 0} \mathcal{C}_i$. Very particular elements of $\mathcal{C}_P^{\phi}$ that will play an important role are the $(\phi(k),i)$ for all $0\leq k\leq i$. In $\mathcal{C}_i$, these elements form a chain $(\phi(i),i)>(\phi(i-1),i) >\cdots >(\phi(0),i)$ and $(\phi(i),i)$ is the maximum element of $\mathcal{C}_i$. 
If $T$ is an order filter of $\mathcal{C}_P^{\phi}$, then for any $i$, we denote $T_i:=T\cap \mathcal{C}_i$ and $T_{<i}:=\sqcup_{k<i} T_k$ (and similarly for $T_{\leq i}$, $T_{>i}$ and $T_{\geq i}$). See Figures \cref{fig:figsquarephi2} and \cref{fig:firstExample} for examples of the following definition.

\begin{definition}
\label{def:TamPphi}
An order filter $T$ of $\mathcal{C}_P^{\phi}$ is \emph{torclosed} if and only if for all $i<j$, having both $(x,i) \in T$ and $(\phi(i+1),j) \in T$ implies $(x,j)\in T$. 
We call \emph{$(P,\phi)$-Tamari} the poset on the torclosed sets of $\mathcal{C}_P^{\phi}$ ordered by inclusion. Since the intersection of torclosed sets is torclosed and this poset has a maximum $\mathcal{C}_P^{\phi}$, it is a complete lattice with meet given by intersection, that we denote by \emph{$\Tam(P,\phi)$}.
For $T$ an order filter of $\mathcal{C}_P^{\phi}$, we define the \emph{completion} of $T$ to be the minimal torclosed set containing $T$.
\end{definition}

Except in this section, in this article we will restrict our study to the finite $(P,\phi)$-Tamari lattices. Moreover, when we will look at applications of this construction in \cref{sectionMainExamples}, all our examples will be with posets $P$ having a minimum $\hat{0}$. To have more concise statements, we will often restrict to this setting (but more general statements can sometimes be derived).

\begin{figure}
\begin{subfigure}{.5\textwidth}
\centering
\begin{tikzpicture}
\begin{scope}
\node (a) at (0,0) {$\phi(0)$};
\node (b) at (-1,1) {$b$};
\node (c) at (1,1) {$c$};
\node (d) at (0,2) {$\phi(1)$};
\draw[->,>=latex] (a) -- (b);
\draw[->,>=latex] (b) -- (d);
\draw[->,>=latex] (a) -- (c);
\draw[->,>=latex] (c) -- (d);
\draw (0,-1) node {$P$};
\end{scope}

\begin{scope}[xshift=2.9cm]
\node[scale=0.8] (5) at (-0.3,0) {$(\phi(0),0)$};
\node[scale=0.8] (4) at (2,0) {$(\phi(0),1)$};
\node[scale=0.8] (3) at (1,1) {$(b,1)$};
\node[scale=0.8] (2) at (3,1) {$(c,1)$};
\node[scale=0.8] (1) at (2,2) {$(\phi(1),1)$};
\draw[->,>=latex] (4) -- (2);
\draw[->,>=latex] (2) -- (1);
\draw[->,>=latex] (4) -- (3);
\draw[->,>=latex] (3) -- (1);
\draw (-0.3,-0.8) node {$\mathcal{C}_0$};
\draw (2,-0.8) node {$\mathcal{C}_1$};
\draw [decorate,decoration={brace,amplitude=7pt, mirror,raise=2ex}]
(-0.8,-1) -- (3.3,-1) node[midway,yshift=-2.5em]{$\mathcal{C}_P^{\phi}$};
\end{scope}
\end{tikzpicture}
    \caption{A poset $P$ with a chain $\phi$, and the poset $\mathcal{C}_P^{\phi}$.}
    \label{fig:posetsquarephi2}
\end{subfigure}%
\begin{subfigure}{.5\textwidth}
\centering
    \begin{tikzpicture}
\begin{scope}[xscale=1.5,yscale=1.4]
\node[scale=0.8] (v0) at (0,0) {$\emptyset$};
\node[scale=0.8] (v1) at (1.5,1) {$\{(\phi(1),1)\}$};
\node[scale=0.8] (v2) at (0.5,2) {$\{(\phi(1),1),(b,1)\}$};
\node[scale=0.8] (v3) at (2.5,2) {$\{(\phi(1),1),(c,1)\}$};
\node[scale=0.8] (v4) at (-1,2.5) {$\{(\phi(0),0)\}$};
\node[scale=0.8] (v5) at (1.5,3) {$\{(\phi(1),1),(c,1),(b,1)\}$};
\node[scale=0.8] (v6) at (1.5,4) {$\{(\phi(1),1),(c,1),(b,1),(\phi(0),1)\}$};
\node[scale=0.8] (v7) at (0,5) {$\mathcal{C}_P^{\phi}$};

\draw[->,>=latex] (v0)--(v1);
\draw[->,>=latex] (v0)--(v4);
\draw[->,>=latex] (v1)--(v2);
\draw[->,>=latex] (v1)--(v3);
\draw[->,>=latex] (v2)--(v5);
\draw[->,>=latex] (v3)--(v5);
\draw[->,>=latex] (v5)--(v6);
\draw[->,>=latex] (v6)--(v7);
\draw[->,>=latex] (v4)--(v7);
\end{scope}
\end{tikzpicture}
\caption{$\Tam(P,\phi)$ from \cref{fig:posetsquarephi2}.}
\label{fig:latticesquarephi2}
\end{subfigure}
\caption{}
\label{fig:figsquarephi2}
\end{figure}

\begin{figure}
\begin{subfigure}{0.5\textwidth}
\centering
\begin{tikzpicture}
\begin{scope}[scale=1]
\node[draw,circle,inner sep=2pt] (a) at (0,0) {$a$};
\node[draw,circle,inner sep=2pt] (b) at (-1,1) {$b$};
\node (c) at (1,1) {$c$};
\node[draw,circle,inner sep=2pt] (d) at (0,2) {$d$};
\draw[->,>=latex] (a) -- (b);
\draw[->,>=latex] (b) -- (d);
\draw[->,>=latex] (a) -- (c);
\draw[->,>=latex] (c) -- (d);
\draw (0,-1) node {$P$};
\end{scope}

\begin{scope}[xshift= 2cm, scale= 1]
\node[draw,circle,inner sep=2pt] (7) at (0,0) {$7$};
\node[draw,circle,inner sep=2pt] (6) at (1,0) {$6$};
\node[draw,circle,inner sep=2pt] (5) at (1,1) {$5$};
\node[draw,circle,inner sep=2pt] (4) at (3,0) {$4$};
\node[draw,circle,inner sep=2pt] (3) at (2,1) {$3$};
\node (2) at (4,1) {$2$};
\node[draw,circle,inner sep=2pt] (1) at (3,2) {$1$};
\draw[->,>=latex] (6) -- (5);
\draw[->,>=latex] (4) -- (3);
\draw[->,>=latex] (3) -- (1);
\draw[->,>=latex] (4) -- (2);
\draw[->,>=latex] (2) -- (1);
\draw (0,-0.8) node {$\mathcal{C}_0$};
\draw (1,-0.8) node {$\mathcal{C}_1$};
\draw (3,-0.8) node {$\mathcal{C}_2$};
\draw [decorate,decoration={brace,amplitude=7pt, mirror,raise=2ex}]
(-0.2,-1) -- (3.7,-1) node[midway,yshift=-2.5em]{$\mathcal{C}_P^{\phi}$};
\end{scope}
\end{tikzpicture}
    \caption{A poset $P$ with a chain $\phi$ whose elements are circled, and the poset $\mathcal{C}_P^{\phi}$ whose elements $(\phi(i),j)$ are circled. The numbering is the one from \cref{sectionCongruenceTopo}}
    \label{fig:posetfirstExample}
\end{subfigure}%
\begin{subfigure}{.5\textwidth}
\centering
\begin{tikzpicture}
\begin{scope}[xscale=1.2]
\node[scale=0.7] (v17) at (0,7.5) {$\mathcal{C}_P^{\phi}$};
\node[scale=0.7] (v16) at (-2,4) {$\{5,6,7\}$};
\node[scale=0.7] (v15) at (4,5) {$\{1,2,3,4,7\}$};
\node[scale=0.7] (v14) at (4,3.5) {$\{1,2,7\}$};
\node[scale=0.7] (v13) at (4,2) {$\{1,7\}$};
\node[scale=0.7,red] (v12) at (-2,1) {$\{7\}$};
\node[scale=0.7] (v11) at (0.5,6) {$\{1,2,3,4,5,6\}$};
\node[scale=0.7,red] (v10) at (-1,3) {$\{5,6\}$};
\node[scale=0.7] (v9) at (2,5) {$\{1,2,3,4,5\}$};
\node[scale=0.7] (v8) at (1,4) {$\{1,2,3,5\}$};
\node[scale=0.7] (v7) at (0,3) {$\{1,3,5\}$};
\node[scale=0.7,red] (v6) at (-1,1) {$\{5\}$};
\node[scale=0.7,red] (v5) at (3,4) {$\{1,2,3,4\}$};
\node[scale=0.7] (v4) at (2,3) {$\{1,2,3\}$};
\node[scale=0.7,red] (v3) at (1,2) {$\{1,3\}$};
\node[scale=0.7,red] (v2) at (3,2) {$\{1,2\}$};
\node[scale=0.7,red] (v1) at (2.5,0.5) {$\{1\}$};
\node[scale=0.7] (v0) at (0,-0.5) {$\emptyset$};

\draw[->,>=latex] (v0)-- node[above,text=blue] {\scriptsize $1$} (v1);
\draw[->,>=latex] (v0)--node[above,text=blue] {\scriptsize $5$}(v6);
\draw[->,>=latex] (v0)--node[below,text=blue] {\scriptsize $7$}(v12);
\draw[->,>=latex] (v1)--node[left,text=blue] {\scriptsize $2$}(v2);
\draw[->,>=latex] (v1)-- node[below,text=blue] {\scriptsize $3$}(v3);
\draw[->,>=latex] (v1)-- node[below,text=blue] {\scriptsize $7$}(v13);
\draw[->,>=latex] (v2)--node[above,text=blue] {\scriptsize $3$}(v4);
\draw[->,>=latex] (v2)--node[left,text=blue] {\scriptsize $7$}(v14);
\draw[->,>=latex] (v3)--node[left,text=blue] {\scriptsize $2$}(v4);
\draw[->,>=latex] (v3)--node[right,text=blue] {\scriptsize $5$}(v7);
\draw[->,>=latex] (v4)--node[left,text=blue] {\scriptsize $4$}(v5);
\draw[->,>=latex] (v4)--node[right,text=blue] {\scriptsize $5$}(v8);
\draw[->,>=latex] (v5)--node[right,text=blue] {\scriptsize $5$}(v9);
\draw[->,>=latex] (v5)--node[left,text=blue] {\scriptsize $7$}(v15);
\draw[->,>=latex] (v6)--node[above,text=blue] {\scriptsize $1$}(v7);
\draw[->,>=latex] (v6)--node[left,text=blue] {\scriptsize $6$}(v10);
\draw[->,>=latex] (v7)--node[left,text=blue] {\scriptsize $2$}(v8);
\draw[->,>=latex] (v8)--node[left,text=blue] {\scriptsize $4$}(v9);
\draw[->,>=latex] (v9)--node[above,text=blue] {\scriptsize $6$}(v11);
\draw[->,>=latex] (v10)--node[above,text=blue] {\scriptsize $7$}(v16);
\draw[->,>=latex] (v10)--node[above,text=blue] {\scriptsize $1$}(v11);
\draw[->,>=latex] (v11)--node[right,text=blue] {\scriptsize $7$}(v17);
\draw[->,>=latex] (v12)--node[left,text=blue] {\scriptsize $5$}(v16);
\draw[->,>=latex]
  (v12)
  .. controls +(2,1) and +(-1,-1)
  ..node[pos=0.35,below,text=blue] {\scriptsize $1$} (v13);
\draw[->,>=latex] (v13)--node[right,text=blue] {\scriptsize $2$}(v14);
\draw[->,>=latex] (v14)--node[right,text=blue] {\scriptsize $3$}(v15);
\draw[->,>=latex] (v15)--node[above,text=blue] {\scriptsize $5$}(v17);
\draw[->,>=latex] (v16)--node[above,text=blue] {\scriptsize $1$}(v17);
\end{scope}
\end{tikzpicture}
\caption{$\Tam(P,\phi)$ from \cref{fig:posetfirstExample}.}
\label{fig:latticefirstExample}
\end{subfigure}
\caption{}
\label{fig:firstExample}
\end{figure}

\begin{remark}
\cref{def:TamPphi} is motivated by \cref{thm:torsionAuslander}. We will see in \cref{sec:Auslander} that for a particular choice of poset $P$ and chain $\phi$, the torclosed sets correspond to $d$-torsion classes and $\Tam(P,\phi)$ to the lattice of $d$-torsion classes ordered by inclusion. This explains the name \emph{torclosed} for \emph{torsion closed}. We will prove in \cref{sec:Chain} that $\Tam(C_n,C_n)\cong \Tam_{n+1}$ (\cref{propTamari}), which explains the name \emph{Tamari} for these lattices, as well as the fact that any $\Tam(P,\phi)$ with $|\phi|=n$ contains a sublattice isomorphic to $\Tam_{n+1}$ (\cref{propcompareTamari}).
\end{remark}

The following four results are immediate:

\begin{lemma}
 $\Tam(P,\phi)$ is finite if and only if $\phi$ is finite.   
\end{lemma}

\begin{lemma}
  In $\Tam(P,\phi)$, we have $I_{\Tam(P,\phi)}\big(\bigcup_{i<k} \mathcal{C}_i \big) =\Tam(P,\phi_{|\{0,1,\dots,k-1\}})$ for every $k$.
\end{lemma}

\begin{lemma}
Let $T,T' \in \Tam(P,\phi)$. Then $(T\bigvee T')_{<k}=T_{<k} \bigvee T'_{<k}$, where on the right this is the join in the finite lattice $\Tam(P,\phi_{|\{0,1,\dots,k-1\}})$.
\end{lemma}

\begin{lemma}
\label{lem:procheenproche}
Let $k$ be a nonnegative integer. Let $T\in \Tam(P,\phi_{|\{0,1,\dots,k-1\}})$ and $F$ be an order filter of $\mathcal{C}_k$. The completion of $T\cup F$ only differs from $T\cup F$ in the component $\mathcal{C}_k$. It can be obtained by first computing the completion of $T_{k-1}\cup F$, whose elements in $\mathcal{C}_k$ form an order filter denoted by $F_1$. Then computing the completion of $T_{k-2}\cup F_1$, whose elements in $\mathcal{C}_k$ form an order filter denoted by $F_2$. And so on until obtaining $F_k$. Then the completion of $T\cup F$ is $T\cup F_k$.
\end{lemma}

\begin{proposition}
     $\Tam(P,\phi)$ has a minimum element $\emptyset$ and a maximum element $\mathcal{C}_P^{\phi}$. Its atoms are $\{(\phi(i),i)\}$ for all $i\geq 0$. Let $k\geq 0$.
     If $m$ is a minimal element of $I_P(\phi(k))$ that is not in $I_P(\phi(k-1))$, then $\mathcal{C}_P^{\phi} \setminus \{m\}$ is a coatom. If no such elements exist then $\mathcal{C}_P^{\phi} \setminus \mathcal{C}_k$ is a coatom. All coatoms are obtained in this way.
\end{proposition}

\begin{proof}
We only prove the statement related to the coatoms, as the rest is straightforward. Any element of the form $\mathcal{C}_P^{\phi} \setminus \{m\}$, without assumptions on $m$, is a coatom. Let $k\geq 0$. The element $\mathcal{C}_P^{\phi} \setminus \mathcal{C}_k$ is torclosed, and any element that covers $\mathcal{C}_P^{\phi} \setminus \mathcal{C}_k$ contains $(\phi(k),k)$, which is the maximum of $\mathcal{C}_k$. Suppose that every minimal element of $I_P(\phi(k))$ is an element in $I_P(\phi(k-1))$. In this case, the minimal torclosed set that contains both the elements of $\mathcal{C}_{k-1}$ and the element $(\phi(k),k)$ is $\mathcal{C}_{k-1} \sqcup \mathcal{C}_k$. This proves that $\mathcal{C}_P^{\phi} \setminus \mathcal{C}_k$ is a coatom.

Conversely, suppose that $T$ is a coatom. If there exist two elements $i<j$ such that $T_i\neq \mathcal{C}_i$ and $T_j\neq \mathcal{C}_j$, then $T < T\vee \mathcal{C}_j \neq \mathcal{C}_P^{\phi}$ since $T\vee \mathcal{C}_j$ does not contain $\mathcal{C}_i$, which proves that $T$ is not a coatom. Thus there exists a unique $k$ such that $T_k \neq \mathcal{C}_k$. Minimal elements of $I_P(\phi(k))$ that are not in $I_P(\phi(k-1))$ cannot be obtained through a completion. Thus we must have $T=\mathcal{C}_P^{\phi} \setminus \{m\}$ with $m$ such an element if these minimal elements exist. But if there are no such elements, we saw that we must have $T=\mathcal{C}_P^{\phi} \setminus \mathcal{C}_k$ since the minimal torclosed set containing $T\cup \{(\phi(k),k)\}$ is the maximum $\mathcal{C}_P^{\phi}$.
\end{proof}

\subsection{Lattice properties of $\Tam(P,\phi)$}
\label{sec:finiteTam}

In the remainder of this article, we always assume that $\Tam(P,\phi)$ is finite, and that the chain $\phi$ in $P$ is given by $\phi(0)<\phi(1)<\cdots <\phi(n-1)$.

\begin{lemma}
\label{lemCovers}
Let $T\lessdot T'$ be a cover relation of $\Tam(P,\phi)$. Then there exists $k<n$ such that $T'\setminus T \subseteq \mathcal{C}_k$ and $T'\setminus T$ has a maximum element.
\end{lemma}

\begin{proof}
Recall that $T<T'$ means $T\subsetneq T'$. Let $k$ be the biggest nonnegative integer such that $T'_k\setminus T_k \neq \emptyset$. Let $R$ be the subset of $T'$ obtained from $T$ by replacing $T_k$ with $T'_k$. Then $R$ is torclosed, and $T<R\leq T'$. Since $T\lessdot T'$, we have $T'=R$. This proves that there exists $k<n$ such that $T'\setminus T \subseteq \mathcal{C}_k$. 
Now let us prove that $T'\setminus T$ has a maximum. We know that $T'\setminus T \subseteq \mathcal{C}_k$. Seeking a contradiction, suppose that $m_1$ and $m_2$ are two maximal elements of $T'\setminus T$. It could be that one of them is of the form $(\phi(i),k)$ with $i\leq k$. But since these latter elements form a chain in $\mathcal{C}_k$, and $m_1$ and $m_2$ are incomparable, without loss of generality we can assume that $m_1$ is not of the form $(\phi(i),k)$ with $i\leq k$. Thus $T\cup \{m_1\}$ is torclosed, but it is strictly between $T$ and $T'$, which is a contradiction with the fact that $T\lessdot T'$ is a cover relation.
\end{proof}

We now characterize the join-irreducible elements of $\Tam(P,\phi)$ (for an example see \cref{fig:latticefirstExample} where these elements are colored in red). 

\begin{proposition} \label{prop:joinirr}
The join-irreducible elements of $\Tam(P,\phi)$ are the principal order filters of $\mathcal{C}_P^{\phi}$. Thus the subposet of $\Tam(P,\phi)$ formed from 
its join-irreducible elements is the dual of the poset $\mathcal{C}_P^{\phi}$. 
Then $|\,\mathrm{JIrr}(\Tam{(P,\phi)})\,| = |\mathcal{C}_P^{\phi}|$.
\end{proposition}

\begin{proof}
It is immediate that these order filters are join-irreducible elements. We need to prove that these are the only join-irreducible elements. 
Seeking a contradiction, suppose that $T\in \mathrm{JIrr}(\Tam(P,\phi))$ is not of the previous form. Let $k\geq 0$ be minimal such that $T_k \neq \emptyset$. If $T_k \subseteq \mathcal{C}_k$ has two minimal elements $a$ and $b$, then $T$ covers both $T\setminus \{a\}$ and $T\setminus \{b\}$, which is absurd since $T$ is a join-irreducible element. Thus there exists $a$ such that $T_k$ is the principal order filter of $\mathcal{C}_k$ generated by $a$. Then the only torclosed set that $T$ covers is $T\setminus \{a\}$. By hypothesis $T\neq T_k$, thus there exists $j>k$ such that $T_j\neq \emptyset$. The subset $T\setminus T_j$ is torclosed, and $(T\setminus T_j) < T$. But $(T\setminus T_j) \not\leq T\setminus \{a\}$, which is absurd since the only element that $T$ covers is $T\setminus \{a\}$. 

The other statements follow immediately from what we just proved.
\end{proof}

We denote $J_{(x,i)} = F_{\mathcal{C}_P^{\phi}}((x,i))$, which is the minimal torclosed set containing $(x,i)$. Then \cref{prop:joinirr} tells us that the join-irreducibles of $\Tam(P,\phi)$ are the $J_{(x,i)}$ for $(x,i)\in \mathcal{C}_P^{\phi}$.

For any $x\leq \phi(n-1)$, if $x\not\geq \phi(0)$, let $l(x)=0$, otherwise let $l(x)=\mathrm{max}\{i<n\mid x\geq \phi(i)\}$. 
In order to be able to analyze the meet-irreducible elements, we add the hypothesis that $P$ has a minimum, which is not very restrictive as all the examples that we will see have this property.

\begin{proposition} \label{prop:meetirr}
Suppose that $P$ has a minimum element $\hat{0}$.
The meet-irreducible elements of $\Tam(P,\phi)$ are of two kinds, which are the following. For any $(x,k)\in \mathcal{C}_P^{\phi}$, the following element is the \emph{first kind} of meet-irreducible:
$$M_{(x,k)} = \Big(\bigsqcup_{0\leq i<l(x)} \mathcal{C}_i \Big) \, \bigsqcup \,\Big(\bigsqcup_{l(x)\leq i \leq k} \{(y,i) \in \mathcal{C}_P^{\phi} \mid y\not\leq x\}\,\Big)\,\bigsqcup\, \Big(\bigsqcup_{k<i<n} \mathcal{C}_i\Big) .$$

\noindent For any $(x,k)\in \mathcal{C}_P^{\phi}$ and any $(\phi(j),k)$ with $1\leq j\leq k$ such that $x\not\sim\phi(j)$, the following element is the \emph{second kind} of meet-irreducible:
$$M_{(x,k),(\phi(j),k)} = \Big(\bigsqcup_{0\leq i<j} \mathcal{C}_i \Big)\,\bigsqcup\,\Big(\bigsqcup_{j\leq i \leq k} \{(y,i) \in \mathcal{C}_P^{\phi} \mid y\not\leq x,\,y\not \leq \phi(j)\}\,\Big) \,\bigsqcup \,\Big(\bigsqcup_{k<i<n} \mathcal{C}_i\Big) .$$
\end{proposition}

\begin{proof}
We start by proving that any meet-irreducible is of one of the two kinds. Let $T$ be a meet-irreducible, covered by $T^*$. 
Let $k<n$ be the biggest nonnegative integer such that $T_k\neq \mathcal{C}_k$. Since $T<T\cup \mathcal{C}_k$, then $T^{*}\setminus T \subseteq \mathcal{C}_k$. Moreover, by \cref{lemCovers}, $T^{*}\setminus T$ has a maximum element $(x,k)$. In particular, $(x,k)$ is a maximal element of $\mathcal{C}_k \setminus T_k$.
In fact, for any maximal element $m$ of $\mathcal{C}_k \setminus T_k$ such that $m\neq (\phi(i),k)$ for every $i\leq k$, we have $T\lessdot T\cup \{m\}$. Since $T$ is meet-irreducible and the elements $(\phi(i),k)$ for $i\leq k$ form a chain, then $(x,k)$ is the maximum of $\mathcal{C}_k \setminus T_k$, or there exists $j\leq k$ such that the maximal elements of $\mathcal{C}_k \setminus T_k$ are $(x,k)$ and $(\phi(j),k)$. Let us prove that in the former case $T=M_{(x,k)}$, whereas in the latter case $T=M_{(x,k),(\phi(j),k)}$. In the remainder of this proof, for these two different cases we will refer to the \emph{former case} and to the \emph{latter case}.

Let $A_i=\{(y,i)\in \mathcal{C}_P^{\phi} \mid y\not\leq x\}$ and $B_i=\{(y,i) \in \mathcal{C}_P^{\phi} \mid y\not\leq x,\,y\not\leq \phi(j)\}$.
In the former case we have $(\phi(l(x)+1),k)\in T$ but $(x,k)\not\in T$, and in the latter case we have $(\phi(j+1),k)\in T$ but $(x,k)\not\in T$ and $(\phi(j),k)\not\in T$. Since $T$ is torclosed, then in the former case for all $i\in \{l(x),l(x)+1,\dots,k\}$, we have $T_i\subseteq A_i$, and in the latter case for all $i\in \{j,j+1,\dots,k\}$, we have $T_i\subseteq B_i$.

Then this forces in the former case to have $T_{<l(x)} = \bigsqcup_{0\leq i<l(x)} \mathcal{C}_i$. Indeed, since $T_i\subseteq A_i$ for all $i\in \{l(x),l(x)+1,\dots,k\}$, then $(\phi(p),i)\not \in T$ for all $p\leq l(x)$ and $i\in \{l(x),l(x)+1,\dots,k\}$. Moreover, $T_q=\mathcal{C}_q$ for all $k<q<n$, which proves that the addition of elements of $\sqcup_{i<l(x)} \,\mathcal{C}_i$ to $T$ does not affect the others components of $T$. Thus $T\leq T\vee \big(\sqcup_{i<l(x)} \,\mathcal{C}_i\big) =  T\cup \big(\sqcup_{i<l(x)} \,\mathcal{C}_i\big)$. If the inequality is strict, then $T^{*}\setminus T \subseteq \sqcup_{i<l(x)} \,\mathcal{C}_i$, which is absurd. Thus, we have $T=T\cup \big(\sqcup_{i<l(x)} \,\mathcal{C}_i\big)$. This proves that $T_{<l(x)} = \bigsqcup_{0\leq i<l(x)} \mathcal{C}_i$. We would prove very similarly, by replacing $l(x)$ by $j$, that in the latter case $T_{<j} = \bigsqcup_{0\leq i<j} \mathcal{C}_i$.

Moreover, it is immediate that $\sqcup_{l(x)\leq i\leq k}\,A_i$ and $\sqcup_{j\leq i\leq k}\,B_i$ are torclosed sets, and with the same argument as above we have that $T_{<l(x)} \sqcup \big(\sqcup_{l(x)\leq i\leq k}\,A_i\big) \sqcup T_{>k}$ and $T_{<j} \sqcup \big(\sqcup_{j\leq i\leq k}\,B_i\big) \sqcup T_{>k}$ are torclosed. But both of these torclosed sets contain $T$, and since they can only differ in components that are not the $k$-th component, then in fact they are equal to $T$ in the respective cases. 

It remains to prove that the two torclosed sets $M_{(x,k)}$ and $M_{(x,k),(\phi(j),k)}$ are meet-irreducibles. Adding to these torclosed sets respectively an element $(y,i) \not\in A_i$, for $i\in \{l(x),l(x)+1,\dots,k\}$, or $(y,i) \not\in B_i$, for $i\in \{j,j+1,\dots,k\}$, will by completion add $(y,k)$ to each of these torclosed sets. Then any torclosed set that covers one of these two torclosed sets add only elements of the $k$-th component. Thus any element that covers $M_{(x,k)}$ contains $(x,k)$, then $M_{(x,k)}$ is meet-irreducible and its only cover is the minimal torclosed set containing $M_{(x,k)}\cup \{(x,k)\}$. For $M_{(x,k),(\phi(j),k)}$, the minimal torclosed set containing $M_{(x,k),(\phi(j),k)}\cup \{(\phi(j),k)\}$ contains also $(x,k)$. Indeed, $(\hat{0},j-1) \in M_{(x,k),(\phi(j),k)}$ thus by completion $(\hat{0},k)$ is in this minimal torclosed set, then $(x,k)$ also. Thus $M_{(x,k),(\phi(j),k)}$ has only one cover which is $M_{(x,k),(\phi(j),k)} \cup \{(x,k)\}$.
\end{proof}


\begin{proposition}
\label{propJSD}
Let $L=\Tam{(P,\phi)}$.
The lattice $L$ is join-semidistributive. 
Suppose additionally that $P$ has a minimum element. 
Then $L$ is semidistributive if and only if each element of $I_P(\phi(n-1))$ is comparable to all the elements of the chain $\phi(1)<\phi(2)<\cdots <\phi(n-1)$.
\end{proposition}

\begin{proof}
We use \cref{lemEquiJSDpphi}. Let $T\lessdot T'$ in $L$. 
By \cref{lemCovers}, there exists $k<n$ such that $T'\setminus T \subseteq \mathcal{C}_k$ and it has a maximum element $m$. The torclosed set generated by $m$, which is $J_m$, is the minimum element of $I_L(T')\setminus I_L(T)$. Using \cref{lemEquiJSDpphi}, this proves that $L$ is join-semidistributive.

The second statement follows from \cref{lemJSDjoinmeetSD}, with the number of join-irreducibles given in \cref{prop:joinirr} which is equal to the number of meet-irreducibles if and only if we do not have meet-irreducibles of the second kind as described in \cref{prop:meetirr}.
\end{proof}

\begin{example}
See \cref{fig:firstExample} for the smallest counter-example to semidistributivity. The number of join-irreducibles in this figure is $7$, whereas the number of meet-irreducibles is $8$. This is because we have the meet-irreducible of the second kind $M_{2,3}=\{1,7\}$.    
\end{example}

The following result proves that the distributive lattices that are $(P,\phi)$-Tamari lattices are the ones that have exactly one atom. Thus by forgetting the minimum $\emptyset$ of these $(P,\phi)$-Tamari lattices, we recover all distributive lattices.

\begin{proposition}
If $\phi$ has only one element, then $\Tam(P,\phi) \cong J(I_P(\phi(0))^{op})$. These correspond to all the distributive lattices that have exactly one atom. If $|\phi|\geq 2$, then $\Tam(P,\phi)$ is not graded, and thus is not distributive.
\end{proposition}

\begin{proof}
If $\phi$ has one element, then the torclosed sets are the order filters of $I_P(\phi(0))$. Then the result follows from the fundamental theorem of finite distributive lattices. Suppose that $|\phi|\geq 2$. By \cref{prop:longestchainsTam}, $\ell(\Tam(P,\phi))=|\mathcal{C}_P^{\phi}|$. But all maximal chains in the order filter of $\Tam(P,\phi)$ generated by the atom $\{(\phi(0),0)\}$ have length at most $|\mathcal{C}_P^{\phi}|-2=\ell(\Tam(P,\phi))-2$. Indeed, the moment we will add $(\phi(1),1)$, the completion will also add with it at least another element $(\phi(0),1)$. Thus the maximal chains of $\Tam(P,\phi)$ do not have all the same length, which means that $\Tam(P,\phi)$ is not graded.
\end{proof}

\begin{proposition}
Suppose that $P$ has a minimum $\hat{0}$ and $\phi(0)=\hat{0}$. Let $L=\Tam(P,\phi)$. Then $L$ is complemented and pseudocomplemented.
\end{proposition}

\begin{proof}
Let $T\in L$ be a torclosed set that is not $\emptyset$ neither $\mathcal{C}_P^{\phi}$. Let $I=\{i<n \mid T_i\neq \emptyset\}$. Since $\phi(0)=\hat{0}$ and $T\neq \mathcal{C}_P^{\phi}$, then $J:=\{0,1,\dots,n-1\}\setminus I$ is not empty. Indeed, if $J=\emptyset$ then $(\phi(i),i)\in T$ for all $i<n$, and successive completions would give $(\hat{0},i)\in T$ for all $i<n$, but then $T=\mathcal{C}_P^{\phi}$. Let $T':=\bigsqcup_{i\in J} \mathcal{C}_i$. Then $T'$ is torclosed and $T\wedge T'=\emptyset$. Also, we have $T\vee T' = \mathcal{C}_P^{\phi}$ for the same reason as why $J\neq \emptyset$ using $\phi(0)=\hat{0}$. This proves that $T'$ is a complement of $T$. Thus $L$ is complemented. Moreover, it is immediate that $T'=\max\,\{R\in L\mid T\wedge R=\emptyset \}$, which proves that $L$ is meet-pseudocomplemented. We conclude that $L$ is pseudocomplemented using for example \cite[Théorème 4.3]{CHAMENINAMBUA199289}, which proves that a complemented lattice is pseudocomplemented if and only if it is meet-pseudocomplemented. 
\end{proof}

We define on $\mathcal{C}_P^{\phi}$ another partial order $\leq_{prod}$, called the \emph{product order}, defined by $(x,i) \leq_{prod} (y,j)$ if and only if $x\leq y$ and $i\leq j$. See \cref{fig:spineproductorder} for an example of this poset.

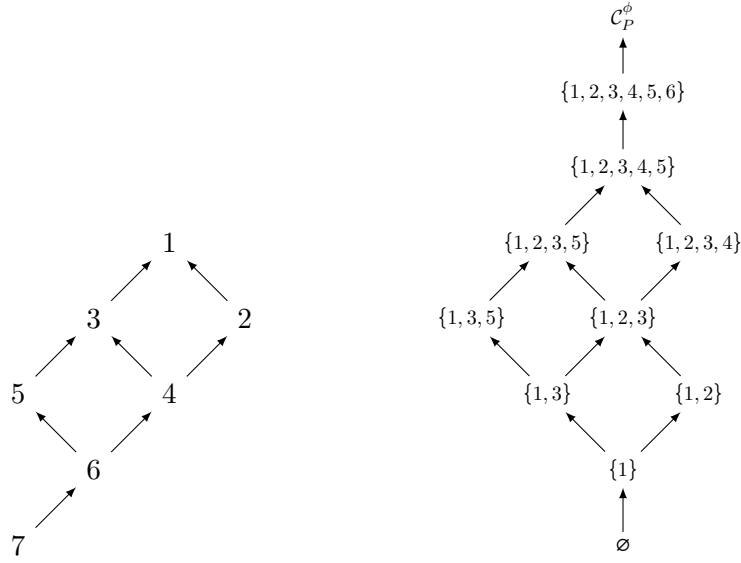
\begin{figure}
    \centering
\begin{tikzpicture}
\node (7) at (0,0) {$7$};   
\node (5) at (0,2) {$5$}; 
\node (6) at (1,1) {$6$}; 
\node (3) at (1,3) {$3$}; 
\node (4) at (2,2) {$4$}; 
\node (1) at (2,4) {$1$}; 
\node (2) at (3,3) {$2$}; 
\draw[->,>=latex] (7)--(6);
\draw[->,>=latex] (6)--(4);
\draw[->,>=latex] (4)--(2);
\draw[->,>=latex] (5)--(3);
\draw[->,>=latex] (3)--(1);
\draw[->,>=latex] (6)--(5);
\draw[->,>=latex] (4)--(3);
\draw[->,>=latex] (2)--(1);

\begin{scope}[xshift=6cm]
\node[scale=0.7] (v17) at (2,7) {$\mathcal{C}_P^{\phi}$};
\node[scale=0.7] (v11) at (2,6) {$\{1,2,3,4,5,6\}$};
\node[scale=0.7] (v9) at (2,5) {$\{1,2,3,4,5\}$};
\node[scale=0.7] (v8) at (1,4) {$\{1,2,3,5\}$};
\node[scale=0.7] (v7) at (0,3) {$\{1,3,5\}$};
\node[scale=0.7] (v5) at (3,4) {$\{1,2,3,4\}$};
\node[scale=0.7] (v4) at (2,3) {$\{1,2,3\}$};
\node[scale=0.7] (v3) at (1,2) {$\{1,3\}$};
\node[scale=0.7] (v2) at (3,2) {$\{1,2\}$};
\node[scale=0.7] (v1) at (2,1) {$\{1\}$};
\node[scale=0.7] (v0) at (2,0) {$\emptyset$};

\draw[->,>=latex] (v0)--(v1);
\draw[->,>=latex] (v1)--(v2);
\draw[->,>=latex] (v1)--(v3);
\draw[->,>=latex] (v2)--(v4);
\draw[->,>=latex] (v3)--(v4);
\draw[->,>=latex] (v3)--(v7);
\draw[->,>=latex] (v4)--(v5);
\draw[->,>=latex] (v4)--(v8);
\draw[->,>=latex] (v5)--(v9);
\draw[->,>=latex] (v7)--(v8);
\draw[->,>=latex] (v8)--(v9);
\draw[->,>=latex] (v9)--(v11);
\draw[->,>=latex] (v11)--(v17);
\end{scope}
\end{tikzpicture}
    \caption{On the left we have $(\mathcal{C}_P^{\phi},\leq_{prod})$ from \cref{fig:firstExample}, and on the right $J\big(\,(\mathcal{C}_P^{\phi},\leq_{prod})^{op}\,\big)$.}
    \label{fig:spineproductorder}
\end{figure}

\begin{proposition}
\label{prop:longestchainsTam}
The linear extensions $e_1\prec e_2 \prec \cdots \prec e_{m}$ of the dual of $(\mathcal{C}_P^{\phi},\leq_{prod})$ correspond to the longest chains $\emptyset \subsetneq \{e_1\} \subsetneq\{e_{1}, e_{2}\} \subsetneq\dots \subsetneq\{e_1,e_2,\dots , e_{m} \}=\mathcal{C}_P^{\phi}$ of $\Tam(P,\phi)$. Thus $\ell(\Tam{(P,\phi)})=|\mathcal{C}_P^{\phi}|$.
\end{proposition}

\begin{proof}
By definition, a linear extension of the dual of  $(\mathcal{C}_P^{\phi},\leq_{prod})$ is one that takes $(y,j)$ before taking $(x,i)$, for all $j\geq i$ and $y\geq x$ in $P$, where $x\leq \phi(i)$ and $y\leq \phi(j)$. It follows that if $e_1\prec e_2 \prec \cdots \prec e_{m}$ is a linear extension of the dual of $(\mathcal{C}_P^{\phi},\leq_{prod})$, then $\{e_1,e_2,\dots ,e_k\}$ is torclosed for any $k\in [m]$. Thus $\emptyset \subsetneq \{e_1\} \subsetneq\{e_{1}, e_{2}\} \subsetneq\dots \subsetneq\{e_1,e_2,\dots , e_{m} \}=\mathcal{C}_P^{\phi}$ is a maximal chain of $\Tam(P,\phi)$, and is obviously of longest length as we only add one new element at each step.
Thus the longest chains are of length $m=|\mathcal{C}_P^{\phi}|$, which is the length of $\Tam(P,\phi)$. Conversely, let $\emptyset \subsetneq \{e_1\} \subsetneq\{e_{1}, e_{2}\} \subsetneq\dots \subsetneq\{e_1,e_2,\dots , e_{m} \}=\mathcal{C}_P^{\phi}$ be a longest chain of $\Tam(P,\phi)$ (now we know that they have this form). Let $T(i)=\{e_1,\,e_2,\dots , e_i\}$. Then the longest chain is $\emptyset \subsetneq T(1) \subsetneq T(2) \subsetneq \cdots \subsetneq T(m)$. Let us prove that the sequence $e_1,\,e_2,\dots,\,e_m$ defines a linear extension of the dual of $(\mathcal{C}_P^{\phi},\leq_{prod})$.
Seeking a contradiction, suppose that there exists $i<j$ such that $e_i \leq_{prod} e_j$. Note that since $i<j$, then $e_j\not\in T(i)$. Let $e_i=(x,k)$ and $e_j=(y,p)$. Then $e_i \leq_{prod} e_j$ means $x\leq y$ in $P$ and $k\leq p$. If $p=k$, since $e_i=(x,k)\leq (y,k)=e_j$ and $e_i\in T(i)$ which is torclosed, then $e_j\in T(i)$. This is absurd, thus $p>k$.

Let $e_l=(\phi(k+1),p)$. If $e_l\in T(i)$, then by completion we have $(x,p)\in T(i)$. But $e_j=(y,p)\geq (x,p)$, thus since $T(i)$ is torclosed we have $e_j\in T(i)$. This is absurd, thus $e_l\not\in T(i)$. Then $i<l$. Let us explain why this is impossible. Since $(x,p)<(\phi(k+1),p)=e_l$ then a torclosed set that contains $(x,p)$ also contains $e_l$. Since both these elements are not in $T(i)$ and following the chain $\emptyset \subsetneq T(1) \subsetneq T(2) \subsetneq \cdots \subsetneq T(m)$ we add only a unique element at each step, we will add to a torclosed set that contains $T(i)$ the element $e_l$ before adding later the element $(x,p)$. But we proved that as soon as a torclosed set contains both $T(i)$ and the element $e_l$, then it contains $(x,p)$. This is absurd. This finishes the proof of the proposition.
\end{proof}

\begin{corollary} \label{coro:joinextremalTam}
The lattice $\Tam(P,\phi)$ is join-extremal.  
\end{corollary}

\begin{proof}
This follows from \cref{prop:joinirr} and \cref{prop:longestchainsTam}.  
\end{proof}

\begin{proposition}
\label{prop:spineTam}
The elements of the spine of $\Tam(P,\phi)$ correspond to the order filters of $(\mathcal{C}_P^{\phi},\leq_{prod})$. 
\end{proposition}

\begin{proof}
This follows from \cref{prop:longestchainsTam} where we proved that the torclosed sets on a longest chain of $\Tam(P,\phi)$ are the order ideals of the linear extensions of the dual of $(\mathcal{C}_P^{\phi},\leq_{prod})$. 
\end{proof}

\begin{example}
The spine of the lattice from \cref{fig:latticefirstExample} is drawn on the right of \cref{fig:spineproductorder}.   
\end{example}

\subsection{Enumeration of $\Tam(P,\phi)$}
\label{sec:enumgeneral}

In this short section, we study the number of elements of $\Tam(P,\phi)$.

Let $m_P^{\phi} = |\Tam{(P,\phi)}|$. Let $I_i=I_P(\phi(i))$ for any $i<n$. Then $|J(I_i)|$ is the number of order filters of $I_i$, which also corresponds to the number of order ideals or the number of antichains of $I_i$. 
We will later prove in \cref{propcompareTamari} that $\Tam(P,\phi)$ always contains a sublattice isomorphic to the Tamari lattice of size $n+1$, which gives $m_P^{\phi} \geq \frac{1}{n+2}\binom{2n+2}{n+1}$. Counting torclosed sets whose elements lie in only one component gives another lower bound $\sum_{i<n} |J(I_i)| -(n-1)$. 
Counting the order filters gives the upper bound $|J(I_0)|\times |J(I_1)| \times \cdots \times |J(I_{n-1})|$. 

\begin{proposition}
\label{prop:enumatmost2}
If $\phi$ has one element, then $m_P^{\phi}=|J(I_0)|$. If $\phi$ has two elements, then \mbox{$m_P^{\phi}=|J(I_0)|+|J(\mathcal{C}_P^{\phi},\leq_{prod})|-1$}.
\end{proposition}

\begin{proof}
The case $|\phi|=1$ is immediate. If $\phi$ has two elements, then a torclosed set may or may not contain $(\phi(1),1)$. The number of torclosed sets that do not contain $(\phi(1),1)$ is $|J(I_0)|$ since they correspond to the torclosed sets $T$ such that $T_1=\emptyset$. On the other hand, the number of torclosed sets containing $(\phi(1),1)$ is $|J(\mathcal{C}_P^{\phi},\leq_{prod})|-1$. Indeed, in this case an order filter $T$ of $\mathcal{C}_P^{\phi}$ is torclosed if and only if $(x,0) \in T$ implies $(x,1)\in T$.
\end{proof}

\begin{proposition}
\label{propComptage3}
Suppose that $P$ has a minimum $\hat{0}$ and $\phi(0)=\hat{0}$. If $\phi$ has three elements, then
\mbox{$m_P^{\phi}=1 + |J(I_1)| + \big|J(\mathcal{C}_P^{\phi}\,,\,\leq_{prod})\big| + |J(I_2\setminus I_1)|$}.
\end{proposition}

\begin{proof}
We first count the torclosed sets that do not contain $(\phi(2),2)$. Among these, the number of torclosed sets that do not contain $(\phi(1),1)$ is $2$, whereas we have $|J(I_1)|$ such torclosed sets that contain $(\phi(1),1)$. 
Now, let us count the torclosed sets that contain $(\phi(2),2)$. Among these, the number of torclosed sets that do not contain $(\phi(0),0)$ is $|J(\mathcal{C}_P^{\phi},\leq_{prod})|-2$, whereas we have $|J(I_2\setminus I_1)|+1$ such torclosed sets that contain $(\phi(0),0)$. The result follows by summing the cardinalities of these disjoint families of torclosed sets.
\end{proof}

\begin{example}
    For \cref{fig:latticefirstExample}, \cref{propComptage3} gives $m_P^{\phi} = 1+3+11+3=18$ elements. 
\end{example}

We will go slightly further, allowing $|\phi|=4$, in the particular example of the higher torsion classes (see \cref{prop:number4torsions}). We will also give a formula for the number of torclosed sets in the particular case where $P$ is a chain (see \cref{sec:Chain}). However, no general counting formula is known.

\subsection{Edge-labellings and congruences}
\label{sectionCongruenceTopo}

In this section, we obtain further properties of $\Tam(P,\phi)$, mainly results about its topology and congruences. We still assume that $\Tam(P,\phi)$ is finite, the chain being $\phi : \phi(0)<\phi(1)<\cdots <\phi(n-1)$.

By \cref{prop:longestchainsTam}, the choice of any linear extension $\mathcal{L}: e_1\prec e_2\prec \cdots \prec e_{m}$ of the dual of $(\mathcal{C}_P^{\phi},\leq_{prod})$ yields a longest chain 
$\psi_{\mathcal{L}} :\emptyset=T(0) \lessdot T(1) \lessdot T(2) \lessdot \cdots \lessdot T(m)=\mathcal{C}_P^{\phi}$ of $\Tam(P,\phi)$, where $T(i):=\{e_1,\,e_{2},\dots , e_i\}$. 
Moreover, $\mathcal{L}$ induces a numbering $\alpha_{\mathcal{L}}:\mathcal{C}_P^{\phi} \rightarrow [m]$ defined by $\alpha_{\mathcal{L}}(e_i)=i$.  For $T\subseteq \mathcal{C}_P^{\phi}$, we also define $\alpha_{\mathcal{L}}(T):=\{\alpha_{\mathcal{L}}(e_i) \mid e_i\in T\}$. This enables us to identify a torclosed set as a set of integers.
We define an edge-labelling $\omega_{\mathcal{L}}$ of $\Tam(P,\phi)$, which is well-defined by \cref{lemCovers}, which maps $T\lessdot T'$ to the integer $\alpha_{\mathcal{L}}(\max(T'\setminus T))$. This latter integer is easily seen to be equal to $\min\big(\alpha_{\mathcal{L}}(T')\setminus \alpha_{\mathcal{L}}(T)\big)$. This latter formulation shows that the edge-labelling $\omega_{\mathcal{L}}$ is particularly convenient if the torclosed sets are already represented as set of integers.

For the remainder of this section, let $\mathcal{L}$ be a linear extension of the dual of
$(\mathcal{C}_P^{\phi},\leq_{prod})$ such that:
\begin{itemize}
\item $(y,j)\prec (x,i)$ whenever $j>i$, with $(x,i)\in \mathcal{C}_i$ and $(y,j)\in \mathcal{C}_j$,
\item if $(x,k)\not\leq (\phi(i),k)$, then $(x,k)\prec (\phi(i),k)$, for all $(x,k)\in \mathcal{C}_k$ and all $i\leq k$.
\end{itemize}

Such linear extensions exist since the elements $(\phi(i),k)$ for $0\leq i\leq k$ form a chain in $\mathcal{C}_k$. They correspond to the linear extensions that add elements in decreasing order of component number, and avoid adding each element of the form $(\phi(i),k)$ as long as possible.
See \cref{fig:posetfirstExample} for an example of the numbering $\alpha_{\mathcal{L}}$ of $\mathcal{C}_P^{\phi}$ and \cref{fig:latticefirstExample} for the associated (blue) edge-labelling $\omega_{\mathcal{L}}$.

\begin{proposition}
\label{prop:gamma}
In \cref{deflabellingpphi}, if we set $\psi$ to be the chain $\psi_{\mathcal{L}} : T(0)\lessdot \cdots \lessdot T(m)$,
we have $\gamma_{1,\psi_{\mathcal{L}}}=\omega_{\mathcal{L}}$ and $\gamma_{2',\psi_{\mathcal{L}}}=\omega_{\mathcal{L}}$.
\end{proposition}

\begin{proof}
The first equality $\gamma_{1,\psi_{\mathcal{L}}}=\omega_{\mathcal{L}}$ follows easily from the definitions of the labellings, since with our conventions $\delta_{\psi_{\mathcal{L}}}(J_{e_i})=i$ for all $i\in [m]$. The second equality requires our particular choice of linear extension $\mathcal{L}$. Recall that $\gamma_{2',\psi_{\mathcal{L}}}(T\lessdot T') = \min\{i\mid T\vee T(i)\geq T'\}$. Since $T\lessdot T'$ and $\max(T'\setminus T)=e_{\omega_{\mathcal{L}}(T\lessdot T')}$, we have $T\vee T(i) \geq T'$ if and only if $e_{\omega_{\mathcal{L}}(T\lessdot T')}\in T\vee T(i)$. 
Thus $\gamma_{2',\psi_{\mathcal{L}}}(T\lessdot T')\leq \omega_{\mathcal{L}}(T\lessdot T')$. Seeking a contradiction, suppose that there exists $j<\omega_{\mathcal{L}}(T\lessdot T')$ such that $\gamma_{2',\psi_{\mathcal{L}}}(T\lessdot T')=j$.
Then $T\vee T(j) \geq T'$, thus $e_{\omega_{\mathcal{L}}(T\lessdot T')}\in T\vee T(j)$. By definition $e_{\omega_{\mathcal{L}}(T\lessdot T')}\not\in T$, and since $j<\omega_{\mathcal{L}}(T\lessdot T')$, we have $e_{\omega_{\mathcal{L}}(T\lessdot T')}\not\in T(j)$. Then $e_{\omega_{\mathcal{L}}(T\lessdot T')}\not\in T\cup T(j)$.
But $e_{\omega_{\mathcal{L}}(T\lessdot T')}\in  T\vee T(j)$, which proves that $e_{\omega_{\mathcal{L}}(T\lessdot T')}$ is obtained with a completion.
By minimality of $j$, we have $T\vee T(j-1)\not\geq T'$, thus the completion that makes $e_{\omega_{\mathcal{L}}(T\lessdot T')}\in T\vee T(j)$ comes from the addition of $e_j$ to $T(j-1)$. Since by hypothesis $j\neq \omega_{\mathcal{L}}(T\lessdot T')$, thus there exists $p\leq k$ such that $e_j=(\phi(p),k)$, where $T'\setminus T \subseteq \mathcal{C}_k$. Moreover, $e_j$ and $e_{\omega_{\mathcal{L}}(T\lessdot T')}$ are incomparable.
By definition of our particular choice of linear extension $\mathcal{L}$, then $e_{\omega_{\mathcal{L}}(T\lessdot T')} \prec e_j$, which means $\omega_{\mathcal{L}}(T\lessdot T') < j$. This is absurd. 
\end{proof}

\begin{remark}
Our particular choice of linear extension $\mathcal{L}$ is important. For example, in \cref{fig:firstExample} let us take the linear extension $\mathcal{L}'$ that switches the numbers $2$ and $3$ in our numbering of $\mathcal{C}_P^{\phi}$ in \cref{fig:posetfirstExample}. Then let $T=\{1,7\}$ and $T'=\{1,3,7\}$. We have $\omega_{\mathcal{L}'}(T\lessdot T')=\gamma_{1,\psi_{\mathcal{L}'}}(T\lessdot T')=3$, but $\gamma_{2',\psi_{\mathcal{L}'}}(T\lessdot T')=2$.
\end{remark}

\begin{theorem}
\label{thmleftmodular}
$\Tam(P,\phi)$ is left modular. The edge-labelling $\omega_{\mathcal{L}}$ is an $EL$-labelling.
\end{theorem}

\begin{proof}
From \cref{prop:leftmodularwithlabellings} and \cref{prop:gamma} follows the fact that $\Tam(P,\phi)$ is left modular. Then $\omega_{\mathcal{L}}$ is an EL-labelling by \cref{prop:ELlabellingLiu}. 
\end{proof}

Using $\omega_{\mathcal{L}}$ we obtain the following topological result about $\Tam(P,\phi)$.

\begin{proposition} \label{prop:onedecreasingchain}
The EL-labelling $\omega_{\mathcal{L}}$ is such that in any interval $[T,T']$ of $\Tam(P,\phi)$ there is at most one maximal decreasing chain. Thus $\mu(T,T') \in \{-1,0,1\}$ and if $]T,T'[$ is nonempty, then its order complex has the homotopy type of either a sphere or a point.    
\end{proposition}

\begin{proof}
Let $[T,T']$ be an interval of $\Tam(P,\phi)$, and let $I=\{i_1,\,i_2,\dots,\,i_k\}$ be the set of indices \mbox{$0\leq i_1<\cdots <i_k<n$} of the nonempty components of $T'\setminus T$. Let us prove that we have at most one maximal decreasing chain in $[T,T']$ with respect to $\omega_{\mathcal{L}}$. By our choice of linear extension $\mathcal{L}$, if $e_k\in\mathcal{C}_i$ and $e_l\in\mathcal{C}_j$ with $i<j$, thus $l<k$. Then, since by \cref{lemCovers} the cover relations happen between two torclosed sets that differ in only one component, any maximal decreasing chain of $[T,T']$ will contain the subchain $T < T\cup T'_{i_1} < T\cup T'_{i_1} \cup T'_{i_2} <\cdots < T\cup T'_{i_1} \cup \cdots \cup T'_{i_k} =T'$. Seeking a contradiction, suppose that we have two different maximal decreasing chains $C_1$ and $C_2$ of $[T,T']$. As both chains contain the subchain described above, and two consecutive elements of this subchain differ in only one component, we look to the component $\mathcal{C}_{i_j}$ with the least index where $C_1$ and $C_2$ differ. Suppose that the first time these chains differ, following them from bottom to top, is witnessed by $R\lessdot S$ in $C_1$ and $R\lessdot U$ in $C_2$. Since $T\cup T_{i_1}'\cup \cdots \cup T_{i_j}'$ is torclosed and greater than both $R$ and $S$, the elements in components $\mathcal{C}_p$ for $p>i$ have no effect on which subsets of elements from $\mathcal{C}_i$ can be added to $R$ and $S$.
This just means that we add all the elements of $T_{i_j}' \setminus R$ in $C_1$ and $C_2$ as if we did not have elements in components with bigger indices.  By \cref{lemCovers}, if $V\lessdot W$ is a cover relation in $\Tam(P,\phi)$, then $W\setminus V$ has a maximum element $e_l$, and $\omega_{\mathcal{L}}(V\lessdot W)=l$ by definition of $\omega_{\mathcal{L}}$.
Without loss of generality suppose that $\omega_{\mathcal{L}}(R\lessdot S)<\omega_{\mathcal{L}}(R\lessdot U)$.
Let $V\lessdot W$ be the cover relation of $C_1$ such that $e_{\omega_{\mathcal{L}}(R\lessdot U)} \in W\setminus V$. Since $e_{\omega_{\mathcal{L}}(R\lessdot U)}$ is a maximal element of $T_{i_j}'\setminus R$, then $\omega_{\mathcal{L}}(V\lessdot W)=\omega_{\mathcal{L}}(R\lessdot U)$. Thus the chain $C_1$ is not decreasing. This is absurd.

Then the last sentence of the proposition follows by results recalled in \cref{sec:posetsandlattices} in the paragraph about poset topology.
\end{proof}

We now turn our attention to the congruences of $\Tam(P,\phi)$. The next two results characterize the join dependency relation $D$. Recall that we denote by $J_{(x,i)}$ the minimal torclosed set containing $(x,i)$, and that \cref{prop:joinirr} tells us that the join-irreducibles of $\Tam(P,\phi)$ are the $J_{(x,i)}$ for $(x,i)\in \mathcal{C}_P^{\phi}$.

\begin{lemma}
\label{lem:Drelationinonecompo}
Let $i<n$. If $J_{(x,i)}\,D\,J_{(y,i)}$, then there exists $k\leq i$ such that $y=\phi(k)$ and $y\not \leq x$. If $P$ has a minimum $\hat{0}$ and $\phi(0)=\hat{0}$, then the last statement is an equivalence.  
\end{lemma}

\begin{proof}
We suppose $J_{(x,i)}\,D\,J_{(y,i)}$, which means that $x\neq y$ and there exists $p\in \Tam(P,\phi)$ such that $J_{(x,i)}\leq J_{(y,i)}\vee p$ and $J_{(x,i)}\not\leq (J_{(y,i)})_{*}\vee p$. Then $J_{(x,i)}\not\leq J_{(y,i)}$, as noted at the end of \cref{sec:posetsandlattices}. This means $(y,i)\not\leq (x,i)$, thus $y\not\leq x$. It remains to prove that there exists $k\leq i$ such that $y=\phi(k)$. Seeking a contradiction, suppose that it is not the case. Then $J_{(y,i)}$ and $(J_{(y,i)})_{*}=J_{(y,i)}\setminus \{(y,i)\}$ contain the same elements of the form $(\phi(k),i)$ for $k\leq i$. Thus $(\,(J_{(y,i)})_{*}\vee p)_i = (\,J_{(y,i)}\vee p )_i$ or $(\,(J_{(y,i)})_{*}\vee p)_i = (J_{(y,i)}\vee p )_i \setminus \{(y,i)\}$. As $J_{(x,i)}=(J_{(x,i)})_i$, we have $(J_{(x,i)})_i\leq (J_{(y,i)}\vee p)_i$ and $(J_{(x,i)})_i\not\leq (\,(J_{(y,i)})_{*}\vee p)_i$. In particular, $(\,(J_{(y,i)})_{*}\vee p)_i \neq (J_{(y,i)}\vee p )_i$, thus $(\,(J_{(y,i)})_{*}\vee p)_i = (J_{(y,i)}\vee p )_i \setminus \{(y,i)\}$. Then, we have $J_{(x,i)}\leq (J_{(y,i)}\vee p)_i$ and $J_{(x,i)}\not\leq (\,(J_{(y,i)})_{*}\vee p)_i= (J_{(y,i)}\vee p )_i \setminus \{(y,i)\}$. This means that $(x,i)=(y,i)$, which is absurd since $y\not\leq x$.

Conversely, suppose that $P$ has a minimum $\hat{0}$, that $\phi(0)=\hat{0}$, and let $(x,i)$ and $(y,i)$ be in $\mathcal{C}_i$ such that there exists $k\leq i$ such that $y=\phi(k)$ and $y\not \leq x$. Since $y\not\leq x$, then $y\neq \phi(0)=\hat{0}$. This proves that $k\geq 1$. Let $p=\mathcal{C}_{k-1} \in \Tam(P,\phi)$. We have $J_{(x,i)} \subseteq  \mathcal{C}_i \cup \mathcal{C}_{k-1} = J_{(\phi(k),i)} \vee p$ since this latter torclosed set contains $(\hat{0},k-1)$ and $(\phi(k),i)$, thus by completion contains $(\hat{0},i)=\min(\mathcal{C}_i)$. But $(J_{(\phi(k),i)})_{*} \vee p = (J_{(\phi(k),i)})_{*} \cup \mathcal{C}_{k-1}$ since $(\phi(k),i)\not\in (J_{(\phi(k),i)})_{*}$. Then $J_{(x,i)}\not\leq (J_{(\phi(k),i)})_{*} \vee p$. This proves that $J_{(x,i)}\,D\,J_{(y,i)}$.
\end{proof}

\begin{lemma}
\label{lem:Drelatinindiffcompo} 
For all $i\neq j$, we have $J_{(x,j)}\,D\,J_{(y,i)}$ if and only if $i<j$, $y\leq x$, $ \phi(i+1) \not \leq x$ and there is no $y'\leq \phi(i)$ such that $y<y'\leq x$.
\end{lemma}

\begin{proof}
Let $(x,j)$ and $(y,i)$ be in $\mathcal{C}_P^{\phi}$ such that $i<j$, $y\leq x$, $ \phi(i+1) \not \leq x$ and there is no $y'\leq \phi(i)$ such that $y<y'\leq x$. Let $p=J_{(\phi(i+1),j)}$. We have both $(y,i)$ and $(\phi(i+1),j)$ in the torclosed set $J_{(y,i)}\vee p$, thus by completion $(y,j)\in J_{(y,i)}\vee p$. Since $(x,j)\geq (y,j)$ and they are in the same component $\mathcal{C}_j$, then $(x,j)\in J_{(y,i)}\vee p$. Thus $J_{(x,j)}\leq J_{(y,i)}\vee p$.

If $y=\phi(i)$, then $(J_{(y,i)})_{*}=\emptyset$. And since $\phi(i+1)\not\leq x$, then $(x,j)\not\in J_{(\phi(i+1),j)} =(\,J_{(y,i)})_{*}\vee J_{(\phi(i+1),j)}$. This proves $J_{(x,j)}\,D\,J_{(y,i)}$. Otherwise $y\neq \phi(i)$. Then $(\,(J_{(y,i)})_{*}\vee J_{(\phi(i+1),j)})_j =\{(z,j)\mid \exists y'\in P,\, y< y'\leq \phi(i) \text{ and }y'\leq z\leq \phi(j)\}$. Then $(x,j)$ is not in this latter set since there is no $y'\leq \phi(i)$ such that $y<y'\leq x$. Thus $J_{(x,j)}\not\leq (\,(J_{(y,i)})_{*}\vee J_{(\phi(i+1),j)})_j$, which proves that $J_{(x,j)}\,D\,J_{(y,i)}$.

Conversely, suppose that we have $J_{(x,j)}\,D\,J_{(y,i)}$, which means that there exists $p\in \Tam(P,\phi)$ such that $J_{(x,j)}\leq J_{(y,i)}\vee p$ and $J_{(x,j)}\not\leq (J_{(y,i)})_{*}\vee p$. If $j<i$, then $J_{(y,i)}$ does not contribute in $J_{(x,j)}\leq J_{(y,i)}\vee p$, which means that $J_{(x,j)}\leq p$. Then $J_{(x,j)}\leq (J_{(y,i)})_{*}\vee p$, which is absurd. This proves that $i<j$. As observed, we do not have $J_{(x,j)}\leq p$, thus $x\not\in p$. Since $J_{(x,j)}\leq J_{(y,i)}\vee p$, this means that $J_{(y,i)}$ is essential when computing the join to contain $J_{(x,j)}$, even though it is not in the same component. Thus we have completions involving $J_{(y,i)}$, which proves both $(\phi(i+1),j)\in p$ and $(y,j)\leq (x,i)$. The latter gives $y\leq x$, while the former gives $\phi(i+1)\not\leq x$, because if it was not the case then $(x,j)\in p$ since $p$ is torclosed, which is absurd. It remains to prove that there is no $y'\leq \phi(i)$ such that $y<y'\leq x$. Seeking a contradiction, suppose that such a $y'$ exists. Then $(y',i)\in (J_{(y,i)})_{*}$, which with $(\phi(i+1),j)\in p$ gives by completion $(y',j)\in (J_{(y,i)})_{*}\vee p$. Since $y'\leq x$, this gives $(x,i)\in (J_{(y,i)})_{*}\vee p$. Thus $J_{(x,j)}\leq (J_{(y,i)})_{*}\vee p$, which is absurd.
\end{proof}

\begin{example}
In \cref{fig:posetfirstExample}, the $D$ relations between join-irreducibles generated by elements in the same component are $J_6\,D\,J_5$, $J_4\,D\,J_3$, $J_4\,D\,J_1$, $J_3\,D\,J_1$, $J_2\,D\,J_3$, $J_2\,D\,J_1$ and the other relations are $J_6\,D\,J_7$, $J_4\,D\,J_7$, $J_2\,D\,J_7$, $J_4\,D\,J_6$, $J_3\,D\,J_5$, $J_2\,D\,J_5$.
\end{example}

The following theorem gives another proof that $\Tam(P,\phi)$ is join-semidistributive, since a join-congruence uniform lattice is join-semidistributive.

\begin{theorem}
\label{thmLW}
$\Tam(P,\phi)$ is join-congruence uniform.
\end{theorem}

\begin{proof}
By \cref{propDayLW}, we need to prove that the $D$ relation is acyclic. This follows from \cref{lem:Drelationinonecompo} and \cref{lem:Drelatinindiffcompo}.
\end{proof}

\begin{remark}
The lattice $\Tam(P,\phi)$ is not polygonal in general (see \cref{fig:figsquarephi2}).
\end{remark}

\begin{definition}
Let $K:=(K_0,K_1,\dots,K_n)$ be a non-decreasing sequence of non-negative integers such that $K_i\leq i$ for all $i\leq n$. We call such a sequence a \emph{Kupisch-like series}. Denote \mbox{$R_K :=\{(x,i)\in \mathcal{C}_P^{\phi}\mid x\geq \phi(K_i)\}$}. On $\Tam(P,\phi)$ we define the \emph{Kupisch equivalence relation} $T\equiv_K T'$ if and only if $T\cap R_K = T'\cap R_K$.
\end{definition}

\begin{proposition}
\label{propKupisch}
Suppose that $P$ is a join-semilattice. Then a Kupisch equivalence relation is a lattice congruence of $\Tam(P,\phi)$.
\end{proposition}

\begin{proof}
Let $K$ be a Kupisch-like series. Since $\Tam(P,\phi)$ is join-congruence uniform (\cref{thmLW}), by \cref{propDayLW} it suffices to prove that the set $S$ of join-irreducibles that $\equiv_K$ contracts satisfies that $J_{(x,j)} D J_{(y,i)}$ and $J_{(y,i)}\in S$ together imply $J_{(x,j)}\in S$. We have
$$J_{(y,i)}\in S \iff J_{(y,i)} \equiv_K (J_{(y,i)})_{*} \iff (y,i)\not\in R_K \iff y\not\geq \phi(K_i) .$$

\noindent Suppose that $J_{(x,j)} D J_{(y,i)}$ and $y\not\geq \phi(K_i)$. We need to prove that $x\not \geq \phi(K_j)$.  We consider two cases; first $j=i$, then $j\neq i$.

Suppose $j=i$. By \cref{lem:Drelationinonecompo}, $J_{(x,i)} D J_{(y,i)}$ implies that there exists $k\leq i$ such that $y=\phi(k)$ and $x\not\geq y$. Since $\phi(k)=y\not\geq \phi(K_i)$, and $\phi(k)$ is comparable to $\phi(K_i)$, then $\phi(K_i)> \phi(k)=y$. This proves that $x\not \geq \phi(K_j)$, because if it was not the case then $x\geq \phi(K_j)>y$, which is absurd since $x\not\geq y$.

Suppose $j\neq i$. By \cref{lem:Drelatinindiffcompo}, $J_{(x,j)} D J_{(y,i)}$ implies that $i<j$, $y\leq x$, $ \phi(i+1) \not \leq x$ and there is no $y'\leq \phi(i)$ such that $y<y'\leq x$. Seeking a contradiction, assume that $x \geq \phi(K_j)$. Since $x\not\geq \phi(i+1)$ and $\phi(i+1)$ is comparable to $\phi(K_j)$, then $\phi(K_j)< \phi(i+1)$. Thus $\phi(K_j)\leq \phi(i)$. Since $y\not\geq \phi(K_i)$, either $y<\phi(K_i)$, or $\phi(K_i)$ and $y$ are incomparable. 

If $y<\phi(K_i)$, then $y<\phi(K_i)\leq \phi(K_j)$ since $K$ is non-decreasing. Thus $y<\phi(K_j)\leq x$, and we proved above that $\phi(K_j)\leq \phi(i)$. This is absurd since we know that no such element $y':=\phi(K_j)$ exists.

Otherwise $\phi(K_i)$ and $y$ are incomparable. Since $P$ is a join-semilattice, we can define in $P$ the element $y':=\phi(K_j) \vee y$. We proved that $\phi(K_j)\leq \phi(i)$, and by definition $y\leq \phi(i)$. Then $y'\leq \phi(i)$. By hypothesis, we have $\phi(K_j)\leq x$ and $y\leq x$, thus $y'\leq x$. To have something absurd, it suffices to prove that $y<y'$. By definition of $y'$, we have $y\leq y'$. If $y=y'$, then $\phi(K_j) \leq y$, and since $\phi(K_i)\leq \phi(K_j)$ because $K$ is non-decreasing, it would give $\phi(K_i)\leq y$, which is absurd. Thus $y<y'$. This finishes the proof of the proposition.
\end{proof}

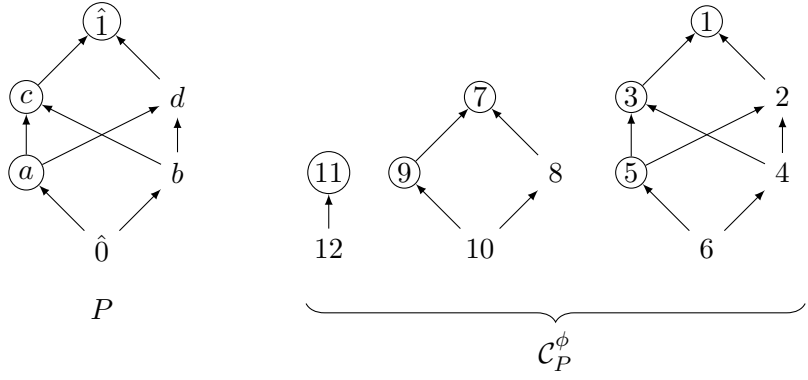
\begin{figure}
    \centering
\begin{tikzpicture}
\begin{scope}
\node (0) at (0,0) {$\hat{0}$};
\node[draw,circle,inner sep=2pt] (a) at (-1,1) {$a$};
\node (b) at (1,1) {$b$};
\node[draw,circle,inner sep=2pt] (c) at (-1,2) {$c$};
\node (d) at (1,2) {$d$};
\node[draw,circle,inner sep=1pt] (1) at (0,3) {$\hat{1}$};
\draw[->,>=latex] (0) -- (a);
\draw[->,>=latex] (0) -- (b);
\draw[->,>=latex] (a) -- (c);
\draw[->,>=latex] (a) -- (d);
\draw[->,>=latex] (b) -- (c);
\draw[->,>=latex] (b) -- (d);
\draw[->,>=latex] (c) -- (1);
\draw[->,>=latex] (d) -- (1);
\draw (0,-0.8) node {$P$};
\end{scope}

\begin{scope}[xshift= 2cm, scale= 1]
\node (12) at (1,0) {$12$};
\node[draw,circle,inner sep=1pt] (11) at (1,1) {$11$};
\node (10) at (3,0) {$10$};
\node[draw,circle,inner sep=1pt] (9) at (2,1) {$9$};
\node (8) at (4,1) {$8$};
\node[draw,circle,inner sep=1pt] (7) at (3,2) {$7$};
\node[draw,circle,inner sep=1pt] (5) at (5,1) {$5$};
\node (6) at (6,0) {$6$};
\node (4) at (7,1) {$4$};
\node (2) at (7,2) {$2$};
\node[draw,circle,inner sep=1pt] (3) at (5,2) {$3$};
\node[draw,circle,inner sep=1pt] (1) at (6,3) {$1$};
\draw[->,>=latex] (12) -- (11);
\draw[->,>=latex] (10) -- (9);
\draw[->,>=latex] (10) -- (8);
\draw[->,>=latex] (9) -- (7);
\draw[->,>=latex] (8) -- (7);
\draw[->,>=latex] (6) -- (5);
\draw[->,>=latex] (6) -- (4);
\draw[->,>=latex] (5) -- (2);
\draw[->,>=latex] (5) -- (3);
\draw[->,>=latex] (4) -- (3);
\draw[->,>=latex] (4) -- (2);
\draw[->,>=latex] (3) -- (1);
\draw[->,>=latex] (2) -- (1);
\draw [decorate,decoration={brace,amplitude=7pt, mirror,raise=2ex}]
(0.7,-0.4) -- (7.3,-0.4) node[midway,yshift=-2.5em]{$\mathcal{C}_P^{\phi}$};
\end{scope}
\end{tikzpicture}
    \caption{On the left is a poset $P$ with a chain $\phi$ whose elements are circled. On the right is $\mathcal{C}_P^{\phi}$, with the numbering from \cref{sectionCongruenceTopo}.}
    \label{fig:counterexamplecongruence}
\end{figure}

\begin{remark}
The hypothesis that $P$ is a join-semilattice in \cref{propKupisch} is important. For example, the choice of $P$ and $\phi$ in \cref{fig:counterexamplecongruence} produces a lattice $\Tam(P,\phi)$ of $50$ elements in which there are Kupisch equivalence relations that are not lattice congruences. Indeed, let $K=(0,0,0)$ and $\equiv_K$ be the associated Kupisch equivalence relation on $\Tam(P,\phi)$. Then $J_2 \,D\,J_8$ and $J_8 \equiv_K (J_8)_{*}$, but $J_2 \not\equiv_K (J_2)_{*}$.  
\end{remark}

\section{Main examples}
\label{sectionMainExamples}

\subsection{The case of the chains}
\label{sec:Chain}

In this section, we study the particular case where $P$ is a chain. We suppose that $P=C_m$ where \mbox{$C_m: 0<1<\cdots <m-1$}. Since $\phi$ is a subchain of $C_m$, it can be viewed as a nonempty subset of $\{0,1,\dots,m-1\}$. If $\phi=\{0,1,\dots,m-1\}$, we will write $\phi=C_m$.
The components $\mathcal{C}_i$ for $i<n$ are themselves chains, isomorphic to $C_{\phi(i)}$, and having an order filter of $\mathcal{C}_{C_m}^{\phi}$ just amounts to knowing how many elements we have in each component $\mathcal{C}_i$ for $i<n$. Thus, these order filters $T$ are identified with words on nonnegative integers $u=u_0 u_1\dots u_{n-1}$ where for all $i <n$, $u_i:=|\mathcal{C}_{i} \cap T|$. The \emph{length} of a word is its number of letters.

\begin{proposition}
  $\Tam(C_m,\phi)$ is a congruence uniform, extremal and left modular lattice.  
\end{proposition}

\begin{proof}
The left modular property follows from \cref{thmleftmodular}.
Using \cref{propJSD}, it follows that $\Tam(C_m,\phi)$ is semidistributive, and $\Tam(C_m,\phi)$ is join-congruence uniform by \cref{thmLW}. Thus $\Tam(C_m,\phi)$ is congruence uniform since a semidistributive join-congruence uniform lattice is congruence uniform (this is an easy consequence of \cite[Lemma 1]{CASPARD200471}). By \cref{coro:joinextremalTam}, $\Tam(C_m,\phi)$ is join-extremal, but since it is semidistributive then it is extremal.
\end{proof}

\begin{figure}
\begin{subfigure}{.5\textwidth}
\centering
\begin{tikzpicture}[scale=0.8]
\draw[step=1cm, gray, very thin, dashed] (-1,-1) grid (6, 6);
\draw[->,>=latex] (-1,0)--(5,0);
\draw[very thick] (0,0) -- (0,2);
\draw[very thick] (2,0) -- (2,6);
\draw[very thick] (4,0) -- (4,2);
\draw[very thick] (5,0) -- (5,5);
\draw (0,0) node {$\bullet$};
\draw (1,0) node {$\bullet$};
\draw (2,0) node {$\bullet$};
\draw (3,0) node {$\bullet$};
\draw (4,0) node {$\bullet$};
\draw (5,0) node {$\bullet$};
\draw (0,0) node[below] {$1$};
\draw (1,0) node[below] {$0$};
\draw (2,0) node[below] {$5$};
\draw (3,0) node[below] {$0$};
\draw (4,0) node[below] {$1$};
\draw (5,0) node[below] {$4$};
\draw[red] (2,6) -- (-1,0);
\draw[red] (4,2) -- (3,0);
\draw[red] (5,5) -- (2.5,0);
\end{tikzpicture}
    \caption{The word $u=105014$ which represents a torclosed set of $\Tam(C_{12},\phi_2)$. The lines of slope $2$ from the top of each segment are represented in red.}
    \label{fig:chain4}
\end{subfigure}%
\begin{subfigure}{.5\textwidth}
\centering
\begin{tikzpicture}
\begin{scope}
\node (a) at (0,0) {$0$};
\node[draw,circle,inner sep=1pt] (b) at (0,1) {$1$};
\node (c) at (0,2) {$2$};
\node[draw,circle,inner sep=1pt] (d) at (0,3) {$3$};
\draw[->,>=latex] (a) -- (b);
\draw[->,>=latex] (b) -- (c);
\draw[->,>=latex] (c) -- (d);
\draw (0,-0.8) node {$C_4$};
\end{scope}

\begin{scope}[xshift=1.3cm,xscale=0.7]
\node[scale=0.8] (6) at (0,0) {$6$};
\node[scale=0.8] (5) at (0,1) {$5$};
\node[scale=0.8] (4) at (1,0) {$4$};
\node[scale=0.8] (3) at (1,1) {$3$};
\node[scale=0.8] (2) at (1,2) {$2$};
\node[scale=0.8] (1) at (1,3) {$1$};
\draw[->,>=latex] (6) -- (5);
\draw[->,>=latex] (4) -- (3);
\draw[->,>=latex] (3) -- (2);
\draw[->,>=latex] (2) -- (1);
\draw (0,-0.8) node {$\mathcal{C}_0$};
\draw (1,-0.8) node {$\mathcal{C}_1$};
\draw [decorate,decoration={brace,amplitude=7pt, mirror,raise=2ex}]
(-0.3,-0.9) -- (1.3,-0.9) node[midway,yshift=-2.5em]{$\mathcal{C}_{C_4}^{\phi}$};
\end{scope}

\begin{scope}[xshift=4.5cm,yshift=-1.5cm,xscale=1.2,yscale=1.7]
\node (0) at (0,0) {$00$};
\node (1) at (-1,1) {$10$};
\node (2) at (1,0.5) {$01$};
\node (3) at (-1,2) {$20$};
\node (4) at (1,1) {$02$};
\node (5) at (1,1.5) {$03$};
\node (6) at (0,2) {$13$};
\node (7) at (1,2.2) {$04$};
\node (8) at (1,3) {$14$};
\node (9) at (0,3.5) {$24$};
\draw[->,>=latex] (0) -- (1);
\draw[->,>=latex] (0) -- (2);
\draw[->,>=latex] (1) -- (3);
\draw[->,>=latex] (1) -- (6);
\draw[->,>=latex] (2) -- (4);
\draw[->,>=latex] (3) -- (9);
\draw[->,>=latex] (4) -- (5);
\draw[->,>=latex] (5) -- (6);
\draw[->,>=latex] (5) -- (7);
\draw[->,>=latex] (6) -- (8);
\draw[->,>=latex] (7) -- (8);
\draw[->,>=latex] (8) -- (9);
\end{scope}
\end{tikzpicture}
\caption{$\Tam(C_4,\phi_2)$ with elements the words $u$.}
\label{fig:combegiraudo}
\end{subfigure}
\caption{}
\label{fig:examplechain}
\end{figure}

\begin{proposition}
\label{lemconditionchaintorclosed}
The word $u$ corresponds to an order filter of $\mathcal{C}_{C_m}^{\phi}$ if and only if for all $i\in [n]$, $u_i\leq \phi(i)+1$. Moreover $u$ corresponds to a torclosed set if and only if for all $i< n$ and $1\leq k< n-i$, having both $u_{i+k} > \phi(i+k) - \phi(i+1)$ and $u_i\neq 0$ implies $u_{i+k} \geq u_i +\phi(i+k)-\phi(i)$.
\end{proposition}

\begin{proof}
The word $u$ corresponds to an order filter of $\mathcal{C}_{C_m}^{\phi}$ if and only if for all $i\in[n]$, $u_i\leq |\mathcal{C}_{i}|=\phi(i)+1$. Denote $T$ the order filter associated to $u$. It remains to understand on the word $u$ the completion operation. Let $i< n$ and $1\leq k< n-i$. We look at the completions between components $\mathcal{C}_i$ and $\mathcal{C}_{i+k}$. The completions between these two components happen if and only if $T_i\neq \emptyset$ and $(\phi(i+1),i+k) \in T_{i+k}$, and in this case we need $(x,i)\in T_i$ implies $(x,i+k)\in T_{i+k}$. We have $T_i\neq \emptyset \iff u_i\neq 0$ and $(\phi(i+1),i+k) \in T_{i+k} \iff u_{i+k} >\phi(i+k)-\phi(i+1)$. Moreover, $(x,i)\in T_i \Longrightarrow (x,i+k)\in T_{i+k}$ means $u_{i+k}-u_i\geq \phi(i+k)-\phi(i)$. This finishes the proof of the proposition.
\end{proof}

Let $p$ be a positive integer.
Before giving a general recursive formula to compute $|\Tam(C_m,\phi)|$ (\cref{prop:inductionformula}), we explore in some details the case $\Tam(C_{np},\phi_p)$ where by definition $\phi_p(i)=(i+1)\,p -1$ for all $i<n$.
For every $i< n$ such that $u_i\neq 0$, draw a segment from $(i,0)$ to $(i,u_i+p-1)$ in the Cartesian plane, where these points are respectively called the base and the top of the segment. Then \cref{lemconditionchaintorclosed} says that one can draw lines of slope $p$ passing through the $x$-axis and the top of each segment without crossing any segment nor touching the base of a segment (it can touch the top of a segment). These drawings are called \emph{diagrams}. See \cref{fig:chain4} for an example of a diagram, and see \cref{fig:combegiraudo} for the lattice $\Tam(C_4,\phi_2)$ whose elements are represented as their associated words $u$. When $u_i=0$, we have what we call an empty segment, which is represented as a point corresponding to the base of this empty segment.
We call a word $u$ which represents a torclosed set of $\Tam(C_n,C_n)$ a \emph{Tamari word} of length $m$. 
For a Tamari word $u$ of length $n$, let $B_p(u)$ be the word of length $m$ whose $i$-th letter is $p\, u_i-(p-1)$ for all $i<n$, and let $E_p(u)$ be the word of length $n$ whose $i$-th letter is $p\, u_i$ for all $i<n$.

\begin{lemma}
\label{lem:intervalBE}
Let $u$ be a Tamari word of length $n$. Then both $B_p(u)$ and $E_p(u)$ correspond to elements of $\Tam(C_{np},\phi_p)$. Moreover, the intervals $I_p(u):=[B_p(u),E_p(u)]$ for all Tamari words $u$ of length $n$ form a partition of $\Tam(C_{np},\phi_p)$. 
\end{lemma}

\begin{proof}
The fact that $B_p(u)$ and $E_p(u)$ are both elements of $\Tam(C_{np},\phi_p)$ follows easily from the diagrams using the fact that $u_j-u_i\geq j-i$ for all $i<j$ since $u$ is a Tamari word. This can also be seen directly on the chains forming $\mathcal{C}_{C_{np}}^{\phi}$ (see the right of \cref{fig:proofpropenumerationchain}). From this latter point of view follows the result about the partition of $\Tam(C_{np},\phi_p)$, since for a torclosed set $T$ of $\Tam(C_{np},\phi_p)$ the unique $u$ such that $T\in I_p(u)$ is the word whose $i$-th letter is the number of elements of $(\phi(j),i)$ in $T$ for $j\leq i$.
\end{proof}

The leaves of a (complete) binary tree $T$ are numbered by a reverse preorder traversal of $T$, which means visit the root, then traverse its subtrees from right to left. Given a binary tree $T$, each leaf is the leftmost leaf of some binary subtree of $T$. For the leaf numbered $i$, let $w_i$ be the number of leaves minus one of the largest binary subtree of $T$ whose leftmost leaf is the one numbered $i$. We call $w_i$ the weight of the leaf $i$, and the sequence $(w_1,w_2,\dots,w_k)$ of all weights is called the weight sequence of the binary tree.
Let $\varphi$ be the map that sends a Tamari word $u$ of length $n$ to the binary tree $T$ whose weight sequence is $(0,u_0,u_1,\dots,u_{n-1},n+1)$. For an example, see the top right of \cref{fig:prooflemdilatation} which represents $\varphi(1014012)$. The extra $n+1$ can be seen on the diagram with the segments by adding a segment of height $n+1$ at the $x$-coordinate $n$. The map $\varphi$ is in fact a bijection (see \cite{10.1093/comjnl/29.2.171}). It follows (this corresponds to the case $p=1$):

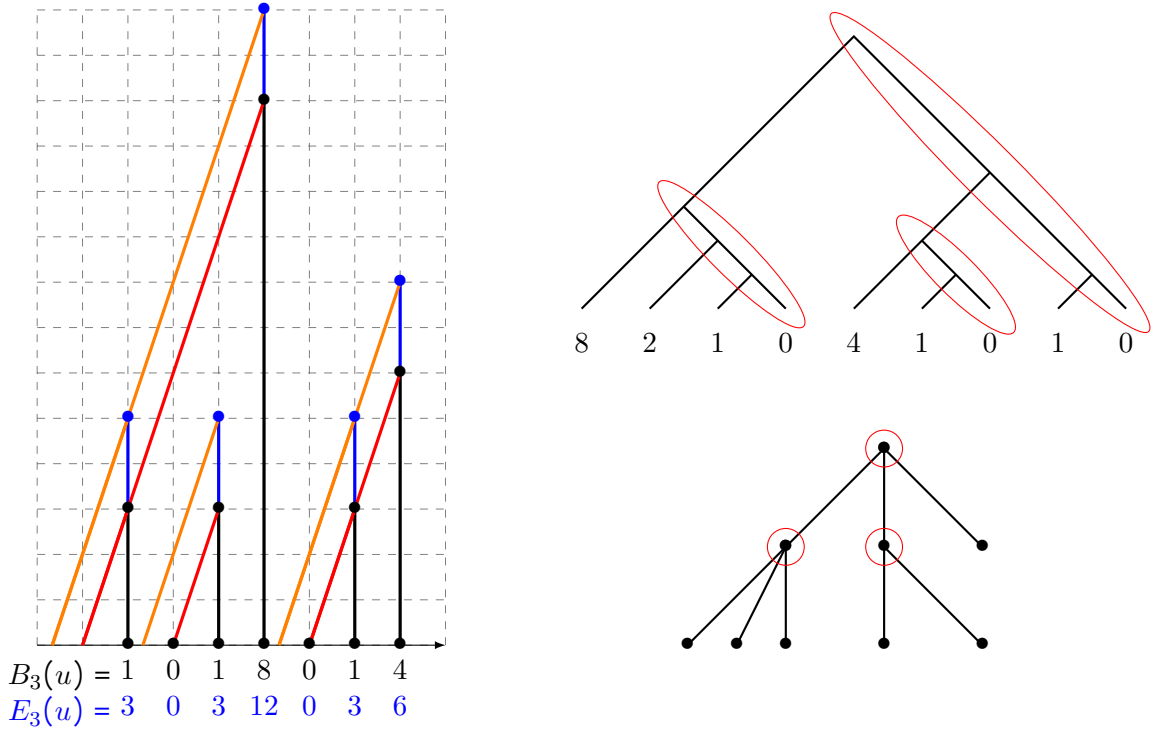
\begin{figure}
    \centering
\begin{tikzpicture}
\begin{scope}[scale=0.6]
\draw[step=1cm, gray, very thin, dashed] (-2,0) grid (7, 14);
\draw[->,>=latex] (-2,0)--(7,0);
\draw[very thick] (0,0) -- (0,3);
\draw[very thick] (2,0) -- (2,3);
\draw[very thick] (3,0) -- (3,12);
\draw[very thick] (5,0) -- (5,3);
\draw[very thick] (6,0) -- (6,6);

\draw[blue,very thick] (0,3) -- (0,5);
\draw[blue,very thick] (2,3) -- (2,5);
\draw[blue,very thick] (3,12) -- (3,14);
\draw[blue,very thick] (5,3) -- (5,5);
\draw[blue,very thick] (6,6) -- (6,8);

\draw[red,very thick] (-1,0) -- (0,3);
\draw[red,very thick] (1,0) -- (2,3);
\draw[red,very thick] (-1,0) -- (3,12);
\draw[red,very thick] (4,0) -- (5,3);
\draw[red,very thick] (4,0) -- (6,6);

\draw[orange,very thick] (-1.667,0) -- (0,5);
\draw[orange,very thick] (0.33,0) -- (2,5);
\draw[orange,very thick] (-1.67,0) -- (3,14);
\draw[orange,very thick] (3.33,0) -- (5,5);
\draw[orange,very thick] (3.33,0) -- (6,8);
\draw (0,0) node {$\bullet$};
\draw (1,0) node {$\bullet$};
\draw (2,0) node {$\bullet$};
\draw (3,0) node {$\bullet$};
\draw (4,0) node {$\bullet$};
\draw (5,0) node {$\bullet$};
\draw (6,0) node {$\bullet$};

\draw (0,3) node {$\bullet$};
\draw (2,3) node {$\bullet$};
\draw (3,12) node {$\bullet$};
\draw (5,3) node {$\bullet$};
\draw (6,6) node {$\bullet$};

\draw (0,5) node[blue] {$\bullet$};
\draw (2,5) node[blue] {$\bullet$};
\draw (3,14) node[blue] {$\bullet$};
\draw (5,5) node[blue] {$\bullet$};
\draw (6,8) node[blue] {$\bullet$};

\draw (-1.5,-0.1) node[below] {$B_3(u)=$};
\draw (0,-0.1) node[below] {$1$};
\draw (1,-0.1) node[below] {$0$};
\draw (2,-0.1) node[below] {$1$};
\draw (3,-0.1) node[below] {$8$};
\draw (4,-0.1) node[below] {$0$};
\draw (5,-0.1) node[below] {$1$};
\draw (6,-0.1) node[below] {$4$};

\draw (-1.5,-0.9) node[blue,below] {$E_3(u)=$};
\draw (0,-0.9) node[blue,below] {$3$};
\draw (1,-0.9) node[blue,below] {$0$};
\draw (2,-0.9) node[blue,below] {$3$};
\draw (3,-0.9) node[blue,below] {$12$};
\draw (4,-0.9) node[blue,below] {$0$};
\draw (5,-0.9) node[blue,below] {$3$};
\draw (6,-0.9) node[blue,below] {$6$};
\end{scope}

\begin{scope}[xshift=6cm, yshift=4cm, scale=0.9]
\draw (0,0) node {$8$};
\draw (1,0) node {$2$};
\draw (2,0) node {$1$};
\draw (3,0) node {$0$};
\draw (4,0) node {$4$};
\draw (5,0) node {$1$};
\draw (6,0) node {$0$}; 
\draw (7,0) node {$1$};
\draw (8,0) node {$0$};  

\draw[thick] (0,0.5) -- (4,4.5) -- (8,0.5);
\draw[thick] (3,0.5) -- (1.5,2);
\draw[thick] (1,0.5)--(2,1.5);
\draw[thick] (2,0.5)--(2.5,1);
\draw[thick] (4,0.5)--(6,2.5);
\draw[thick] (5,1.5)--(6,0.5);
\draw[thick] (5,0.5)--(5.5,1);
\draw[thick] (7,0.5)--(7.5,1);

\draw[red,rotate around={135:(6,2.5)}] (6,2.5) ellipse (3.3cm and 0.4cm);
\draw[red,rotate around={135:(5.5,1)}] (5.5,1) ellipse (1.2cm and 0.3cm);
\draw[red,rotate around={135:(2.2,1.3)}] (2.2,1.3) ellipse (1.5cm and 0.3cm);
\end{scope}

\begin{scope}[xshift=10cm,scale=1.3]
\draw[thick] (-2,0)--(-1,1)--(0,2)--(1,1);
\draw[thick] (-1.5,0)--(-1,1);
\draw[thick] (-1,0)--(-1,1);
\draw[thick] (0,0)--(0,2);
\draw[thick] (1,0)--(0,1);
\draw (-2,0) node {$\bullet$};
\draw (-1,1) node {$\bullet$};
\node[draw=red, circle, inner sep=2pt] at (0,2) {$\bullet$};
\draw (1,1) node {$\bullet$};
\draw (-1.5,0) node {$\bullet$};
\node[draw=red, circle, inner sep=2pt] at (-1,1) {$\bullet$};
\draw (0,0) node {$\bullet$};
\node[draw=red, circle, inner sep=2pt] at (0,1) {$\bullet$};
\draw (0,2) node {$\bullet$};
\draw (1,0) node {$\bullet$};
\draw (-1,0) node {$\bullet$};
\end{scope}
\end{tikzpicture}
    \caption{On the top right is represented the tree $\varphi(u)$ for the Tamari word $u=1014012$. On the left are represented $B_3(u)$ and $E_3(u)$. On the bottom right is the tree $T_u$ of the interval $[B_3(u),E_3(u)]$.}
    \label{fig:prooflemdilatation}
\end{figure}

\begin{proposition} [\cite{10.1093/comjnl/29.2.171}]
\label{propTamari}
$\Tam(C_n,C_n)\cong \Tam_{n+1}$.
\end{proposition}

Thus the Tamari lattice is a very special case of a $(P,\phi)$-Tamari lattice. In addition to being a special case of a $(P,\phi)$-Tamari lattice, we can recover the Tamari lattice in any $(P,\phi)$-Tamari lattice as a sublattice:

\begin{proposition}
\label{propcompareTamari}
Let $P$ be any finite poset and $\phi :\phi(0)<\phi(1)<\cdots <\phi(n-1)$ be any chain in $P$. The induced subposet of $\Tam(P,\phi)$ on the torclosed sets that are order filters generated only by elements of the form $(\phi(i),j)$ for $i\leq j<n$ is a sublattice and is isomorphic to $\Tam_{n+1}$.
\end{proposition}

\begin{proof}
By definition of this subposet, denoted by $P_{\phi}$, and of the completions, it is immediate that $P_{\phi}$ is a sublattice of $\Tam(P,\phi)$. Since, for each $j<n$, the elements $(\phi(i),j)$ for $i\leq j$ form a chain in $\mathcal{C}_j$, an element $T$ of $P_{\phi}$ can be identified as the word $u=u_0 u_1\dots u_{n-1}$ where $u_j$ is the number of elements of the chain $(\phi(j),j)>(\phi(j-1),j)>\cdots >(\phi(0),j)$ in $T$. Then $P_{\phi} \cong \Tam_{n+1}$ follows from \cref{propTamari}.
\end{proof}

\begin{example}
    In \cref{fig:latticefirstExample}, forgetting the torclosed sets $\{1,2\}$, $\{1,2,3\}$, $\{1,2,7\}$ and $\{1,2,3,5\}$ gives the sublattice $\Tam_4$.
\end{example}

We now want to study $\Tam(C_{np},\phi_p)$ when $p\geq 2$. We define a map from binary trees to planar rooted trees, called the \emph{gluing} map. It maps a binary tree to the planar rooted tree obtained by identifying together all the vertices on a right-descendant path starting at any internal node. For example, the binary tree on the top right of \cref{fig:prooflemdilatation} is mapped by the gluing to the planar rooted tree on the bottom right of the same figure (we circled in red the elements that are glued together). Let $\psi$ be the bijection that maps a Tamari word $u$ of length $n$ to the gluing of the tree $\varphi(u)$. For example, the planar rooted tree on the bottom right of \cref{fig:prooflemdilatation} is $\psi(1014012)$. See also on the right of \cref{fig:bijectionmiror}. The following result follows from the definition of $\varphi$.
For an illustration, see the left of \cref{fig:bijectionmiror} with its caption, and compare it to the top right of \cref{fig:prooflemdilatation}.

\begin{lemma} \label{lem:reformvarphi}
The bijection $\varphi$ is equivalently defined to be the map which sends a Tamari word $u$ of length $n$ to the binary tree obtained by taking the mirror image (reflection with respect to the y-axis) of the diagram of $u$ drawn with an extra empty segment indexed by $0$ at the beginning, an extra segment of height $n+1$ at the end, and lines of slope $1$ from the empty segments ($0$ letters) until reaching the last segment of height $n+1$ or before crossing another segment. 
\end{lemma}

It follows from \cref{lem:reformvarphi} (the lines of slope $-1$ correspond to the red lines in \cref{fig:bijectionmiror}):

\begin{proposition} \label{prop:reformpsi}
The bijection $\psi$ is equivalently defined to be the map which sends a Tamari word $u$ of length $n$ to the planar rooted tree obtained by identifying together all the nodes that are on the same line of slope $1$ of $\varphi(u)$ as constructed in \cref{lem:reformvarphi} (or lines of slope $-1$ after the mirror image).
\end{proposition}

\begin{proposition} \label{prop:numberinterval}
Let $u$ be a Tamari word of length $n$. We have
\begin{align}
|I_p(u)| = \prod_{x \text{ an internal node of } \varphi(u)}    \binom{(\text{number of children of x}) -1+(p-1)}{p-1}
\end{align}
\end{proposition}

\begin{proof}
We consider the diagrams of $B_p(u)$ and $E_p(u)$ drawn together as is done on the left of \cref{fig:prooflemdilatation}. Passing from $B_p(u)$ to $E_p(u)$ corresponds to adding $p-1$ to the height of all the nonempty segments (this corresponds to adding the blue segments in \cref{fig:prooflemdilatation}). The elements of $I_p(u)$ correspond to the diagrams with segments whose heights are between these of $B_p(u)$ and these of $E_p(u)$. We can increase freely the height of a given segment of $x$-coordinate $i$ independently of the others before reaching the maximum height $p\, u_i+(p-1)$, unless the top of this segment in $u$ is touched by the line of slope $p$ from the top of another segment (in \cref{fig:prooflemdilatation} the segments which we cannot freely increase the height independently from the other segments are the first and fourth segments). 

If we have $k\geq 1$ segments in $u$ of $x$-coordinate $x_1,x_2,\dots,x_k$ whose tops are on the same line of slope $p$ (and there are no other segments whose top is on that line), then just looking at these segments without changing the others we have as many elements in $I_p(u)$ as the number of sequences $i_1,i_2,\dots,i_k$ satisfying $0\leq i_1\leq i_2 \leq \cdots \leq i_k \leq p-1$. Here $i_j$ corresponds to the number we add to the height of the segment of $x$-coordinate $x_j$. The number of such sequences is $\binom{p-1+k}{k} = \binom{k+(p-1)}{p-1}$. The line of slope $p$ we consider reaches a point on the $x$-axis whose $x$-coordinate is an integer. This point corresponds to a $0$ (empty segment) of the weight sequence of the binary tree $\varphi(u)$ (see the left of \cref{fig:bijectionmiror}). By \cref{prop:reformpsi}, in $\psi(u)$ the top of the segments of $x$-coordinates $x_1,x_2,\dots,x_k$ are identified, thus the resulting internal node that it creates has $k+1$ children. Thus there are as many elements of $I_p(u)$ as number of trees $\psi(u)$ where each internal node is labeled by a positive integer less than or equal to $\binom{(\text{number of children of x})-1+(p-1)}{p-1}$. This finishes the proof of the proposition.
\end{proof}

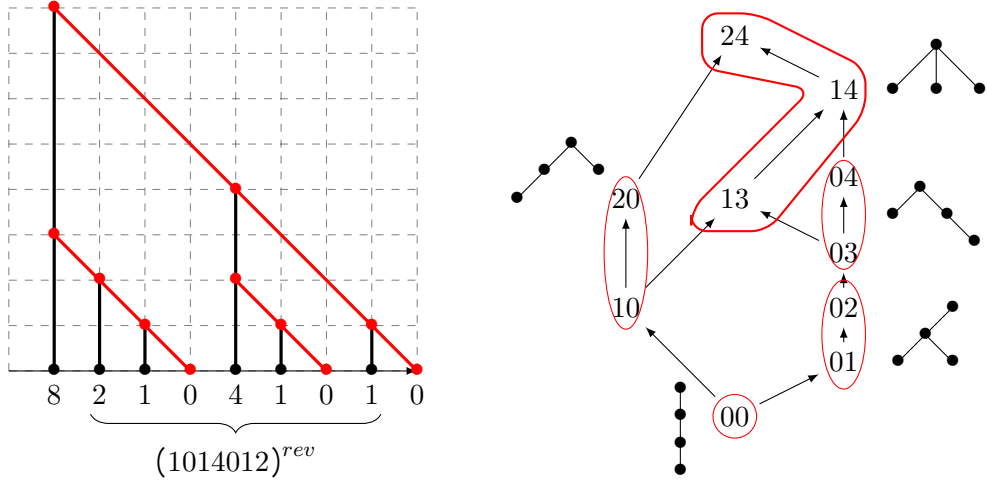
\begin{figure}
    \centering
\begin{tikzpicture}[scale=0.6]
\draw[step=1cm, gray, very thin, dashed] (-1,0) grid (8, 8);
\draw[->,>=latex] (-1,0)--(8,0);
\draw[very thick] (0,0) -- (0,8);
\draw[very thick] (1,0) -- (1,2);
\draw[very thick] (2,0) -- (2,1);
\draw[very thick] (4,0) -- (4,4);
\draw[very thick] (5,0) -- (5,1);
\draw[very thick] (7,0) -- (7,1);

\draw[red,very thick] (0,3) -- (3,0);
\draw[red,very thick] (0,8) -- (8,0);
\draw[red,very thick] (4,2) -- (6,0);

\draw (0,0) node {$\bullet$};
\draw (1,0) node {$\bullet$};
\draw (2,0) node {$\bullet$};
\draw (4,0) node {$\bullet$};
\draw (5,0) node {$\bullet$};
\draw (7,0) node {$\bullet$};

\draw (0,3) node[red] {$\bullet$};
\draw (0,8) node[red] {$\bullet$};
\draw (1,2) node[red] {$\bullet$};
\draw (2,1) node[red] {$\bullet$};
\draw (4,2) node[red] {$\bullet$};
\draw (4,4) node[red] {$\bullet$};
\draw (5,1) node[red] {$\bullet$};
\draw (7,1) node[red] {$\bullet$};
\draw (3,0) node[red] {$\bullet$};
\draw (6,0) node[red] {$\bullet$};
\draw (8,0) node[red] {$\bullet$};

\draw (0,-0.1) node[below] {$8$};
\draw (1,-0.1) node[below] {$2$};
\draw (2,-0.1) node[below] {$1$};
\draw (3,-0.1) node[below] {$0$};
\draw (4,-0.1) node[below] {$4$};
\draw (5,-0.1) node[below] {$1$};
\draw (6,-0.1) node[below] {$0$};
\draw (7,-0.1) node[below] {$1$};
\draw (8,-0.1) node[below] {$0$};   

\draw [decorate,decoration={brace,amplitude=7pt, mirror,raise=2ex}]
(0.8,-0.4) -- (7.2,-0.4) node[midway,yshift=-2.5em]{$(1014012)^{rev}$};

\begin{scope}[xshift=15cm,yshift=-1cm, scale=1.2]
\begin{scope}[scale=2]
\node (0) at (0,0) {$00$};
\node (1) at (-1,1) {$10$};
\node (2) at (1,0.5) {$01$};
\node (3) at (-1,2) {$20$};
\node (4) at (1,1) {$02$};
\node (5) at (1,1.5) {$03$};
\node (6) at (0,2) {$13$};
\node (7) at (1,2.2) {$04$};
\node (8) at (1,3) {$14$};
\node (9) at (0,3.5) {$24$};
\draw[->,>=latex] (0) -- (1);
\draw[->,>=latex] (0) -- (2);
\draw[->,>=latex] (1) -- (3);
\draw[->,>=latex] (1) -- (6);
\draw[->,>=latex] (2) -- (4);
\draw[->,>=latex] (3) -- (9);
\draw[->,>=latex] (4) -- (5);
\draw[->,>=latex] (5) -- (6);
\draw[->,>=latex] (5) -- (7);
\draw[->,>=latex] (6) -- (8);
\draw[->,>=latex] (7) -- (8);
\draw[->,>=latex] (8) -- (9); 

\draw[red] (0,0) ellipse (0.2cm and 0.2cm);
\draw[red] (1,0.75) ellipse (0.2cm and 0.5cm);
\draw[red] (-1,1.5) ellipse (0.2cm and 0.7cm);
\draw[red] (1,1.85) ellipse (0.2cm and 0.5cm);

\draw[red, thick, rounded corners=6pt]
  (-0.4,1.9)--(-0.4,1.7) --(0.3,1.7) --(1.2,2.8) --(1.2,3.2) --
  (0.2,3.7) --(-0.3,3.7) --(-0.3,3.2) --(0.7,3) --
  cycle;
\end{scope}

\begin{scope}[xshift=3.7cm,yshift=6cm,scale=0.8]
\draw (-1,0)--(0,1)--(1,0);
\draw (0,0)--(0,1);
\draw (-1,0) node {$\bullet$};
\draw (0,0) node {$\bullet$};
\draw (1,0) node {$\bullet$};
\draw (0,1) node {$\bullet$};
\end{scope}

\begin{scope}[xshift=3.4cm,yshift=3.7cm,scale=0.5]
\draw (-1,0)--(0,1)--(1,0)--(2,-1);
\draw (-1,0) node {$\bullet$};
\draw (0,1) node {$\bullet$};
\draw (1,0) node {$\bullet$};
\draw (2,-1) node {$\bullet$};
\end{scope}

\begin{scope}[xshift=4cm,yshift=1cm,scale=0.5]
\draw (-2,0)--(-1,1)--(0,2);
\draw (-1,1)--(0,0);
\draw (-2,0) node {$\bullet$};
\draw (-1,1) node {$\bullet$};
\draw (0,2) node {$\bullet$};
\draw (0,0) node {$\bullet$};
\end{scope}

\begin{scope}[xshift=-3cm,yshift=4cm,scale=0.5]
\draw (-2,0)--(-1,1)--(0,2)--(1,1);
\draw (-2,0) node {$\bullet$};
\draw (-1,1) node {$\bullet$};
\draw (0,2) node {$\bullet$};
\draw (1,1) node {$\bullet$};
\end{scope}

\begin{scope}[xshift=-1cm,yshift=-1cm,scale=0.5]
\draw (0,0)--(0,3);
\draw (0,0) node {$\bullet$};
\draw (0,1) node {$\bullet$};
\draw (0,2) node {$\bullet$};
\draw (0,3) node {$\bullet$};
\end{scope}    
\end{scope}
\end{tikzpicture}
\caption{On the left this is the mirror image of the diagram of $u=1014012$ with an extra initial $0$ and extra final $8$ and lines of slope $1$ drawn in red starting at each $0$. The segments with the red lines form the binary tree $\varphi(u)$, and $\psi(u)$ is obtained by identifying the red points that are on the same red lines. On the right we show the partition of $\Tam(C_4,\phi_2)$ with the intervals $I_2(u)$ in red, and we represented the associated tree $\psi(u)$ for each interval.}
\label{fig:bijectionmiror}
\end{figure}

Let $d_{0,p}=d_{1,p}=1$ and $d_{n,p}=|\Tam(C_{(n-1)p},\phi_p)|$ for all $n>1$.

\begin{theorem} \label{thm:genechains}
The generating function $D_p(x)$ of $(d_{n,p})_{n\geq 0}$ satisfies $D_p(x)=1+\frac{x\,D_p(x)}{(1-x\,D_p(x))^p}$.
\end{theorem}

\begin{proof}
By \cref{lem:intervalBE}, we have $d_{n,p} = \sum_{u \text{ a Tamari word of length }n} |I_p(u)|$. With \cref{prop:numberinterval}, it follows
$$d_{n,p} =\sum_{u \text{ a Tamari word of length }n} \quad \prod_{x \text{ an internal node of } \varphi(u)}    \binom{(\text{number of children of x}) -1+(p-1)}{p-1}.$$    
Thus $d_{n,p}$ is the number of planar rooted trees on $n+1$ vertices where each internal node is labeled by a positive integer less than or equal to $\displaystyle \binom{(\text{number of children of x}) -1+(p-1)}{p-1}$. It follows:
\begin{align*}
D_p(x) &= 1 + \sum_{k\geq 1} \binom{k +p-2}{p-1} \big(x\, D_p(x) \big)^k = 1 +x\,D_p(x)\,\sum_{k\geq 0} \binom{k+p-1}{p-1} \big(x\,D_p(x)\big)^k .
\end{align*}
Using $\frac{1}{(1-X)^p} = \sum_{k\geq 0} \binom{k+p-1}{p-1} X^k$, we have $D_p(x)=1+\frac{x\,D_p(x)}{(1-x\,D_p(x))^p}$.
\end{proof}

\begin{corollary} \label{coro:numberchainp}
For all $n\geq 0$ and $p\geq 1$ we have 
$\displaystyle d_{n,p} = \frac{1}{n+1} \sum_{k=0}^{n} \binom{n+1}{k} \binom{(p-1)\,k+n-1}{p\,k-1}$.
\end{corollary}

\begin{proof}
This follows from \cref{thm:genechains} by using the Lagrange inversion formula. Let $R(x)=x\,D_p(x)$. The formula of \cref{thm:genechains} can be written $R(x)=x\, \big(1+ \frac{R(x)}{(1-R(x))^p} \big)$. Let \mbox{$\Phi(u)=1+\frac{u}{(1-u)^p}$}. Thus we have $R(x)=x\,\Phi(R(x))$ with $R(0)=0$. By the Lagrange inversion formula, we have
$$d_{n,p}=[x^{n+1}]\, R(x) =\frac{1}{n+1} [u^n] \Phi(u)^{n+1} = \frac{1}{n+1} [u^n] \Big(1+\frac{u}{(1-u)^p} \Big)^{n+1}.$$
By the binomial theorem and the expansion of $\frac{1}{(1-X)^p}$ recalled in the proof of \cref{thm:genechains}, we have
$$\Big(1+\frac{u}{(1-u)^p} \Big)^{n+1} =  \sum_{k=0}^{n+1} \binom{n+1}{k} \frac{u^k}{(1-u)^{p\,k}}= \sum_{k=0}^{n+1} \binom{n+1}{k} u^k \sum_{m\geq 0} \binom{m+p\,k-1}{p\,k-1} u^m.$$
This finishes the proof of the corollary by taking $m=n-k$ and using the above formula for $d_{n,p}$.
\end{proof}

Note that when $p=1$, we recover as expected by \cref{propTamari} a formula giving the Catalan numbers, the terms in the sum being the Narayana numbers. The sequence $(d_{n,2})_{n\geq 0}$ corresponds to the sequence $A109081$ of the $\mathrm{OEIS}$, whose first terms are $1,1,3,10,37, 146, 602, 2563$. When $p\in \{3,4,5\}$ the associated sequences $(d_{n,p})_{n\geq 0}$ are respectively the sequences $A161797$, $A321798$ and $A321799$, but for $p\geq 6$ there are no corresponding entries in the $\mathrm{OEIS}$.

\begin{remark}
Combe and Giraudo \cite{Combe_2022} also obtained similar, but different, generalizations of the Tamari lattice called $\delta$-canyon lattices. In our generalization $\Tam(C_{np},\phi_p)$ we have lines of slope $p$, whereas they keep the lines of slopes $1$ that give the Tamari lattice but they relax the conditions on the height of the segments.   
\end{remark}

Now we are interested in the general case $\Tam(C_n,\phi)$. Denote by $c_{k_0,k_1,\dots,k_{n-1}}$ the number of elements of $\Tam(C_n,\phi)$ where $k_0=\phi(0)+1$ and $k_i=\phi(i)-\phi(i-1)$ for all $1\leq i<n$ (see the left-hand side of \cref{fig:proofpropenumerationchain}). By convention, let $c_{\emptyset}=1$. For example, the case $\phi=C_n$, which is the Tamari lattice $\Tam_{n+1}$ by \cref{propTamari}, gives $|Tam_{n+1}|=c_{1,1,\dots,1}$ with $n$ ones in the index.

\begin{proposition}
\label{prop:inductionformula}
We have $c_{k_0}=k_0+1$ and $c_{k_0,k_1}=1+k_1+\dfrac{k_0(k_0+5)}{2}$. For any $n\geq 3$, we have 


\begin{align*}
c_{k_0,k_1,\dots,k_{n-1}} &= c_{k_0,k_1,\dots,k_{n-2}} + c_{k_0,k_1,\dots,k_{n-3}}\times k_{n-1} + c_{k_0,k_1,\dots,k_{n-4}} \Big(\dfrac{(k_{n-2}+2)(k_{n-2}+1)}{2}-1 \Big) \\
&\quad +\sum_{i=0}^{n-3} c_{k_0,k_1,\dots,k_{i-2}} \sum_{j=1}^{k_i} c_{j,k_{i+1},k_{i+2},\dots,k_{n-2}}.  
\end{align*}
\end{proposition}

\begin{proof}
The first two formulas follow from \cref{prop:enumatmost2}. Now we suppose that $n\geq 3$. We consider the lattice $\Tam(C_n,\phi)$ and denote $k_0=\phi(0)+1$ and $k_q=\phi(q)-\phi(q-1)$ for all $q\in [n-1]$.
Let $i<n$ and $j\in [k_i]$. 
Let $S(i,j)$ be the set of the torclosed sets of $\Tam(C_n,\phi)$ that contain $(\phi(i)-j+1,n-1)$ but not $(\phi(i)-j,n-1)$, and $S(n)$ be the set of torclosed sets that do not contain $(\phi(n-1),n-1)$.  
Since all the sets that we just introduced form a partition of the elements of $\Tam(C_n,\phi)$, then 
$$c_{k_0,k_1,\dots,k_{n-1}}=|S(n)| + \sum_{i=0}^{n-1} \sum_{j=1}^{k_i} |S(i,j)|.$$

We have $|S(n)|=c_{k_0,k_1,\dots,k_{n-2}}$. Now we want to know $|S(i,j)|$.
A torclosed set in $S(i,j)$ is the disjoint union of a torclosed set of $\Tam(C_n,\phi_{|\{0,1,\dots,i-2\}})$ and of a specific torclosed set $T$ of $\Tam(C_n,\phi)$ with elements only in the components $\mathcal{C}_{i}$ through $\mathcal{C}_{n-2}$, where $T_i$ has at most $j$ elements. More precisely, these latter torclosed sets in $S(i,j)$ correspond bijectively to the torclosed sets of $\Tam(C_{n-\phi(i-1)},\phi')$ where $\phi'(0)=j$ and $\phi'(p)-\phi'(p-1)=k_{i+p}$ for all $p\in [n-i-2]$. See \cref{fig:proofpropenumerationchain} for an illustration of $S(i,j)$. This gives for all $i<n-2$, 
$|S(i,j)|= c_{k_0,k_1,\dots,k_{i-2}} \times c_{j,k_{i+1},k_{i+2},\dots,k_{n-2}}$,
and $|S(n-2,j)|= c_{k_0,k_1,\dots,k_{n-4}} \times c_{j}$, and $|S(n-1,j)|= c_{k_0,k_1,\dots,k_{n-3}}$. The formula of the proposition for $n\geq 3$ then follows from the above formula for $c_{k_0,k_1,\dots,k_{n-1}}$ by extracting the terms of the sum corresponding to $i=n-1$ and $i=n-2$.
\end{proof}

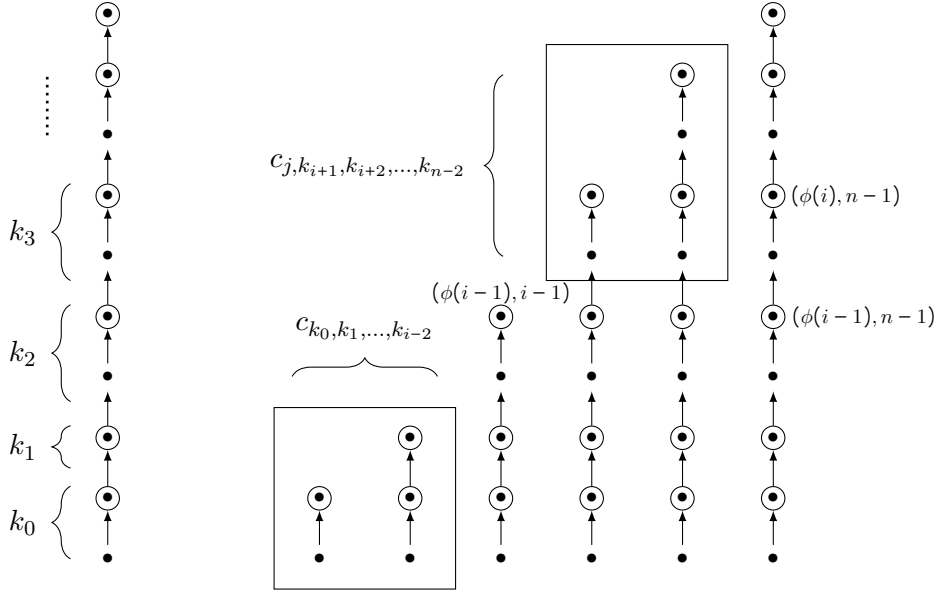
\begin{figure}
    \centering
\begin{tikzpicture}[scale=0.8]
\node[scale=0.8] (36) at (5,0) {$\bullet$};
\node[scale=0.8][draw,circle,inner sep=1pt] (37) at (5,1) {$\bullet$};
\node[scale=0.8][draw,circle,inner sep=1pt] (38) at (5,2) {$\bullet$};
\node[scale=0.8] (39) at (5,3) {$\bullet$};
\node[scale=0.8][draw,circle,inner sep=1pt] (40) at (5,4) {$\bullet$};
\node[scale=0.8] (41) at (5,5) {$\bullet$};
\node[scale=0.8][draw,circle,inner sep=1pt] (42) at (5,6) {$\bullet$};
\node[scale=0.8] (43) at (5,7) {$\bullet$};
\node[scale=0.8][draw,circle,inner sep=1pt] (44) at (5,8) {$\bullet$};
\node[scale=0.8][draw,circle,inner sep=1pt] (45) at (5,9) {$\bullet$};

\draw[->,>=latex] (36) -- (37);
\draw[->,>=latex] (37) -- (38);
\draw[->,>=latex] (38) -- (39);
\draw[->,>=latex] (39) -- (40);
\draw[->,>=latex] (40) -- (41);
\draw[->,>=latex] (41) -- (42);
\draw[->,>=latex] (42) -- (43);
\draw[->,>=latex] (43) -- (44);
\draw[->,>=latex] (44) -- (45);

\begin{scope}[rotate=90]
\draw [decorate,decoration={brace,amplitude=7pt,raise=2ex}]
(0,-4.8) -- (1.2,-4.8)
 node[midway,xshift=-2.5em]{$k_0$};
\end{scope}

\begin{scope}[rotate=90]
\draw [decorate,decoration={brace,amplitude=7pt,raise=2ex}]
(1.5,-4.8) -- (2.2,-4.8)
 node[midway,xshift=-2.5em]{$k_1$};
\end{scope}

\begin{scope}[rotate=90]
\draw [decorate,decoration={brace,amplitude=7pt,raise=2ex}]
(2.6,-4.8) -- (4.2,-4.8)
 node[midway,xshift=-2.5em]{$k_2$};
\end{scope}

\begin{scope}[rotate=90]
\draw [decorate,decoration={brace,amplitude=7pt,raise=2ex}]
(4.6,-4.8) -- (6.2,-4.8)
 node[midway,xshift=-2.5em]{$k_3$};
\end{scope}

\draw[line width=1pt,dotted] (4,7) -- (4,8);

\begin{scope}[xshift=8.5cm,xscale=1.5]
\node[scale=0.8] (6) at (0,0) {$\bullet$};
\node[scale=0.8][draw,circle,inner sep=1pt] (5) at (0,1) {$\bullet$};

\node[scale=0.8] (4) at (1,0) {$\bullet$};
\node[scale=0.8][draw,circle,inner sep=1pt] (3) at (1,1) {$\bullet$};
\node[scale=0.8][draw,circle,inner sep=1pt] (2) at (1,2) {$\bullet$};

\node[scale=0.8] (7) at (2,0) {$\bullet$};
\node[scale=0.8][draw,circle,inner sep=1pt] (8) at (2,1) {$\bullet$};
\node[scale=0.8][draw,circle,inner sep=1pt] (9) at (2,2) {$\bullet$};
\node[scale=0.8] (10) at (2,3) {$\bullet$};
\node[scale=0.8][draw,circle,inner sep=1pt] (11) at (2,4) {$\bullet$};

\node[scale=0.8] (12) at (3,0) {$\bullet$};
\node[scale=0.8][draw,circle,inner sep=1pt] (13) at (3,1) {$\bullet$};
\node[scale=0.8][draw,circle,inner sep=1pt] (14) at (3,2) {$\bullet$};
\node[scale=0.8] (15) at (3,3) {$\bullet$};
\node[scale=0.8][draw,circle,inner sep=1pt] (16) at (3,4) {$\bullet$};
\node[scale=0.8] (17) at (3,5) {$\bullet$};
\node[scale=0.8][draw,circle,inner sep=1pt] (18) at (3,6) {$\bullet$};

\node[scale=0.8] (19) at (4,0) {$\bullet$};
\node[scale=0.8][draw,circle,inner sep=1pt] (20) at (4,1) {$\bullet$};
\node[scale=0.8][draw,circle,inner sep=1pt] (21) at (4,2) {$\bullet$};
\node[scale=0.8] (22) at (4,3) {$\bullet$};
\node[scale=0.8][draw,circle,inner sep=1pt] (23) at (4,4) {$\bullet$};
\node[scale=0.8] (24) at (4,5) {$\bullet$};
\node[scale=0.8][draw,circle,inner sep=1pt] (25) at (4,6) {$\bullet$};
\node[scale=0.8] (26) at (4,7) {$\bullet$};
\node[scale=0.8][draw,circle,inner sep=1pt] (28) at (4,8) {$\bullet$};

\node[scale=0.8] (36) at (5,0) {$\bullet$};
\node[scale=0.8][draw,circle,inner sep=1pt] (37) at (5,1) {$\bullet$};
\node[scale=0.8][draw,circle,inner sep=1pt] (38) at (5,2) {$\bullet$};
\node[scale=0.8] (39) at (5,3) {$\bullet$};
\node[scale=0.8][draw,circle,inner sep=1pt] (40) at (5,4) {$\bullet$};
\node[scale=0.8] (41) at (5,5) {$\bullet$};
\node[scale=0.8][draw,circle,inner sep=1pt] (42) at (5,6) {$\bullet$};
\node[scale=0.8] (43) at (5,7) {$\bullet$};
\node[scale=0.8][draw,circle,inner sep=1pt] (44) at (5,8) {$\bullet$};
\node[scale=0.8][draw,circle,inner sep=1pt] (45) at (5,9) {$\bullet$};

\draw [decorate,decoration={brace,amplitude=7pt,raise=2ex}]
(-0.3,2.6) -- (1.3,2.6) node[midway,yshift=2.5em]{$c_{k_0,k_1,\dots,k_{i-2}}$};

\begin{scope}[rotate=90]
\draw [decorate,decoration={brace,amplitude=7pt,raise=2ex}]
(5,-2.3) -- (8,-2.3)
 node[midway,xshift=-5.5em]{$c_{j,k_{i+1},k_{i+2},\dots,k_{n-2}}$};
\end{scope}

\draw[->,>=latex] (6) -- (5);
\draw[->,>=latex] (4) -- (3);
\draw[->,>=latex] (3) -- (2);

\draw (-0.5,-0.5) rectangle (1.5,2.5);
\draw (2.5,4.6) rectangle (4.5,8.5);

\draw[->,>=latex] (7) -- (8);
\draw[->,>=latex] (8) -- (9);
\draw[->,>=latex] (9) -- (10);
\draw[->,>=latex] (10) -- (11);

\draw[->,>=latex] (12) -- (13);
\draw[->,>=latex] (13) -- (14);
\draw[->,>=latex] (14) -- (15);
\draw[->,>=latex] (15) -- (16);
\draw[->,>=latex] (16) -- (17);
\draw[->,>=latex] (17) -- (18);

\node[scale=0.7] (a) at (2,4.4) {$(\phi(i-1),i-1)$};
\node[scale=0.7] (b) at (6,4) {$(\phi(i-1),n-1)$};
\node[scale=0.7] (b) at (5.8,6) {$(\phi(i),n-1)$};

\draw[->,>=latex] (19) -- (20);
\draw[->,>=latex] (20) -- (21);
\draw[->,>=latex] (21) -- (22);
\draw[->,>=latex] (22) -- (23);
\draw[->,>=latex] (23) -- (24);
\draw[->,>=latex] (24) -- (25);
\draw[->,>=latex] (25) -- (26);
\draw[->,>=latex] (26) -- (28);

\draw[->,>=latex] (36) -- (37);
\draw[->,>=latex] (37) -- (38);
\draw[->,>=latex] (38) -- (39);
\draw[->,>=latex] (39) -- (40);
\draw[->,>=latex] (40) -- (41);
\draw[->,>=latex] (41) -- (42);
\draw[->,>=latex] (42) -- (43);
\draw[->,>=latex] (43) -- (44);
\draw[->,>=latex] (44) -- (45);
\end{scope}    
\end{tikzpicture}
    \caption{On the left we have $C_n$ with $\phi$ whose elements are circled. On the right we have
    an illustration of $S(i,j)$ in the proof of \cref{prop:inductionformula} where the elements $(\phi(i),j)\in \mathcal{C}_P^{\phi}$ for $i\leq j<n$ are circled.}
    \label{fig:proofpropenumerationchain}
\end{figure}

By letting $\mathrm{Cat}_0=\mathrm{Cat}_1=1$ and, for any $n>1$, $\mathrm{Cat}_n=c_{1,1,\dots,1}$ where the number of $1$s in the index is $n-1$, we recover from \cref{prop:inductionformula} the formula $\mathrm{Cat}_{n+1}=\mathrm{Cat}_{n} + \mathrm{Cat}_{n-1} + \mathrm{Cat}_{n-2}\times 2 + \sum_{i=0}^{n-3} \mathrm{Cat}_{i}\,\mathrm{Cat}_{n-i} = \sum_{i=0}^{n} \mathrm{Cat}_{i}\,\mathrm{Cat}_{n-i}$, which is a well-known recursive formula of the Catalan numbers.


\subsection{Lattices of $d$-torsion classes of the $(d-1)$-Auslander algebras of type $\textbf{A}$}
\label{sec:Auslander}

Let $d$ and $n$ be positive integers. Recall the notations and results from \cref{sec:backgroundhighertorsionclasses}. The poset $os_n^{d+1}$ is ordered with the product order $\leq$ (see \cref{fig:posetos33} for an example). By \cref{thm:torsionAuslander}, a subset $T\subseteq os_n^{d+1}$ is an element of $L_n^d$ if and only if it satisfies the following two conditions:

\begin{enumerate}[(1)]
\item \label{cond1} For all $i<n$, $\{x\in T\,|\,x_{d+1}=i\}$ is an order filter of the subposet $\{x\in os_n^{d+1}\,|\,x_{d+1}=i\}$. 
\item \label{cond2} For all $x,z\in T$, if $x\rightsquigarrow \tau_d(z)$, then any $y\in os_n^{d+1}$ with $y_i\in \{x_i,z_i\}$ for each $i$ must be in $T$.
\end{enumerate}

We now prove that we can replace condition \eqref{cond2} by a simpler one (in a sense that will be explained later). We first prove a lemma.

\begin{lemma}
\label{lem:12prime}
    Let $T\subseteq os_n^{d+1}$. Then $T\in L_n^d$ if and only if it satisfies \eqref{cond1} above and
\begin{equation}
\tag{$2'$}\label{2prime}
\begin{aligned}
&\text{For all $i<j<n$, if $(x_1,\dots,x_d,i)\in T$ and $(x_2+1,x_3+1,\dots,x_d+1,i+1,j)\in T$,} \\
&\text{then $(x_1,\dots,x_d,j)\in T$.}
\end{aligned}
\end{equation}
\end{lemma}

\begin{proof}
Let $T\subseteq os_n^{d+1}$. We prove that $T$ satisfies both \eqref{cond1} and \eqref{2prime} if and only if it satisfies both \eqref{cond1} and \eqref{cond2}. We suppose that $T$ satisfies \eqref{cond1}. 

Suppose that $T$ satisfies \eqref{cond2}. Let $i<j<n$. Suppose that $(x_1,\dots,x_d,i)\in T$ and $(x_2+1,x_3+1,\dots,x_d+1,i+1,j)\in T$. Since $(x_1,\dots,x_d,i) \rightsquigarrow \tau_d (x_2+1,x_3+1,\dots,x_d+1,i+1,j)$, by \eqref{cond2} we have $(x_1,\dots,x_d,j)\in T$. This proves that $T$ satisfies \eqref{2prime}.

Conversely, suppose that $T$ satisfies \eqref{2prime}. Let $x,z\in T$. Suppose $x\rightsquigarrow \tau_d(z)$, meaning $x_1\leq z_1-1\leq x_2 \leq z_2-1\leq \dots \leq x_{d+1} \leq z_{d+1}-1$. We have that any $y\in os_n^{d+1}$ with $y_i\in \{x_i,z_i\}$ for each $i$ is above $(x_1,\dots,x_d,z_{d+1})$ for the product order. Thus, using \eqref{cond1}, it suffices to prove $(x_1,\dots,x_d,z_{d+1})\in T$. Since $x=(x_1,\dots,x_d,x_{d+1})\in T$, it suffices to prove that $(x_2+1,x_3+1,\dots,x_d+1,x_{d+1}+1,z_{d+1})\in T$, as \eqref{2prime} would then imply that $(x_1,\dots,x_d,z_{d+1})\in T$. Since we have $x_1\leq z_1-1\leq x_2 \leq z_2-1\leq \dots \leq x_{d+1} \leq z_{d+1}-1$, we know that $z\leq (x_2+1,x_3+1,\dots,x_{d+1}+1,z_{d+1})$. Since $z\in T$ and $T$ satisfies \eqref{cond1}, we have $(x_2+1,x_3+1,\dots,x_{d+1}+1,z_{d+1})\in T$. This finishes the proof of the lemma.
\end{proof}

\begin{proposition}
\label{prop:reformulationthmtorsion}
     Let $T\subseteq os_n^{d+1}$. Then $T\in L_n^d$ if and only if it satisfies \eqref{cond1} above and
\begin{equation}
\tag{$2''$}\label{2primeprime}
\begin{aligned}
&\text{For all $i<j<n$, if $(x_1,\dots,x_d,i)\in T$ and $(i+1,\dots,i+1,j)\in T$,} \\
&\text{then $(x_1,\dots,x_d,j)\in T$.}
\end{aligned}
\end{equation}
\end{proposition}

\begin{proof}
Let $T\subseteq os_n^{d+1}$. We prove that $T$ satisfies both \eqref{cond1} and \eqref{2primeprime} if and only if it satisfies both \eqref{cond1} and \eqref{2prime}. Using \cref{lem:12prime}, this will prove the proposition.  We suppose that $T$ satisfies \eqref{cond1}. 

Suppose that $T$ satisfies \eqref{2primeprime}. Let $i<j<n$. Suppose that $(x_1,\dots,x_d,i)\in T$ and $(x_2+1,x_3+1,\dots,x_d+1,i+1,j)\in T$. We need to prove that $(x_1,\dots,x_d,j)\in T$. Since $(x_1,\dots,x_d,i) \in os_n^{d+1}$, then $x_p\leq i$ for all $p\in [d]$. Thus $(i+1,\dots,i+1,j)\geq (x_2+1,x_3+1,\dots,x_d+1,i+1,j) \in T$.  Since $T$ satisfies \eqref{cond1}, then $(i+1,\dots,i+1,j)\in T$. Since $T$ satisfies \eqref{2primeprime}, then $(x_1,\dots,x_d,j)\in T$. Thus $T$ satisfies \eqref{2prime}.

Conversely, suppose that $T$ satisfies \eqref{2prime}. Let $i<j<n$. Suppose that $(x_1,\dots,x_d,i)\in T$ and $(i+1,\dots,i+1,j)\in T$. We need to prove that $(x_1,\dots,x_d,j)\in T$.
For $k \in \{0,1,\dots ,i\}$, let 
\begin{align*}
X_k &:=(\max(x_1,i-k),\max(x_2,i-k), \dots , \max(x_d,i-k), i), \\
L_k &:=(\max(x_2,i-k)+1, \max(x_3,i-k)+1, \dots , \max(x_d,i-k)+1, i+1, j), \\
Y_k &:=(\max(x_1,i-k),\max(x_2,i-k), \dots , \max(x_d,i-k), j).
\end{align*}    

See \cref{fig:proofreformtheorem} for an illustration of these elements and of the following arguments.
Since $T$ satisfies \eqref{2prime}, having both $X_k\in T$ and $L_k\in T$ implies $Y_k\in T$. Let $P_k$ be the property $(X_k,L_k) \in T^2$. Let us prove by induction that $P_k$ is true for all $k \in \{0,1,\dots ,i\}$.

Since $(x_1,\dots,x_d,i) \in os_n^{d+1}$, then $x_p\leq i$ for all $p\in [d]$. Thus $X_0=(i,i,\dots,i)\geq (x_1,\dots,x_d,i)\in T$. Since $T$ satisfies \eqref{cond1}, then $X_0\in T$. We already know that $L_0=(i+1,\dots,i+1,j)\in T$. Thus $P_0$ is true.

Let $k \in \{0,1,\dots ,i-1\}$ and assume that $P_k$ is true. It means that $(X_k,L_k) \in T^2$, which implies $Y_k\in T$ since $T$ satisfies \eqref{2prime}. 
We have $X_{k+1}\geq (x_1,x_2,\dots,x_d,i)\in T$, thus $X_{k+1}\in T$ since $T$ satisfies \eqref{cond1}. 
Let $ l\in [d-1]$. Since $x_{l+1} \geq x_l $, we have $ \max(x_{l+1},i-k-1)+1  \geq \max(x_l,i-k)$. This proves $L_{k+1} \geq Y_k$, and since $Y_k\in T$, then $L_{k+1}\in T$ since $T$ satisfies \eqref{cond1}. Then $(X_{k+1},L_{k+1}) \in T^2$, meaning $P_{k+1}$ is true.

By induction, this proves in particular that $P_i$ is true. Thus $(X_i,L_i) \in T^2$. This implies that $Y_i\in T$ since $T$ satisfies \eqref{2prime}. Thus $Y_i=(x_1,\dots,x_d,j) \in T$. This finishes the proof of the proposition.
\end{proof}

\begin{figure}
    \centering
\begin{tikzpicture}[scale=1]

\draw[gray] (0,0) ellipse (1.3 and 2.5);
\node at (0,-3) {$\mathcal{C}_i$};

\draw[dashed,red] (0,2.5) -- (0,-1);

\node[above] at (0,2.5) {\small $X_0=(i,\dots,i)$};
\node[left] at (0,2) {\small $X_1$};
\node[left] at (0,1) {\small $X_{k-1}$};
\node[left] at (0,0.5) {\small $X_k$};
\node[left] at (0,-1) {\small $X_i=(x_1,\dots,x_d,i)$};

\node at (0,2.5) {$\bullet$};
\node at (0,2) {$\bullet$};
\node at (0,1) {$\bullet$};
\node at (0,0.5) {$\bullet$};
\node at (0,-1) {$\bullet$};

\begin{scope}[xshift=6cm]
\draw[gray] (0,0) ellipse (1.4 and 2.5);
\draw[gray] (0,1) ellipse (1.7 and 3.5);
\draw[gray] (0,1.5) ellipse (1.9 and 4);
\node at (0,-3) {$\mathcal{C}_j$};   

\draw[dashed,red] (0,2.5) -- (0,-1);

\node[above right] at (0,2.5) {\small $Y_0=(i,\dots,i,j)$};
\node[right] at (0,2) {\small $Y_1$};
\node[right] at (0,1) {\small $Y_{k-1}$};
\node[right] at (0,0.5) {\small $Y_k$};
\node[right] at (0,-1) {\small $Y_i=(x_1,\dots,x_d,j)$};

\node at (0,2.5) {$\bullet$};
\node at (0,2) {$\bullet$};
\node at (0,1) {$\bullet$};
\node at (0,0.5) {$\bullet$};
\node at (0,-1) {$\bullet$};

\node[left] at (-0.5,3.7) {\small $L_1$};
\node[left] at (-1,3) {\small $L_2$};
\node[left] at (-1,2) {\small $L_{k}$};
\node[left] at (-1,1.5) {\small $L_{k+1}$};
\node[left] at (-1,0) {\small $L_i$};
\node at (-0.5,3.7) {$\bullet$};
\node at (-1,2) {$\bullet$};
\node at (-1,3) {$\bullet$};
\node at (-1,1.5) {$\bullet$};
\node at (-1,0) {$\bullet$};
\draw[dashed] (0,4.5)--(-1,3)--(-1,1.5)--(-1,0);
\draw[dashed] (-0.5,3.7)--(0,2.5);
\draw[dashed] (-1,3)--(0,2);
\draw[dashed] (-1,1.5)--(0,0.5);
\draw[dashed] (-1,0)--(0,-1);
\draw[dashed] (-1,2)--(0,1);
\node at (0,4.5) {$\bullet$};
\node at (0,5.5) {$\bullet$};
\node[above] at (0,4.5) {\small $L_0=(i+1,\dots,i+1,j)$};
\node[above] at (0,5.5) {\small $(j,\dots,j)$};
\end{scope}
\end{tikzpicture}
    \caption{Illustration of the proof of \cref{prop:reformulationthmtorsion}}
    \label{fig:proofreformtheorem}
\end{figure}
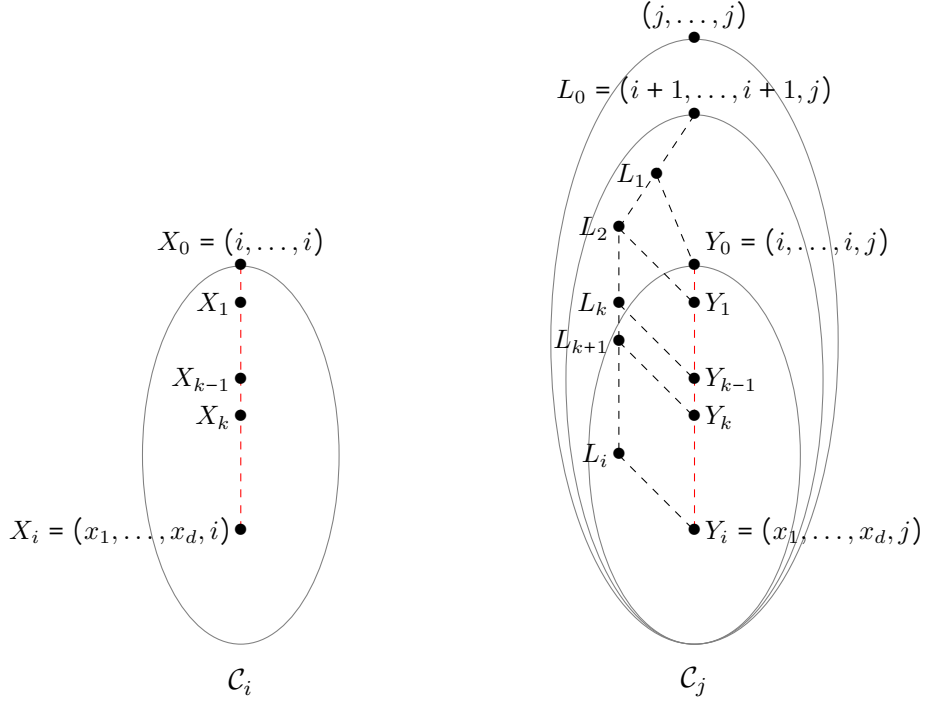

Note that \cref{prop:reformulationthmtorsion} is a nice improvement on \cref{thm:torsionAuslander}. Indeed, we replaced \eqref{cond2} with \eqref{2primeprime}, which still says that $(x_1,\dots,x_d,j)\in T$ as soon as both $(x_1,\dots,x_d,i)$ and another element are in $T$. But the other element no longer depends on $x_1,\dots,x_d$, only on $i$. 

The next theorem follows from the previous result and motivated the definition of the $(P,\phi)$-Tamari lattices. 

\begin{theorem}
\label{thm:torsionAuslanderistam}
$L_n^d = \Tam\big((os_n^d,\leq),\phi\big)$ with $\phi$ defined by $\phi(i)=(i,\dots,i)$ for all $i<n$.
\end{theorem}

\begin{proof}
Both $L_n^d$ and $\Tam\big((os_n^d,\leq),\phi\big)$ are lattices ordered by inclusion, thus it suffices to prove that the $d$-torsion classes $T\in L_n^d$ are the torclosed sets of $\Tam\big((os_n^d,\leq),\phi\big)$.
Recall the definition of $\mathcal{C}_P^{\phi}$ in \cref{sectionGenerality}. By definition, $T\in L_n^d$ satisfies \eqref{1aus} if and only if $T$ is an order filter of $\mathcal{C}_P^{\phi}$.
Since $(\phi(i+1),j)=(i+1,\dots,i+1,j)$, then $T\in L_n^d$ satisfies \eqref{2primeprime} if and only if $T$ is a torclosed set of $\Tam\big((os_n^d,\leq),\phi\big)$. This proves the theorem using \cref{prop:reformulationthmtorsion}.   
\end{proof}

Note that $L_n^d$ is a $(P,\phi)$-Tamari lattice that enjoys nice properties as $os_n^d$ is a lattice and the associated $\phi$ satisfies $\phi(0)=(0,\dots,0)=\hat{0}_{os_n^d}$. From \cref{thm:torsionAuslanderistam} and the results of \cref{sectionPphiTamari} follow:

\begin{corollary}
\label{cor:resultsLnd}
The lattice $L_n^d$ inherits all properties of the $(P,\phi)$-Tamari lattices. In particular, it is a join-semidistributive lattice, and it is semidistributive if and only if $n\leq 2$ or $d=1$. The spine of $L_n^d$ corresponds to  $J((os_n^{d+1})^{op})$. We have $\ell(L_n^d)=|\mathrm{JIrr}(L_n^d)|=\binom{n+d}{d+1}$ and
$|\mathrm{MIrr}(L_n^d)|=(d-1)\,\binom{n+d}{d+2}+\frac{n(n+1)}{2}$.
\end{corollary}

\begin{proof}
It is immediate that $|os_n^d|=\binom{n+d-1}{d}$. We only need to prove the formula for the number of meet-irreducibles. We use \cref{prop:meetirr}. The number of the first kind of meet-irreducibles is $\binom{n+d}{d+1}$. Counting the second kind of meet-irreducibles amounts to count, for any $i<n$ and any $j\leq i$, the number of elements of $\mathcal{C}_i$ incomparable to $(\phi(j),i)$. Since $\phi(0)=\hat{0}_{os_n^d}$, no such elements exist if $i\in \{0,1\}$, nor if $j=\{0,i\}$ since $(\phi(0),i)$ and $(\phi(i),i)$ are respectively the minimum and maximum elements of $\mathcal{C}_i$. Then 
$$|\mathrm{MIrr}(L_n^d)|=\binom{n+d}{d+1}+\sum_{i=2}^{n-1} \,\,\sum_{j=1}^{i-1}\,\, |\{(x,i)\in \mathcal{C}_i \mid (x,i) \not\sim (\phi(j),i) \}|$$

Let $i$ and $j$ such that $2\leq i<n$ and $0<j<i$. The integer in the sum above is also the number of elements of $os_n^d$ less than $\phi(i)=(i,\dots,i)$ that are incomparable to $\phi(j)=(j,\dots,j)$. Such an element needs to start with a $k<j$, and finish with a $l\in \{j+1,j+2,\dots,i\}$, and between the two it should be non-decreasing. Thus when $k$ and $l$ are fixed, the number of such elements is $|os_{l-k+1}^{d-2}|=\binom{l-k+d-2}{d-2}$. This proves that
$$|\mathrm{MIrr}(L_n^d)|=\binom{n+d}{d+1}+\sum_{i=2}^{n-1} \,\,\sum_{j=1}^{i-1}\,\, \sum_{k=0}^{j-1}\,\, \sum_{l=j+1}^{i} \binom{l-k+d-2}{d-2} .$$
By successive use of the Hockey-stick identity, this reduces to
$$|\mathrm{MIrr}(L_n^d)|=\binom{n+d}{d+1}+\sum_{i=2}^{n-1} \,(i-1)\, \binom{i+d}{d} -2\,\binom{n+d}{d+2} +\frac{(n-1)(n+2)}{2}$$
Using the identity $i\,\binom{i+d}{d}=(d+1)\,\binom{i+d}{d+1}$ and again the Hockey-stick identity, we obtain the desired formula.
\end{proof}

\begin{example}
The lattice $L_3^3$ is given in \cref{fig:HasseL33} (which is taken from \cite{August_2025}). As proved in \cref{cor:resultsLnd} it is not semidistributive (as observed in \cite[Example 5.21]{August_2025}), and we can check that the number of meet-irreducibles is $18$, as predicted by our formula. 
\end{example}

\begin{definition}
The \emph{totally symmetric $d$-dimensional partitions} are the order ideals $I$ of $(\mathbb{N}^d,\leq_{prod})$ such that if $x=(x_1,x_2,\dots,x_d)\in I$, then $x_{\sigma}:=(x_{\sigma(1)},x_{\sigma(2)},\dots,x_{\sigma(d)})\in I$ for all permutations $\sigma \in \mathfrak{S}_d$, where $\mathfrak{S}_d$ is the set of all permutations of $[d]$.
\end{definition}

\begin{lemma}
\label{lem:totasym}
Let $I$ be a totally symmetric $d$-dimensional partition and $x\in I$ be a maximal element of $I$. Then $x_{\sigma}$ is also a maximal element of $I$ for all $\sigma\in \mathfrak{S}_d$.     
\end{lemma}

\begin{proof}
Seeking a contradiction, suppose that there exists $\sigma\in \mathfrak{S}_d$ such that $x_{\sigma}$ is not a maximal element of $I$. Then there exists $y\in I$ such that $x_{\sigma} <_{prod} y$. By applying $\sigma^{-1}$ to both sides, which keeps the product order inequality, we obtain $x<_{prod} y_{\sigma^{-1}}$. But $y_{\sigma^{-1}} \in I$ and $x$ is a maximal element of $I$. This is absurd.
\end{proof}

\begin{figure}
\begin{subfigure}{.5\textwidth}
\centering
\begin{tikzpicture}
\begin{scope}[scale=1.1]
\node[draw,circle,inner sep=1pt] (10) at (0,0) {$(0,0,0)$};
\node (9) at (1,1) {$(0,0,1)$};
\node (8) at (2,2) {$(0,1,1)$};
\node[draw,circle,inner sep=1pt] (7) at (3,3) {$(1,1,1)$};
\node (6) at (0,2) {$(0,0,2)$};
\node (5) at (1,3) {$(0,1,2)$};
\node (4) at (0,4) {$(0,2,2)$};
\node (3) at (2,4) {$(1,1,2)$};
\node (2) at (1,5) {$(1,2,2)$};
\node[draw,circle,inner sep=1pt] (1) at (0,6) {$(2,2,2)$};
\draw[->,>=latex] (10) -- (9);
\draw[->,>=latex] (9) -- (6);
\draw[->,>=latex] (9) -- (8);
\draw[->,>=latex] (6) -- (5);
\draw[->,>=latex] (8) -- (5);
\draw[->,>=latex] (8) -- (7);
\draw[->,>=latex] (5) -- (3);
\draw[->,>=latex] (5) -- (4);
\draw[->,>=latex] (7) -- (3);
\draw[->,>=latex] (3) -- (2);
\draw[->,>=latex] (4) -- (2);
\draw[->,>=latex] (2) -- (1);
\end{scope}
\end{tikzpicture}
\caption{The poset $os_3^3$ with the chain $\phi:i\mapsto (i,\dots,i)$ whose elements are circled.}
\label{fig:posetos33}
\end{subfigure}%
\begin{subfigure}{.5\textwidth}
\centering
 \begin{tikzpicture}[line join=bevel, yscale=.3, xscale=0.65,scale=0.85]
	        \node (1) at (127.0bp,810.0bp) {$\bullet$};
			\node (23) at (199.0bp,738.0bp) {$\bullet$};
			\node (25) at (127.0bp,738.0bp) {$\bullet$};
			\node (2) at (55.0bp,810.0bp) {$\bullet$};
			\node (35) at (55.0bp,594.0bp) {$\bullet$};
			\node (26) at (127.0bp,954.0bp) {$\bullet$};
			\node (36) at (190.0bp,882.0bp) {$\bullet$};
			\node (38) at (127.0bp,882.0bp) {$\bullet$};
			\node (22) at (199.0bp,666.0bp) {$\bullet$};
			\node (33) at (271.0bp,666.0bp) {$\bullet$};
			\node (5) at (170.0bp,234.0bp) {$\bullet$};
			\node (4) at (190.0bp,162.0bp) {$\bullet$};
			\node (16) at (199.0bp,594.0bp) {$\bullet$};
			\node (9) at (254.0bp,522.0bp) {$\bullet$};
			\node (15) at (110.0bp,522.0bp) {$\bullet$};
			\node (30) at (326.0bp,378.0bp) {$\bullet$};
			\node (28) at (271.0bp,306.0bp) {$\bullet$};
			\node (29) at (326.0bp,306.0bp) {$\bullet$};
			\node (11) at (215.0bp,306.0bp) {$\bullet$};
			\node (27) at (271.0bp,234.0bp) {$\bullet$};
			\node (43) at (262.0bp,162.0bp) {$\bullet$};
			\node (32) at (326.0bp,522.0bp) {$\bullet$};
			\node (31) at (326.0bp,450.0bp) {$\bullet$};
			\node (19) at (254.0bp,450.0bp) {$\bullet$};
			\node (24) at (127.0bp,666.0bp) {$\bullet$};
			\node (7) at (254.0bp,378.0bp) {$\bullet$};
			\node (3) at (190.0bp,90.0bp) {$\bullet$};
			\node (13) at (182.0bp,450.0bp) {$\bullet$};
			\node (34) at (271.0bp,738.0bp) {$\bullet$};
			\node (14) at (110.0bp,450.0bp) {$\bullet$};
			\node (6) at (215.0bp,234.0bp) {$\bullet$};
			\node (40) at (199.0bp,810.0bp) {$\bullet$};
			\node (37) at (55.0bp,1026.0bp) {$\bullet$};
			\node (39) at (55.0bp,954.0bp) {$\bullet$};
			\node (42) at (110.0bp,90.0bp) {$\bullet$};
			\node (10) at (110.0bp,378.0bp) {$\bullet$};
			\node (8) at (130.0bp,306.0bp) {$\bullet$};
			\node (21) at (55.0bp,306.0bp) {$\bullet$};
			\node (41) at (127.0bp,1026.0bp) {$\bullet$};
			\node (45) at (127.0bp,1120.0bp) {$\mathcal{C}_{os_3^3}^{\phi}$};
			\node (44) at (271.0bp,810.0bp) {$\bullet$};
			\node (18) at (127.0bp,594.0bp) {$\bullet$};
			\node (17) at (182.0bp,522.0bp) {$\bullet$};
			\node (20) at (271.0bp,594.0bp) {$\bullet$};
			\node (0) at (110.0bp,18.0bp) {$\emptyset$};
			\node (12) at (182.0bp,378.0bp) {$\bullet$};
			\draw [<-] (1) ..controls (152.2bp,784.8bp) and (165.32bp,771.68bp)  .. (23);
			\draw [<-] (1) ..controls (127.0bp,784.13bp) and (127.0bp,774.97bp)  .. (25);
			\draw [<-] (2) ..controls (55.0bp,754.39bp) and (55.0bp,667.55bp)  .. (35);
			\draw [<-] (26) ..controls (149.17bp,928.66bp) and (160.26bp,915.99bp)  .. (36);
			\draw [<-] (26) --(38);
			\draw [<-] (23) ..controls (199.0bp,712.13bp) and (199.0bp,702.97bp)  .. (22);
			\draw [<-] (23) ..controls (224.2bp,712.8bp) and (237.32bp,699.68bp)  .. (33);
			\draw [<-] (5) --(4);
			\draw [<-] (16) ..controls (218.51bp,568.46bp) and (227.43bp,556.78bp)  .. (9);
			\draw [<-] (16) ..controls (168.17bp,569.06bp) and (149.82bp,554.21bp)  .. (15);
			\draw [<-] (30) ..controls (306.49bp,352.46bp) and (297.57bp,340.78bp)  .. (28);
			\draw [<-] (30) -- (29); 
			\draw [<-] (30) --(11);
			\draw [<-] (27) ..controls (242.97bp,209.09bp) and (227.16bp,195.03bp)  .. (4);
			\draw [<-] (27) ..controls (267.77bp,208.13bp) and (266.62bp,198.97bp)  .. (43);
			\draw [<-] (28)--(5);
			\draw [<-] (28) ..controls (271.0bp,280.13bp) and (271.0bp,270.97bp)  .. (27);
			\draw [<-] (32) ..controls (326.0bp,496.13bp) and (326.0bp,486.97bp)  .. (31);
			\draw [<-] (32) ..controls (300.8bp,496.8bp) and (287.68bp,483.68bp)  .. (19);
			\draw [<-] (25) ..controls (152.2bp,712.8bp) and (165.32bp,699.68bp)  .. (22);
			\draw [<-] (25) ..controls (127.0bp,712.13bp) and (127.0bp,702.97bp)  .. (24);
			\draw [<-] (7) --(11);
			\draw [<-] (4) ..controls (190.0bp,136.13bp) and (190.0bp,126.97bp)  .. (3);
			\draw [<-] (31) ..controls (326.0bp,424.13bp) and (326.0bp,414.97bp)  .. (30);
			\draw [<-] (31) ..controls (300.8bp,424.8bp) and (287.68bp,411.68bp)  .. (7);
			\draw [<-] (9) ..controls (228.8bp,496.8bp) and (215.68bp,483.68bp)  .. (13);
			\draw [<-] (9) ..controls (254.0bp,496.13bp) and (254.0bp,486.97bp)  .. (19);
			\draw [<-] (34) ..controls (271.0bp,712.13bp) and (271.0bp,702.97bp)  .. (33);
			\draw [<-] (15) ..controls (135.2bp,496.8bp) and (148.32bp,483.68bp)  .. (13);
			\draw [<-] (15) ..controls (110.0bp,496.13bp) and (110.0bp,486.97bp)  .. (14);
			\draw [<-] (6) --(4);
			\draw [<-] (36) ..controls (167.83bp,856.66bp) and (156.74bp,843.99bp)  .. (1);
			\draw [<-] (36) ..controls (193.23bp,856.13bp) and (194.38bp,846.97bp)  .. (40);
			\draw [<-] (37) ..controls (55.0bp,1000.1bp) and (55.0bp,990.97bp)  .. (39);
			\draw [<-] (37) ..controls (34.156bp,999.64bp) and (24.765bp,985.71bp)  .. (19.0bp,972.0bp) .. controls (3.1542bp,934.31bp) and (0.0bp,922.88bp)  .. (0.0bp,882.0bp) .. controls (0.0bp,882.0bp) and (0.0bp,882.0bp)  .. (0.0bp,306.0bp) .. controls (0.0bp,228.53bp) and (56.874bp,150.79bp)  .. (42);
			\draw [<-] (10) ..controls (116.08bp,352.24bp) and (118.34bp,342.67bp)  .. (8);
			\draw [<-] (10) ..controls (90.488bp,352.46bp) and (81.565bp,340.78bp)  .. (21);
			\draw [<-] (8)--(5);
			\draw [<-] (41) ..controls (127.0bp,1000.1bp) and (127.0bp,990.97bp)  .. (26);
			\draw [<-] (41) ..controls (101.8bp,1000.8bp) and (88.685bp,987.68bp)  .. (39);
			\draw [<-] (35) ..controls (55.0bp,527.29bp) and (55.0bp,392.8bp)  .. (21);
			\draw [<-] (29) ..controls (317.8bp,280.8bp) and (304.68bp,267.68bp)  .. (27);
			\draw [<-] (29) ..controls (296.78bp,282.89bp) and (257.24bp,263.12bp)  .. (6);
			\draw [<-] (38)--(1);
			\draw [<-] (38)--(2);
			\draw [<-] (45) --(37);
			\draw [<-] (45) -- (41);
			\draw [<-] (45) -- (44);
			\draw [<-] (18) ..controls (120.92bp,568.24bp) and (118.66bp,558.67bp)  .. (15);
			\draw [<-] (18) ..controls (146.51bp,568.46bp) and (155.43bp,556.78bp)  .. (17);
			\draw [<-] (39) ..controls (55.0bp,911.2bp) and (55.0bp,867.25bp)  .. (2);
			\draw [<-] (22) ..controls (199.0bp,640.13bp) and (199.0bp,630.97bp)  .. (16);
			\draw [<-] (22) ..controls (173.8bp,640.8bp) and (160.68bp,627.68bp)  .. (18);
			\draw [<-] (22) ..controls (224.2bp,640.8bp) and (237.32bp,627.68bp)  .. (20);
			\draw [<-] (17) ..controls (182.0bp,496.13bp) and (182.0bp,486.97bp)  .. (13);
			\draw [<-] (43) ..controls (235.63bp,148.7bp) and (230.66bp,146.25bp)  .. (226.0bp,144.0bp) .. controls (197.52bp,130.25bp) and (164.83bp,115.09bp)  .. (42);
			\draw [<-] (43) ..controls (236.8bp,136.8bp) and (223.68bp,123.68bp)  .. (3);
			\draw [<-] (20) ..controls (264.92bp,568.24bp) and (262.66bp,558.67bp)  .. (9);
			\draw [<-] (20) ..controls (240.17bp,569.06bp) and (221.82bp,554.21bp)  .. (17);
			\draw [<-] (42) ..controls (110.0bp,64.131bp) and (110.0bp,54.974bp)  .. (0);
			\draw [<-] (3) ..controls (162.32bp,65.085bp) and (146.7bp,51.029bp)  .. (0);
			\draw [<-] (44) ..controls (292.7bp,783.97bp) and (302.12bp,770.05bp)  .. (307.0bp,756.0bp) .. controls (326.0bp,685.22bp) and (326.0bp,595.42bp)  .. (32);
			\draw [<-] (44) ..controls (271.0bp,784.13bp) and (271.0bp,774.97bp)  .. (34);
			\draw [<-] (13) ..controls (207.2bp,424.8bp) and (220.32bp,411.68bp)  .. (7);
			\draw [<-] (13) ..controls (182.0bp,424.13bp) and (182.0bp,414.97bp)  .. (12);
			\draw [<-] (11) ..controls (173.8bp,280.8bp) and (170.68bp,267.68bp)  .. (5);
			\draw [<-] (11) --(6);
			\draw [<-] (33) ..controls (271.0bp,640.13bp) and (271.0bp,630.97bp)  .. (20);
			\draw [<-] (19) ..controls (254.0bp,424.13bp) and (254.0bp,414.97bp)  .. (7);
			\draw [<-] (24) ..controls (101.8bp,640.8bp) and (88.685bp,627.68bp)  .. (35);
			\draw [<-] (24) ..controls (127.0bp,640.13bp) and (127.0bp,630.97bp)  .. (18);
			\draw [<-] (14) ..controls (110.0bp,424.13bp) and (110.0bp,414.97bp)  .. (10);
			\draw [<-] (14) ..controls (135.2bp,424.8bp) and (148.32bp,411.68bp)  .. (12);
			\draw [<-] (12) ..controls (162.49bp,352.46bp) and (153.57bp,340.78bp)  .. (8);
			\draw [<-] (12) --(11);
			\draw [<-] (40) ..controls (199.0bp,784.13bp) and (199.0bp,774.97bp)  .. (23);
			\draw [<-] (40) ..controls (224.2bp,784.8bp) and (237.32bp,771.68bp)  .. (34);
			\draw [<-] (21) ..controls (50.661bp,247.51bp) and (47.371bp,148.61bp)  .. (74.0bp,72.0bp) .. controls (77.662bp,61.464bp) and (83.876bp,51.001bp)  .. (0);
		\end{tikzpicture}
		\caption{The lattice $L_3^3$.}
		\label{fig:HasseL33}
\end{subfigure}
\caption{}
\label{fig:3torsionauslander}
\end{figure}

We recall the following result of J.R. Stembridge:

\begin{proposition}[\cite{STEMBRIDGE1995227}]
\label{prop:stembridge}
The number of totally symmetric $3$-dimensional partitions, also called \emph{totally symmetric plane partitions}, which fit in a $3$-dimensional box whose sides all have length $n$ is given by the sequence $(a_n)_{n\geq 0}$ ($A005157$ of the $\mathrm{OEIS}$ \cite{oeis}) defined for all $n\geq 0$ by 
$$\displaystyle a_n:=\prod_{i=1}^n \prod_{j=i}^n \prod_{k=j}^n \,\frac{i+j+k-1}{i+j+k-2}.$$  
\end{proposition}

The first ten values of the sequence $(a_n)_{n\geq 0}$ are $1, 2, 5, 16, 66, 352, 2431, 21760, 252586, 3803648$.

\begin{proposition}
\label{proptotallysym}
 We have an isomorphism of posets between $os_n^d$ and $os_{d+1}^{n-1}$, thus $|J(os_n^d)| = |J(os_{d+1}^{n-1})|$. We have that $|J(os_n^d)|$ is the number of totally symmetric $d$-dimensional partitions which fit in a $d$-dimensional box whose sides all have length $n$.  We deduce that $|J(os_3^d)|=2^{d+1}$ and $|J(os_4^d)|=a_{d+1}$ where 
 $(a_d)_{d\geq 0}$ is defined in \cref{prop:stembridge}. 
\end{proposition}

\begin{proof}
The function which maps $(x_1,x_2,\dots,x_d)\in os_n^d$ to the non-decreasing sequence of size $n-1$ with $n-1-x_d$ zeros, $x_d-x_{d-1}$ ones, and so on where the number of $i$'s is $x_{d-i+1}-x_{d-i}$ for all $i\in [d]$, is an isomorphism of posets from $os_n^d$ to $os_{d+1}^{n-1}$. Indeed this bijection comes from reading row by row or column by column the Ferrers diagram of $(x_1,x_2,\dots,x_d)\in os_n^d$. The fact that it is an isomorphism of posets is easily seen with this representation.

We now prove the combinatorial interpretation for $|J(os_n^d)|$. By definition, we can identify a totally symmetric $d$-dimensional partition with its set of maximal elements. By \cref{lem:totasym}, we in fact just need to specify the maximal elements that are non-decreasing. This shows that the totally symmetric $d$-dimensional partitions which fit in a $d$-dimensional box whose sides all have length $n$ correspond to the antichains of $(os_n^d,\leq)$. These latter are counted by $|J(os_n^d)|$. This proves what we wanted.

Finally, from this combinatorial interpretation the formulas for $|J(os_3^d)|$ and $|J(os_4^d)|$ follow.
\end{proof}

Unfortunately, no formulas are known for the number of totally symmetric $d$-dimensional partitions which fit in a $d$-dimensional box whose sides all have length $n\geq 5$.
Building on the above results, we are able to obtain formulas for the number of elements of $L_n^d$ when $n\leq 4$ (no formulas were known for $n\in \{3,4\}$). We start with a lemma.

\begin{lemma}
\label{lem:nbrorderfilterincomparable}
Let $P$ be a finite poset and $x\in P$. Let $I$ be the subposet of $P$ on the elements that are incomparable to $x$. Then the number of order filters of $P$ where $x$ is a generator (which means $x$ is a minimal element) is equal to $|J(I)|$. Moreover $|J(P)|=|J(P\setminus \{x\})|+|J(I)|$.
\end{lemma}

\begin{proof}
These order filters are identified with their minimal elements, which are the antichains of $P$ that contain $x$. Thus the number of these order filters is the number of antichains of $P$ formed by elements that are all incomparable to $x$. The results follow. 
\end{proof}

\begin{proposition}
\label{prop:number4torsions}
Let $d\geq 1$. We have $|L_1^d|=2$, $|L_2^d|=d+4$, and $|L_3^d| = d+3+5\times 2^d$. Let $K_1:=[(0,\dots,0,2,2),(1,3,\dots,3)]$, $K_2:=[(0,\dots,0,2,3),(1,3,\dots,3)]$ and $K_3:=\{x\in os_4^{d+1}\mid x\not\geq 1\cdots 13\}$ be three subposets of $os_4^{d+1}$. Then
\begin{equation}
\tag{$\ast$}
\label{formulaL4d}
|L_4^d|=8+d+3\times 2^{d+1} + (d+4)\,(a_d-1) + \sum_{i=2}^d a_i + 2\,\big(|J(K_1)|+|J(K_2)|\big) +|J(K_3)|,  
\end{equation}
where the sequence $(a_d)_{d\geq 0}$ is defined in \cref{prop:stembridge}.
\end{proposition}

\begin{proof}
The formulas for $|L_n^d|$ with $n\leq 3$ follow from \cref{thm:torsionAuslanderistam} and \cref{propComptage3}, using for $|L_3^d|$ the formula $|J(os_3^d)|=2^{d+1}$ obtained in \cref{proptotallysym}. 
For $|L_4^d|$ it requires more work. 
By \cref{thm:torsionAuslanderistam}, $L_4^d= \Tam(os_4^d,\phi)$ with $\phi$ defined by $\phi(i)=(i,\dots,i)$ for all $i<n$. Let:

\begin{itemize}
    \item $R_1 :=\{T\in L_4^d\mid (\phi(3),3)\not\in T\}$
    \item $R_2 :=\{T\in L_4^d\mid (\phi(3),3)\in T,\, (\phi(2),3)\not\in T\}$
    \item $R_3 :=\{T\in L_4^d\mid (\phi(2),3)\in T,\,(\phi(1),3)\not\in T\}$
    \item $R_4 :=\{T\in L_4^d\mid (\phi(1),3)\in T,\,(\phi(0),3)\not\in T,\,(\phi(2),2)\not\in T\}$
    \item $R_5 :=\{T\in L_4^d\mid (\phi(1),3)\in T,\,(\phi(0),3)\not\in T,\,(\phi(2),2)\in T\}$
    \item $R_6 :=\{T\in L_4^d\mid (\phi(0),3)\in T\}$
\end{itemize}

These sets form a partition of all the torclosed sets of $L_4^d$, thus $|L_4^d|=\sum_{i=1}^6 |R_i|$.
It is immediate that $|R_1|=|R_6|=|L_3^d|=d+3+5\times 2^d$. 

A torclosed set of $R_2$ is a torclosed set of $\Tam(os_4^d,\phi_{|\{0,1\}})$ together with an order filter of \mbox{$\big(\mathcal{C}_3\setminus\{(\phi(3),3)\}\big) \setminus I_{\mathcal{C}_3}(\phi(2),3)$}. The elements in this latter poset are those elements of $os_4^{d+1}$ that end with two letters $3$, but are not $(3,3,\dots,3)$. It follows that $|R_2|=(d+4)\, (|J(os_4^{d-1})|-1)$. Thus, by \cref{proptotallysym}, $|R_2|=(d+4) (a_d-1)$. 

Let $Q=(\mathcal{C}_2 \cup \mathcal{C}_3,\leq_{prod})$. A torclosed set of $R_3$ may or may not contain $(\phi(0),0)$, does not contain elements of $\mathcal{C}_1$, and on the other elements it should be an order filter of \mbox{$Q\setminus \big(I_Q(\phi(1),3) \cup F_Q(\phi(2),3)\big)$}. The elements in this latter poset are those of $os_4^{d+1}$ that are in \mbox{$[(0,\dots,0,2,2),(1,3,\dots,3)] \cup \{(2,\dots,2)\}=K_1\cup \{(2,\dots,2)\}$}. Then by \cref{lem:nbrorderfilterincomparable}, \mbox{$|R_3|=2\times (|J(K_1)|+|J(I)|)$} where $I$ is the subposet of $K_1\cup \{(2,\dots,2)\}$ on the incomparable elements of $(2,\dots,2)$. We have $I=[(0,\dots,0,2,3),(1,3,\dots,3)]=K_2$.
Thus $|R_3|=2\,(|J(K_1)|+|J(K_2)|)$.

The torclosed sets of $R_4$ are the order filters of the subposet of $os_4^{d+1}$ on the elements $[(0,\dots,0,1,1),(1,\dots,1)] \cup [(0,\dots,0,1,3),(0,3,\dots,3)]$. Denote by $V$ this subposet. We have $|R_4|=|J(V)|$. It will be helpful to add $(0,\dots,0,3)$ to $V$ (the adding forms a new subposet of $os_4^{d+1}$). To do this we use \cref{lem:nbrorderfilterincomparable}, which gives $|R_4|=|J(V\cup \{(0,\dots,0,3)\})|-|J(I')|$ where $I'$ is the subposet of $V\cup \{(0,\dots,0,3)\}$ on the elements that are incomparable to $(0,\dots,0,3)$. Then $I'=[(0,\dots,0,1,1),(1,\dots,1)]$, thus $|R_4|=|J(V\cup \{(0,\dots,0,3)\})|-(d+1)$. In $V\cup \{(0,\dots,0,3)\}$ we give the subposet formed by the incomparable elements to each of the element of the chain $[(0,\dots,0,1,1),(1,\dots,1)]$; there is only one element incomparable to $(0,\dots,0,1,1)$ which is $(0,\dots,0,3)$, for $(1,\dots,1)$ the incomparable elements form the subposet $[(0,\dots,0,3),(0,3,\dots,3)] \cong os_4^{d-1}$, and for any $(0,\dots,0,1,\dots,1)$ with the number of zeros equal to $i\in [d-2]$ the incomparable elements are the elements starting with $i+1$ zeros and ending with a $3$, which form a subposet isomorphic to $os_4^{d-i-1}$. Remember that $|J(os_4^{j})|=a_{j+1}$ (\cref{proptotallysym}). Then applying \cref{lem:nbrorderfilterincomparable} successively for all the elements of the chain $[(0,\dots,0,1,1),(1,\dots,1)]$ gives at the end 
$$|J(V\cup \{(0,\dots,0,3)\})| = |J([(0,\dots,0,3),(0,3,\dots,3)])|+2+a_d+\sum_{i=1}^{d-2} a_{d-i}=2+2\,a_d+\sum_{i=1}^{d-2} a_{d-i} .$$ 
Thus $|R_4|=2+2\,a_d+\sum_{i=1}^{d-2} a_{d-i} -(d+1)= 1+a_d-d+\sum_{i=2}^d a_i$.

The torclosed sets of $R_5$ are the order filters of the subposet of $os_4^{d+1}\setminus \{(2,\dots,2)\}$ on the elements that are neither bigger than $(1,\dots,1,3)$ nor smaller than $(0,\dots,0,3)$. Denote by $W$ this subposet. We have \mbox{$K_3=W\cup \{(0,\dots,0,1),(0,\dots,0),(0,\dots,0,2),(0,\dots,0,3),(2,\dots,2)\}$}. We now use five times \cref{lem:nbrorderfilterincomparable} to add the five above elements to $W$ to obtain $K_3$. As $(0,\dots,0,1)$ is smaller than all the elements of $W\cup \{(0,\dots,0,1)\}$, we have $|J(W\cup\{(0,\dots,0,1)\})|=|J(W)|+1$. Then as $(0,\dots,0,0)$ is smaller than all the elements of $W\cup\{(0,\dots,0,1),(0,\dots,0,0)\}$, we have $|J(W\cup\{(0,\dots,0,1),(0,\dots,0)\})|=|J(W)|+2$. The subposet of $W\cup\{(0,\dots,0,1),(0,\dots,0),(0,\dots,0,2)\}$ on the incomparable elements of $(0,\dots,0,2)$ is $[(0,\dots,0,1,1),(1,\dots,1)]$, thus $|J(W\cup\{(0,\dots,0,1),(0,\dots,0),(0,\dots,0,2)\})|=|J(W)|+2+(d+1)=|J(W)|+3+d$. The subposet of $W\cup\{(0,\dots,0,1),(0,\dots,0),(0,\dots,0,2),(0,\dots,0,3)\}$ on the incomparable elements of $(0,\dots,0,3)$ is $[(0,\dots,0,1,1),(2,\dots,2)[$. Since 
{\small
$$[(0,\dots,0,1,1),(2,\dots,2)[ \,\cup \,\{(2,\dots,2)\}\cup [(0,\dots,0),(0,\dots,0,2)]=[(0,\dots,0),(2,\dots,2)]\cong os_3^{d+1} ,$$}
by using \cref{lem:nbrorderfilterincomparable} we can obtain $|J([(0,\dots,0,1,1),(2,\dots,2)[)|=|J(os_3^{d+1})|-d-4=2^{d+2}-d-4$. Thus 
$$|J(W\cup\{(0,\dots,0,1),(0,\dots,0),(0,\dots,0,2),(0,\dots,0,3)\})|=|J(W)|+2^{d+2}-1 .$$

The incomparable elements of $(2,\dots,2)$ in $K_3$ are the ones that start with a $0$ and end with a $3$, thus they form a subposet of $K_3$ isomorphic to $os_4^{d-1}$. Then 
$$|J(K_3)|=(|J(W)|+2^{d+2}-1)+|J(os_4^{d-1})|=|J(W)|+2^{d+2}-1+a_d .$$
Thus $|R_5|=|J(W)|=|J(K_3)|-2^{d+2}-a_d+1$.

Putting all together using $|L_4^d|=\sum_{i=1}^6 |R_i|$ finishes the proof of the proposition.
\end{proof}

\begin{remark}
With \cref{prop:stembridge} and \eqref{formulaL4d}, we are able to compute $|L_4^7|=6543848$ and \mbox{$|L_4^8|=130286256$}. This completes the column $n=4$ of Table $2$ of \cite{August_2025}. 
\end{remark}

\subsection{Lattices of $d$-torsion classes of the $(d-1)$-Nakayama algebras of type $\textbf{A}$}
\label{sec:Nakayama}

Let $d$ and $n$ be positive integers. Recall the notations and results from the end of \cref{sec:backgroundhighertorsionclasses}, where we introduced the lattice $L_{\underline{l}}^d$ of the $d$-torsion classes of the $(d-1)$-Nakayama algebra associated to the Kupisch series $\underline{l}$. We defined $os_{\underline{l}}^{d+1}= \{y\in os_n^{d+1} \mid y_1\geq y_{d+1}-l_{y_{d+1}} +1 \}  \subseteq os_n^{d+1}$.
We have that $\mathcal{C}_{os_n^{d}}^{\phi} \cap os_{\underline{l}}^{d+1}$ is an order filter of $\mathcal{C}_{os_n^{d}}^{\phi}$. Thus any element of $L_{\underline{l}}^d$ is an order filter of $\mathcal{C}_{os_n^{d}}^{\phi}$.
Recall also that by \cref{thm:torsionAuslanderistam}, $L_n^d= \Tam((os_n^d,\leq),\phi)$ with $\phi$ defined by $\phi(i)=(i,\dots,i)$ for all $i<n$.

\begin{lemma}
\label{lem:newlab}
Let $R\in L_{\underline{l}}^d$. Let $i<j<n$. Suppose that $(x,i)\in R$ and $(\phi(i+1),j)\in R$. Then $(y,j)\in R$ for any $(y,j)\in \mathcal{C}_j\cap os_{\underline{l}}^{d+1}$ such that $(y,j)\geq (x,j)$.
\end{lemma}

\begin{proof}
Let $(y,j)\in \mathcal{C}_j\cap os_{\underline{l}}^{d+1}$ such that $(y,j)\geq (x,j)$. Let us prove that $(y,j)\in R$. We have $(y,i)\geq (x,i) \in R$, thus $(y,i)\in R$ since $R$ is an order filter of $\mathcal{C}_{os_n^d}^{\phi}$. For $k \in \{0,1,\dots ,i\}$, let 
\begin{align*}
X_k &:=(\max(y_1,i-k),\max(y_2,i-k), \dots , \max(y_d,i-k), i), \\
L_k &:=(\max(y_2,i-k)+1, \max(y_3,i-k)+1, \dots , \max(y_d,i-k)+1, i+1, j), \\
Y_k &:=(\max(y_1,i-k),\max(y_2,i-k), \dots , \max(y_d,i-k), j).
\end{align*}

Note that these elements were already introduced in the proof of \cref{prop:reformulationthmtorsion} and were illustrated in \cref{fig:proofreformtheorem} (where $x_1,\dots,x_d$ are replaced by $y_1,\dots,y_d$).
Since $X_k \geq  X_i=(y,i) \in R$ for all $k\leq i$, then $X_k\in R$ for all $k\leq i$ since $R$ is an order filter of $\mathcal{C}_{os_n^d}^{\phi}$. Since $Y_k\geq Y_i=(y,j)\in os_{\underline{l}}^{d+1}$, then $Y_k\in os_{\underline{l}}^{d+1}$ for all $k\leq i$. Let us prove by induction that $Y_k\in R$ for all $k\leq i$. This will finish the proof of the lemma since $Y_i=(y,j)$.
We know that $X_0=(i,\dots,i)\in R$, and $L_0=(i+1,\dots,i+1,j)=(\phi(i+1),j)\in R$ by hyphothesis.
Moreover $X_0\rightsquigarrow \tau_d(L_0)$. Thus by condition $2$ of \cref{thm:torsionNakayama}, since $(i,\dots,i,j)=Y_0 \in os_{\underline{l}}^{d+1}$, we have $Y_0\in R$.
Let $k\in [i]$. Suppose that for all $p<k$, $Y_p\in R$. Let us prove that $Y_k\in R$. We know that $X_k\in R$. Since $L_k\geq Y_{k-1}$ and $Y_{k-1}\in R$ by hypothesis, then $L_k\in R$. Moreover $X_k\rightsquigarrow \tau_d(L_k)$. Thus by condition $2$ of \cref{thm:torsionNakayama}, since $Y_k \in os_{\underline{l}}^{d+1}$ and the $l$-th coordinate of this element is either the $l$-th coordinate of $X_k$ or $L_k$ for all $l\in [d+1]$, we have $Y_k\in R$. This finishes the proof of the induction, thus of the lemma.
\end{proof}

\begin{proposition}
\label{prop:Nakayamarestriction}
The lattice $L_{\underline{l}}^d$ is the lattice, ordered by inclusion, on the sets $T\cap os_{\underline{l}}^{d+1}$ for all torclosed sets $T\in L_n^d$.
\end{proposition}

\begin{proof}
If $T\in L_n^d$, then $T\cap os_{\underline{l}}^{d+1} \in L_{\underline{l}}^d$. Conversely, let $R\in L_{\underline{l}}^d$. Then $R\subseteq \mathcal{C}_{os_n^d}^{\phi}$. Let $T$ be the smallest torclosed set of $L_n^d$ containing $R$. Let us prove that $R=T\cap os_{\underline{l}}^{d+1}$. 
We prove by induction that $T_k\cap os_{\underline{l}}^{d+1} = R_k$ for all $k<n$. It is immediate for $k=0$. Let $k\in [n-1]$. We suppose that $T_i\cap\, os_{\underline{l}}^{d+1} = R_i$ for all $i<k$. We want to prove that $T_k \cap \,os_{\underline{l}}^{d+1} = R_k$.

By definition of $T$, we know that $R_k\subseteq T_k\cap os_{\underline{l}}^{d+1}$. Thus we need to prove that $T_k\cap os_{\underline{l}}^{d+1} \subseteq R_k$. If $R_k=\emptyset$ this is immediate, thus we suppose that $R_k\neq \emptyset$. By \cref{lem:procheenproche}, we can compute $T_k$ by first looking at the completion of $T_{k-1}\cup R_k$, which in $\mathcal{C}_k$ forms the order filter that we denote $F_1$. Then at the completion of $T_{k-2}\cup F_1$, which in $\mathcal{C}_k$ forms the order filter that we denote $F_2$, and so on. At the end we will have $F_k=T_k$. Let us prove by induction that $F_j\cap os_{\underline{l}}^{d+1}\subseteq R_k$ for all $j\in [k]$. This will finish the proof of the proposition since for $j=k$ this will give $T_k\cap os_{\underline{l}}^{d+1}\subseteq R_k$. 

Let us prove that $F_1\cap os_{\underline{l}}^{d+1}\subseteq R_k$. Let $(y,k)\in F_1\cap os_{\underline{l}}^{d+1}$.
If $(y,k)$ is not obtained through the completion of $T_{k-1}\cup R_k$, then $(y,k)\in R_k$ and there is nothing to prove. Thus we assume that there exists $(z,k-1)\in T_{k-1}$ such that $(y,k)\geq (z,k)$. In particular $(\phi(k),k)\in R_k$. Let \mbox{$t=(\min(y_1,k-1),\min(y_2,k-1),\dots,\min(y_d,k-1))\in os_{n}^{d}$}. We have  $(y,k)\geq  (t,k)$. Let us prove that $(t,k-1)\in R_{k-1}$.

Since $(y,k)\in \mathcal{C}_k\cap os_{\underline{l}}^{d+1}=\{(x,k)\in \mathcal{C}_k\mid x\geq \phi(k-l_k+1)\}$, then $y_1\geq k-l_k+1$. Since $k\geq 1$, we have $l_k\geq 2$, which gives $k-1\geq k-l_k+1$. Thus $\min(y_1,k-1)\geq k-l_k+1$. Then
$(t,k)\in os_{\underline{l}}^{d+1}$. Since $(t,k) \leq (k-1,\dots,k-1,k)$, we have $(t,k-1) \in \mathcal{C}_{k-1}$. Moreover, $(t,k-1)\in os_{\underline{l}}^{d+1}$. Indeed, since $l_k-l_{k-1}\leq 1$, then $(k-1)-l_{k-1}+1\leq k-l_k+1$. Thus if $(t,k-1)\not\in os_{\underline{l}}^{d+1}$, which means $t\not\geq k-1-l_{k-1}+1$, then $t\not\geq k-l_{k}+1$, which is absurd since $(t,k)\in os_{\underline{l}}^{d+1}$. Thus $(t,k-1)\in \mathcal{C}_{k-1} \cap os_{\underline{l}}^{d+1}$.
Since $(z,k-1)\in \mathcal{C}_{k-1}$, we have $z_i\leq k-1$ for all $i\in [d]$. Moreover $(z,k)\leq (y,k)$ gives $z_i\leq y_i$ for all $i\in [d]$. Then
$(z,k-1)\leq (t,k-1)$. Since $(z,k-1)\in T_{k-1}$, then $(t,k-1)\in T_{k-1}$. But we proved that $(t,k-1)\in \mathcal{C}_{k-1} \cap os_{\underline{l}}^{d+1}$, thus $(t,k-1)\in T_{k-1}\cap os_{\underline{l}}^{d+1}=R_{k-1}$ by induction hypothesis.

Then we use \cref{lem:newlab} with $(t,k-1)\in R$ and $(\phi(k),k)\in R$. This gives $(y,k)\in R_k$ since $(y,k)\in \mathcal{C}_k \cap os_{\underline{l}}^{d+1}$ and $(y,k)\geq (t,k)$. This is what we wanted to prove. Now if we let $j\in [k]$ and we suppose that $F_i\cap os_{\underline{l}}^{d+1} \subseteq R_k$ for all $i<j$, we need to prove that $F_j\cap os_{\underline{l}}^{d+1} \subseteq R_k$. The proof below uses the same structure as in the proof of $F_1\cap os_{\underline{l}}^{d+1} \subseteq R_k$ above. To help the reader follow these similar arguments, as no confusion is possible since we do not use previously introduced elements, we will reuse some variable names. 

Let $(y,k)\in F_j\cap os_{\underline{l}}^{d+1}$. Let us prove that $(y,k)\in R_k$. If $(y,k)$ is not obtained through the completion of $T_{k-j}\cup F_{j-1}$, then $(y,k)\in F_{j-1}\cap os_{\underline{l}}^{d+1}\subseteq R_k$ by induction hypothesis, and there is nothing to prove. Thus we assume that there exists $(z,k-j)\in T_{k-j}$ such that $(y,k)\geq (z,k)$. In particular $(\phi(k-j+1),k)\in F_{j-1}$. If $(\phi(k-j+1),k)\not\in os_{\underline{l}}^{d+1}$, then $(\phi(k-j+1),k)\leq (\phi(k-l_k+1),k)$. Since $(\phi(k-j+1),k)\in F_{j-1}$ then $(\phi(k-l_k+1),k) \in F_{j-1}\cap os_{\underline{l}}^{d+1} \subseteq R_k$ by induction hypothesis, and this proves that $F_j\cap os_{\underline{l}}^{d+1} \subseteq \mathcal{C}_k\cap os_{\underline{l}}^{d+1} = R_k$ since $\mathcal{C}_k\cap os_{\underline{l}}^{d+1}=\{(x,k)\in \mathcal{C}_k\mid x\geq \phi(k-l_k+1)\}$. Thus we assume that $(\phi(k-j+1),k)\in os_{\underline{l}}^{d+1}$, which gives us $l_k\geq j$. Then $(\phi(k-j+1),k)\in F_{j-1}\cap os_{\underline{l}}^{d+1} \subseteq R_k$ by induction hypothesis, then $(\phi(k-j+1),k) \in R_k$. To have something to prove we thus further assume that $\phi(k-l_k+1)<\phi(k-j+1)$, which gives $l_k>j$. Let $t=(\min(y_1,k-j),\dots,\min(y_d,k-j)) \in os_n^d$. We have $(y,k)\geq (t,k)$. Let us prove that $(t,k-j)\in R_{k-j}$.

Since $(y,k)\in \mathcal{C}_k\cap os_{\underline{l}}^{d+1}=\{(x,k)\in \mathcal{C}_k\mid x\geq \phi(k-l_k+1)\}$, then $y_1\geq k-l_k+1$. We also proved that $l_k>j$, which gives $k-j\geq k-l_k+1$. Thus $\min(y_1,k-j)\geq k-l_k+1$. Then
$(t,k)\in os_{\underline{l}}^{d+1}$. Since $(t,k) \leq (k-j,\dots,k-j,k)$, we have $(t,k-j) \in \mathcal{C}_{k-j}$. 
Moreover, $(t,k-j)\in os_{\underline{l}}^{d+1}$. Indeed, since $l_k-l_{k-j}\leq j$, then $(k-j)-l_{k-j}+1\leq k-l_k+1$. 
Thus if $(t,k-j)\not\in os_{\underline{l}}^{d+1}$, which means $t\not\geq k-j-l_{k-j}+1$, then $t\not\geq k-l_{k}+1$, which is absurd since $(t,k)\in os_{\underline{l}}^{d+1}$. Thus $(t,k-j)\in \mathcal{C}_{k-j} \cap os_{\underline{l}}^{d+1}$.
Since $(z,k-j)\in \mathcal{C}_{k-j}$, we have $z_i\leq k-j$ for all $i\in [d]$. Moreover $(z,k)\leq (y,k)$ gives $z_i\leq y_i$ for all $i\in [d]$. Then
$(z,k-j)\leq (t,k-j)$. Since $(z,k-j)\in T_{k-j}$, then $(t,k-j)\in T_{k-j}$. But we proved that $(t,k-j)\in \mathcal{C}_{k-j} \cap os_{\underline{l}}^{d+1}$, thus $(t,k-j)\in T_{k-j}\cap os_{\underline{l}}^{d+1}=R_{k-j}$ by induction hypothesis.

Then we use \cref{lem:newlab} with $(t,k-j)\in R$ and $(\phi(k-j+1),k)\in R$. This gives $(y,k)\in R_k$ since $(y,k)\in \mathcal{C}_k \cap os_{\underline{l}}^{d+1}$ and $(y,k)\geq (t,k)$. This proves that $F_j\cap os_{\underline{l}}^{d+1} \subseteq R_k$. As explained, this finishes the proof of the proposition.
\end{proof}

\begin{theorem}
The lattice $L_{\underline{l}}^d$ is a lattice quotient of $L_n^d$. Thus it inherits the properties of the $(P,\phi)$-Tamari lattices that are preserved by lattice quotient. In particular it is a join-semidistributive lattice.
\end{theorem}

\begin{proof}
Let $\equiv$ be the equivalence relation on $L_n^d$ defined by $T\equiv T'$ if and only if $T\cap os_{\underline{l}}^{d+1} = T'\cap\, os_{\underline{l}}^{d+1}$.
By \cref{prop:Nakayamarestriction}, the elements of $L_{\underline{l}}^d$ are identified with the equivalence classes of the relation $\equiv$. 
Let $K=(K_0,K_1,\dots,K_{n-1})$ be the Kupisch-like series defined by $K_i:=\max(0,i-l_{i} +1)$ for all $i<n$. Then the associated Kupisch equivalence relation $\equiv_K$ on $L_n^d$ corresponds to $\equiv$. By \cref{thm:torsionAuslanderistam}, $L_n^d=\Tam\big((os_n^d,\leq),\phi\big)$ with $\phi$ defined by $\phi(i)=(i,\dots,i)$ for all $i<n$. Thus the proposition follows from \cref{propKupisch} since $(os_n^d,\leq)$ is a lattice.
\end{proof}

\section{Open questions}
\label{sec:openquestionspphi}

We finish with some open questions.

\begin{question}
Can we find a general formula for the number of elements of $\Tam(P,\phi)$? 
\end{question}

Maybe a general closed formula is too much to ask as counting the number of order ideals is a problem known to be hard, but we could have one for special cases or depending on the number of ideals of (subposets of) $P$. For example, in \cref{sec:Chain} we looked at the special case where $P=C_n$ is a chain. In that case we obtained a recursive formula (\cref{prop:inductionformula}) to count the number of elements of $\Tam(C_n,\phi)$. 
For the case of $\Tam(C_{np},\phi_p)$ we were able to obtain a closed-form expression (\cref{coro:numberchainp}) using a formula satisfied by the generating function of $|\Tam(C_{np},\phi_p)|$ (\cref{thm:genechains}).

\begin{question}
Can we obtain a closed-form expression for the number of elements of $\Tam(C_n,\phi)$?
\end{question}

In the other special case $P=os_n^d$ that we studied in \cref{sec:Auslander}, we were able to give a formula for the number of meet-irreducibles of $L_n^d$ (\cref{cor:resultsLnd}).

\begin{question}
Can we find a simple bijective explanation for the number of meet-irreducibles of $L_n^d$ obtained in \cref{cor:resultsLnd}?    
\end{question}

We have made some progress on the enumeration of the higher torsion classes in that case by giving formulas when $n\leq 4$ is fixed, and $d$ is any positive integer (\cref{prop:number4torsions}). The following asymptotic behavior seems to hold, which says that when the number of components $n$ is fixed, in some ways the lattice $L_n^d$ converges to a distributive lattice when $d$ goes to infinity as the number of elements of $L_n^d$ is comparable to the number of elements of its spine (which is $|J(os_n^{d+1})|$ by \cref{cor:resultsLnd}). For instance, as $d$ goes to infinity, the quotient $\dfrac{|L_n^d|}{|J(os_n^{d+1})|}$
converges to $1$ when $n=2$ and to $\frac{5}{4}$ when $n=3$.

\begin{question}
For $n$ fixed, $|L_n^d|=\mathcal{O}_{d\rightarrow \infty}\big(|J(os_n^{d+1})|\big)$? 
\end{question}

We finish with two more algebraically oriented questions.

\begin{question}
Are the higher Nakayama algebras of type $\mathbf{A}_{\infty}^{\infty}$ also examples of $(P,\phi)$-Tamari lattices?
\end{question}

\begin{question}
    Do the higher torsion classes of a $d$-cluster tilting subcategory always form a join-semidistributive lattice?
\end{question}

\bibliographystyle{alpha}
\bibliography{biblio}

\end{document}